\documentclass[11pt, a4paper, twoside, leqno]{article}
\usepackage{bbm}
\usepackage{pifont}
\usepackage{bbding}
\usepackage{mathrsfs}
\usepackage{pgf,tikz}
\usepackage{amsfonts,amssymb,amsmath,indentfirst,amsthm}
\usepackage{epsfig}
\usepackage{fancyhdr}
\usepackage{graphicx}
\usepackage{cite}
\usepackage{url}
\usepackage{xcolor}
\usepackage{cite}
\usepackage{calc}
\usepackage{cases}
\usepackage[colorlinks,linkcolor=blue]{hyperref}
\usepackage{multirow}  % 用于跨行合并单元格
\usepackage{amsmath}   % 用于数学公式
\usepackage{booktabs}  % 用于绘制更专业的表格线条（三线表风格）
\usepackage{array}     % 增强表格功能
\usepackage{geometry}  % 可选，调整页边距使表格更宽松
\usepackage{threeparttable}
\usetikzlibrary{patterns}
\usepackage{comment}
\usepackage{float}

\def\ess~inf{\mathop{\rm ess~inf}}

\numberwithin{equation}{section}

\newenvironment{key words}{\emph{\texttt{Keywords}}\mbox{  }}{ }
\newtheorem{theorem}{Theorem}[section]
\newtheorem{lemma}[theorem]{Lemma}
\newtheorem{corollary}[theorem]{Corollary}
\newtheorem{proposition}[theorem]{Proposition}

\newtheorem{assumptions}[theorem]{Assumptions}
\newtheorem{example}[theorem]{Example}

\renewenvironment{proof}{\noindent{\textbf{Proof.}}}{\hfill$\Box$}
\theoremstyle{remark}
\newtheorem{remark}[theorem]{\textbf{Remark}}
\newtheorem{remarks}[theorem]{\textbf{Remarks}}

\theoremstyle{plain}

\makeatletter

\newcommand{\Rmnum}[1]{\expandafter\@slowromancap\romannumeral #1@}

\makeatother

\def\RR{{\Bbb R}}
\def\NN{{\Bbb N}}

\def\ka{{\kappa}}

\def\al{{\alpha}}

\def\ga{{\gamma}}

\def\la{{\lambda}}

\def\si{{\sigma}}

\def\ve{{\varepsilon}}
\def \supp {\text{\rm supp\,}}

\allowdisplaybreaks

\let\oldthanks\thanks
\renewcommand{\thanks}[1]{%
	\renewcommand{\thefootnote}{}%  清除当前脚注标记
	\oldthanks{#1}%                 调用原始的 \thanks
	\renewcommand{\thefootnote}{\arabic{footnote}}% 恢复为数字编号（或您需要的样式）
}

\begin{document}
	
\title {\textbf{ $L^p$-Estimates for maximal averages along  mixed homogeneous hypersurfaces in $\mathbb{R}^{3}$ }	
\author{Stefan Buschenhenke, Wenjuan Li, Detlef M\"{u}ller, and Huiju Wang}
	\thanks{This work is supported by the National Key R\&D Program of China (No.2023YFA1010800); Natural Science Foundation of China (No.12301113; No.12271435).}}
	
\date{}
\maketitle
\thispagestyle{empty}

{\bf Abstract:}  In this paper, we  study  $L^p$-estimates for maximal averaging operators $\mathcal M$ along  hypersurfaces  $S$ in  $\mathbb{R}^{3}$ which are the graph of a mixed homogeneous function $\Phi$ which is analytic away from the origin. The closure of such a surface will pass through the origin, so that the usual transversality condition that had been imposed in many previous works on maximal averages along hypersurfaces will not hold even when $\Phi$ is analytic at the origin.

	As our main result,  under mild assumptions which are  satisfied for instance for every  mixed homogeneous polynomial  $\Phi,$  we determine the critical Lebesgue exponent $p_c$ for which  $\mathcal M$ is $L^p$-bounded for every $p>p_c,$ but  unbounded for $p<p_c,$ in terms of multiplicities of the real  roots of the Hessian determinant of $\Phi.$

	 It turns out that the study of the contributions by  neighborhoods of a certain type of roots is  closely related to   recent work by Dendrinos, Ikromov and the first and third author on sharp estimates for a  maximal averaging operator  along a transversal hypersurface of an ``exceptional'' class, whose $L^p$-boundedness had been  an open problem for a long time  and which  has recently  been established by means of their  new theory of FIO-cone multipliers.

\tableofcontents

%%%%%%%%%%%%%%%%%%%%%%%%%%%%%%%%%%%%%%%%%%%%%%%%%%%%%%%%%%%%%%%%%%%%%%%%%%%%%%%%%%%%%%%%%%%%%%%%%%%%

\section{Introduction}
\subsection{Preliminaries}\label{prelim}

If $S$ is a  smooth hypersurface in $\mathbb{R}^n$ with surface measure $d\sigma$ and $\rho\in C_0^{\infty}(S)$ a smooth non-negative function with compact support, then we denote by
\begin{equation}\label{maximalfunsur}
	\mathcal M f(y):=\sup_{t>0}\left|\int_{S}f(y-tx)\rho(x)\, d\sigma(x)\right|, \hspace{0.2cm} y\in \mathbb{R}^n,
\end{equation}
the associated  maximal averaging operator (which depends on $S$ as well as  on the cut-off $\rho$).  An important question in harmonic analysis is to understand on which Lebesgue spaces
$L^p=L^p(\mathbb{R}^n)$ such operators are bounded, i.e., for which $p\ge 1$ an estimate
\begin{equation}\label{Mb}
	\|\mathcal M  f\|_p\le C_p\|f\|_p, \quad  f\in \mathcal{S},
\end{equation}
holds true. The study of this question had been initiated through Stein's celebrated work \cite{St76} on the spherical maximal operator, in which he proved that if $S$ is  the Euclidean unit sphere and $\rho=1$ on $S,$ then in dimension $n\ge 3$ the operator $\mathcal M $ is $L^p$ - bounded if and only if $p> n/(n-1).$ The case $n=2$ was only settled about ten years later by Bourgain \cite{bourgain85}.

Greenleaf \cite{Gr81} extended Stein's result to wider classes of hypersurfaces, assuming  a certain number of principal curvatures to be nonzero, which includes  all hypersurfaces  with non-vanishing Gaussian curvature.

However, it  still remains largely open how to characterize $L^p$-boundedness of maximal operators associated to hypersurfaces $S$ whose Gaussian curvature is allowed to vanish at some points. Let us mention just a few important further results in that direction.

Using the local smoothing estimates from \cite{MSS93}, Iosevich \cite{Io94} obtained sharp results for finite-type curves in $\mathbb{R}^2$. By introducing another key idea, namely to damp the  density $\rho$ along $S$ by a suitable power of the Gaussian curvature on $S,$ Sogge--Stein \cite{SS85} derived $L^p$-bounds for maximal operators associated to hypersurfaces $S$ whose Gaussian curvature does not vanish to infinite order at any point, even though  the lower bounds on $p$ they obtained are not optimal.

Subsequent contributions have been made by various authors. Building on decay estimates from \cite{BNW} for the Fourier transform of measures carried by convex hypersurfaces of finite type, Nagel--Seeger--Wainger \cite{NSW93} studied the boundedness of maximal functions associated to such convex hypersurfaces, and obtained sharp results for graphs of particular classes of mixed homogeneous convex functions, generalizing results of M. Cowling and G. Mauceri in \cite{CM86}.	
For $p>2$, Iosevich--Sawyer \cite{IS97} established  $L^p$-boundedness results  for such hypersurfaces in terms of a certain  notion of contact with affine hyperplanes, by analyzing oscillatory integrals damped by powers of the parametrizing function in local coordinates on $S$. They also show that their estimates are sharp when  the surface does not pass through the origin.  Their paper also contains  interesting results on maximal operators associated to mixed-homogeneous surfaces and non-isotropic dilations.

In \cite{IKM05}, Ikromov--Kempe--M\"{u}ller investigated the $L^p$-boundedness of maximal operators in $\mathbb{R}^3$ associated with hypersurfaces that are given as graphs of functions of the form $c + \phi$, where $\phi$ is a mixed homogeneous function smooth away from the origin and $c$ is a constant. Note that these surfaces do not pass through the origin and satisfy the transversality condition unless $c=0$.

 Under the assumption that the maximal order  of  vanishing  of the Hessian determinant of $\phi$ is finite, the authors proved  boundedness of the corresponding maximal operator on $L^p(\mathbb{R}^3)$ for all $p > h,$ provided  $h\ge 2$. The height $h$ as defined in \cite{IKM05} does agree with Varchenko's  notion of {\it height} $h_\phi$ of $\phi,$   which is defined in terms of Newton diagrams associated to $\phi$ (cf. \cite{IM11} for details). This result improved on some of the afore-mentioned  results in \cite{IS97} on mixed homogeneous functions that only allowed critical points of $\phi$ at the origin. Moreover, the result is again sharp when $c \neq 0$.

In the subsequent work \cite{IKM10} the same authors considered general smooth hypersurfaces of finite type in $\mathbb{R}^3$. In this general context, the previous  condition for a convex surface  to not  pass through the origin has to be replaced by the {\it transversality condition}, which requires  that none of the affine tangent hyperplanes to $S$ passes through the origin. After fixing some point $x^0$ on $S,$ the transversality condition allows to assume that,  after a suitable linear change of coordinates,  $x^0=(0,0,1),$ and that in some neighborhood $\Omega$ of this point the hypersurface $S$ admits a local representation as a graph
\[
\Omega \cap S = \{(x_1,x_2,1+\phi(x_1,x_2)) : (x_1,x_2) \in U\},
\]
where $\phi$ is a smooth function defined on a sufficiently small  open neighborhood $U$ of the origin in $\mathbb{R}^2$, and  where $\phi(0,0)=0$
and $\nabla\phi(0,0)=0.$

If again $h_{\phi} \ge 2$ denotes   the height of $\phi$ in the sense of Varchenko, then for $h_{\phi} \ge 2$ the authors showed that the associated maximal operator $\mathcal M$ is $L^p$ - bounded   for all $p > h_{\phi}$, and that the corresponding $L^p$ - range is sharp  under the transversality condition (with the possible exception of the endpoint $p = h_{\phi}$ when $S$ is non‑analytic). Actually, the $L^p$-estimates in this paper hold true even without the transversality  assumption, which in fact is only required for the proof of sharpness.

The proof of these results does not make use of damping techniques. An alternative proof of a major part of these results  based on such  techniques has been devised by Greenblatt in \cite{Gr13}.

The case where $h_{\phi} <2$ has been studied in the follow-up papers \cite{IM16} and \cite{BDIM19}. By these papers and the preceding one,  say for analytic surfaces $S\subset \Bbb R^3,$   the $L^p$- range  on which $\mathcal M$ is $L^p$ - bounded   has by now been determined completely (with the possible exception of the endpoint) by means of quantities which can be read off from Newton diagrams associated to $\phi,$ with the exception of a small ``exceptional class''  of surfaces with singularities  of type A (in the sense of Arnol'd's classification). For surfaces which are not of type A,  the condition is that $p > \max\{3/2, h_{\phi}\}$. If $S$ has a singularity of type A, i.e., when  exactly one principal curvature  of $S$ vanishes at the point $(0,0,1),$   a new quantity defined in \cite{BDIM19},  the {\it effective multiplicity}, turned out to be relevant. By means of the new theory of FIO-cone multipliers developed in \cite{BDIM25}, two conjectures  on the $L^p$- range  for  $\mathcal  M$ formulated in \cite{BDIM19} could be proved already for a prototypically hypersurface of the exceptional class. The estimates in \cite{BDIM25} will become relevant also for the present paper.

In view of these results,  for hypersurfaces in $\Bbb R^3$ the  $L^p$- range for the  associated maximal operators is fairly well understood under the transversality assumption.

On the other hand, very few sharp results are known for surfaces not satisfying the transversality condition. Actually, to the best of our knowledge, results exist so far only for surfaces passing through the origin.  For such surfaces, one typically
expects the $L^p$-range to become larger than for  their translated cousins satisfying the transversality assumption.

Indeed, as Zimmermann showed in his PhD thesis \cite{Zi14},  matters may change dramatically when the transversality assumption is dropped. By using  a  fine resolution of singularities, involving the geometry of Newton diagrams, he  proved that the condition $p>2$ is sufficient for the $L^p$ - boundedness of  the maximal operator associated to any analytic hypersurface $S\subset \mathbb{R}^3$ passing through the origin and supported in a sufficiently small neighborhood of the origin.

Let us choose $p_c$ minimal so that $\mathcal  M$ is $L^p$ - bounded for every $p>p_c.$ Zimmermann's result raises the following question concerning this critical Lebesgue exponent $p_c$: are there analytic hypersurfaces passing through the origin so that $p_c<2$ for the associated maximal operator?

The answer to this question is ``yes'', as shown by  the second and fourth author of this paper in \cite{LW24}. In \cite{LW24}, the authors study maximal operators $\mathcal M$ associated to hypersurfaces $S$ of the form
$$
S= \{(x_{1},x_{2},\Phi(x_{1},x_{2})): (x_{1},x_{2}) \in U \},
$$
where $\Phi$ is a homogeneous polynomial. It is noteworthy that the optimal lower bound for $p$ in \cite{LW24} is not related to the height of the homogeneous hypersurface.

As an important next step towards an understanding of maximal averages along hypersurfaces not satisfying the transversality condition we shall study here surfaces which are the graph of a mixed-homogeneous function which is real-analytic away from the origin. 

This includes all  mixed-homogeneous   homogeneous polynomial functions. The fact that such polynomials  may correspond to  the contributions by faces of the Newton polygon  of more general, non-homogenous real analytic functions and thus allow to locally approximate  such functions has been one motivation for us to study such maximal operators.

%%%%%%%%%%%%%%%%%%%%%%%%%%%%%%%%%%%%%%%%%%%%%%%%%%%%%%%%%%%%%%%%%%%%%%%%%%%%%%%%%%%%%%%%%%%%%%%%%%%%

\subsection{The main result}\label{main result}

Let us  recall some definitions and notation. Assume  that $\kappa=(\kappa_1,\kappa_2)\in \mathbb{R}^2$ with $\kappa_1,\kappa_2>0$ is  a given ``weight'',  and define a  corresponding group of dilations $\{\delta_r\}_{r>0}$ on $\mathbb{R}^2$ by setting
\[\delta_r(x):=(r^{\kappa_1}x_1,r^{\kappa_2}x_2), \qquad x=(x_1,x_2)\in \mathbb{R}^2.\]
 A function $F$ on $\mathbb{R}^2\setminus\{0\}$ is called {\it $\kappa$-homogeneous of degree $\alpha\in\Bbb R,$ }
 if
\begin{equation*}
F(\delta_r(x))=r^\alpha F(x)\quad\quad\text{for all}\  x\neq 0,r>0.
\end{equation*}
 We shall also use the standard short-hand notation $|\kappa|:=\kappa_1+\kappa_2.$  Assume from here on that $F$ is smooth.
 
 We shall abbreviate partial derivatives  $\partial_i=\frac{\partial}{\partial x_i},\,  i=1,2, $  and shall also use the short-hand notation 
 \[
 F_i:=\partial_i F ,\quad F_{ij}=\partial_i\partial_j F,  \quad \dots .
 \]
 
By   $\mathrm{ord} \, F(x)$ we shall denote the order of vanishing of $F$ in $x,$  i.e., the  smallest  integer $j\in \Bbb N\cup\{\infty\}$ 
 such that $D^k F(x)=0$  for all $k<j,$ where $D^k F(x)$ denotes the $k$-th order total derivative of $F$ in $x$.
 
$\mathrm{H} F(x):=\det (D^2 F(x))$ will denote the Hessian determinant of $F$ at $x.$ 

Note that if $F$ is $\kappa$-homogeneous of degree $\alpha,$ then for $i=1,2,$ $F_i$ is $\kappa$-homogeneous of degree $\alpha-\kappa_i.$   This implies that indeed also every higher order partial derivative $\partial_1^{\beta_1}\partial_2^{\beta_2}F$ is $\kappa$-homogeneous. In particular,  $\mathrm{H} F$ is $\kappa$-homogeneous of degree $2(\alpha-(\kappa_1+\kappa_2))=2(\alpha-|\kappa|).$ 

We shall consider here hypersurfaces 
\begin{equation}\label{SPhi}
S:= \{(x, \Phi(x)): x \in \Bbb R^2\setminus \{0\}\}
\end{equation}
in $\mathbb{R}^3$ which are  given as the graph of a {\it real analytic function} $\Phi:\Bbb R^2\setminus \{0\} \rightarrow \mathbb{R}$ which is  {\it mixed homogeneous} in the sense of  \cite{Schwend}, i.e., $\kappa$-homogeneous of degree $1.$

Note that  $\Phi$ then satisfied the following version of Euler's identity:
\begin{equation}\label{Euler}
	\Phi(x_1,x_2)=\kappa_{1}x_{1}\Phi_{1}(x_{1},x_{2}) + \kappa_{2} x_{2}\Phi_{2}(x_{1},x_{2}).
\end{equation} 

The corresponding maximal operators that we shall study are  of the form
\begin{equation}\label{MS}
	\mathcal M f(y):= \sup_{t>0} |A_{t}f(y)|, \quad y \in \mathbb{R}^{3},
\end{equation}
where the averaging operator $A_{t}$ is defined by
\begin{equation}\label{aver1}
	A_{t}f(y)= \int_{\mathbb{R}^2}f(y-t(x,\Phi(x)))\, \psi(x)\,dx,\quad y \in \mathbb{R}^3.
\end{equation}

Here,  $\psi$ is assumed to be a non-negative smooth function with compact support in $\Bbb R^2$ with $\psi(0)>0$ (note that this maximal operator depends in fact also on $\psi$).

Our goal will be to determine the critical Lebesgue exponent $p_c$ such that $\mathcal M$ is bounded on $L^p(\mathbb{R}^3)$ for all $p>p_c,$ and unbounded for $p<p_c.$ 

\smallskip
Since  (at least for polynomial functions $\Phi$) the case where $\kappa_1=\kappa_2$ has been studied in \cite{LW24}, we  shall consider here the case where $\kappa_1\ne \kappa_2.$ 

We shall also assume that $\Phi$  is non-trivial, i.e.,  $\Phi\ne 0$  (otherwise, since the Hardy-Littlewood maximal operator is bounded on $L^p(\Bbb R^2)$ for every $p>1,$ we  have $p_c=1$). 

\medskip

Let us consider for a moment the case of a polynomial function $\Phi$. By $\ka$-homogeneity it  must then vanish at the origin, i.e., $\Phi(0)=0.$ Observe also that if $K$ is any compact subset of $\Bbb R^2,$ then for $x\in K$ the Gaussian curvature of $S$ at the point $(x,\Phi(x))$ is comparable to 
 $\mathrm{H}\Phi(x),$ up to a multiplicative factor.

And,  if $\mathrm{H} \Phi\equiv 0$ vanishes identically and $\Phi$ is $\kappa$-homogeneous of degree $1,$ with $\kappa_1\ne \kappa_2,$ then it follows from \cite[Proposition 6.6]{Schwend} that  either $\Phi(x_1,x_2)=cx_1^m,$ or $\Phi(x_1,x_2)=cx_2^m,$ with an integer $m\ge 1.$ By means of a suitable linear coordinate change in $\Bbb R^3$ we may then reduce the case $m=1$ to the case where $\Phi=0$, so that we may assume that $m\ge 2.$ But then \cite[Theorem 1.8]{L} shows that  ${\mathcal{M}}$ is bounded on $L^{p}$ for $p>2,$ and   \cite[Lemma 6.3.1]{Zi14} shows  that ${\mathcal{M}}$ is unbounded on $L^{p}$ for $p<2,$  so that here $p_c=2$ (see also \cite{IKM10,Zi14,L,LWZ}).

\smallskip

Note also that if $|\kappa|>1,$ then $\mathrm{H}\Phi$ would be   $\kappa$-homogeneous of negative  degree $2(1-|\ka|)<0.$  For a polynomial function $\Phi$ this  would immediately imply that $\mathrm{H}\Phi$  must vanish identically (since $\mathrm{H}\Phi$ is continuous at the origin), so we are back to the  previous   situation.
\medskip

Coming back to more general functions  $\Phi$ which are only  assumed to be analytic away from the origin, then more general classes of surfaces $S$ as above on which the Gaussian curvature vanishes identically are possible, and we are planning to the study the maximal operators $\mathcal M$  corresponding to those,  and the case where $|\ka|>1,$  in  a subsequent paper, since these studies require different tools than the  ones that we shall develop here for the case where $\mathrm{H} \Phi$ does not vanish identically and $|\ka|\le 1.$

\smallskip

From here on, we  shall therefore assume that $\Phi$ is analytic away from the origin,  and that the following assumptions  hold:

 \begin{assumptions}\label{assumption}
The weight $\kappa$ satisfies $|\kappa|\le 1$ and  $\kappa_1\ne \kappa_2,$ so that in particular 
 $$
 0<\kappa_i<1, \qquad i=1,2,
 $$
 and  $\mathrm{H}\Phi$ does not vanish identically.
\end{assumptions} 
 Let us say that $S$ is {\it non-transversal at the point} $(x,\Phi(x))\in S,$ if the affine tangent plane to $S$ at this point passes through the origin, i.e., if
\begin{equation}\label{nontransversal}
	\Phi(x_1,x_2)-x_1\Phi_1(x_1,x_2)- x_2\Phi_2(x_1,x_2)=0.
\end{equation}
Otherwise we say that $S$ is {\it transversal at this point}.
Since the function $\Phi(x_1,x_2)-x_1\Phi_1(x_1,x_2)- x_2\Phi_2(x_1,x_2)$ is $\kappa$-homogeneous of degree 1, we see that $S$ is transversal to $S$ at $x\in S$ if and only if it is transversal at every point of the orbit $(\delta_r (x), \Phi(\delta_r (x)))=(\delta_r (x), r\Phi(x)),\, r>0,$ of $x$ under the multiplicative group $\Bbb R_{>0}$ acting on $\Bbb R^3$ by the anisotropic dilations 
\begin{equation}\label{Dr}
D_r(x_1,x_2,x_3):=(r^{\kappa_1} x_1, r^{\kappa_2} x_2, r x_3), \qquad r>0,
\end{equation}
on $\Bbb R^3.$

\smallskip
Next, for any function  $F:\mathbb{R}^2\setminus\{0\}\to \Bbb R,$ we shall put $Z_F:=\{x\in\mathbb{R}^2\setminus\{0\} : F(x)=0\}.$ By $S^1:=\{x\in \Bbb R^2: |x|=1\}$ we denote the unit circle in $\Bbb R^2.$ 

If $\Phi$ is  a $\kappa$-homogeneous function  as before, by the analyticity of $\Phi$ our Assumptions \ref{assumption} imply that the zero  variety  $Z_{\mathrm{H}\Phi}\cap S^1$ of $\mathrm{H}\Phi$ in $S^1$ is either empty, or finite. 
For $\zeta\in Z_{\mathrm{H}\Phi}\cap S^1,$ we shall denote  the orbit  
$$\gamma_{(\zeta)}:=\{\delta_r (\zeta), \, r>0\}\subset Z_{\mathrm{H}\Phi}$$
  of $\zeta$ as the corresponding real ``root'' of $\mathrm{H}\Phi.$  
With a slight abuse of notation,  we shall sometimes also  call  $\zeta$ a ``real root'' of $\mathrm{H}\Phi.$ 

\smallskip

The zero variety of $\mathrm{H}\Phi$ is then the  union
\[Z_{\mathrm{H}\Phi}=\bigcup_{\zeta\in Z_{\mathrm{H}\Phi}\cap S^1} \gamma_{(\zeta)}\]
of these real roots.

\smallskip
We say that the  root $\zeta\in Z_{\mathrm{H}\Phi}\cap S^1,$ respectively $\gamma_{(\zeta)},$ is a {\it transversal root} if $S$ is transversal at every point of $\gamma_{(\zeta)}.$ Otherwise we shall speak of a {\it non-transversal root}.

\smallskip

Observe also that by $\ka$-homogeneity, $\Phi_1$ vanishes along $\gamma_{(\zeta)}$ if and only it vanishes at one point of $\gamma_{(\zeta)},$ and similarly for $\Phi_2.$

\smallskip
We shall  decompose the set of  these roots into three classes:

\noindent\textbf{Type A.} We say that $\zeta,$ respectively $\gamma_{(\zeta)},$ is of  Type A if both $\Phi_{1}$ and $\Phi_{2}$ vanish along $\gamma_{(\zeta)}$.

Note that by  Euler's identity \eqref{Euler}, here also $\Phi$ must vanish along the root $\gamma_{(\zeta)}$.

\smallskip

\noindent\textbf{Type B.} We say that $\zeta,$ respectively $\gamma_{(\zeta)},$ is of Type B if both $\Phi_{1}$ and $\Phi_{2}$  do not vanish along  $\gamma_{(\zeta)}.$ Here we shall distinguish the  following two subcases:
\smallskip

We say that $\zeta,$ respectively $\gamma_{(\zeta)},$ is of Type  $\mathrm{B_{T}},$ if it is of Type B and  $S$ is transversal along $\gamma_{(\zeta)}$, 
i.e., if 
\begin{equation}\label{Trans}
\Phi(x_{1},x_{2}) \neq x_{1}\Phi_{1}(x_{1},x_{2}) + x_{2}\Phi_{2}(x_{1},x_{2}) \quad\text{for every} \ (x_{1},x_{2}) \in \gamma_{(\zeta)}.
\end{equation}

 We say that $\zeta,$ respectively $\gamma_{(\zeta)},$ is of Type  $\mathrm{B_{NT}},$ if it is of Type B and  $S$ is non-transversal along $\gamma_{(\zeta)}$, i.e., if 
\[\Phi(x_{1},x_{2}) = x_{1}\Phi_{1}(x_{1},x_{2}) + x_{2}\Phi_{2}(x_{1},x_{2}) \quad\text{for every} \ (x_{1},x_{2}) \in  \gamma_{(\zeta)}.\]

\smallskip

\noindent\textbf{Type C.} We say that $\zeta,$ respectively $\gamma_{(\zeta)},$ is of Type C if exactly one of $\Phi_{1}$ and $\Phi_{2}$ does not vanish along $\gamma_{(\zeta)}.$
\medskip

Finally, for $\zeta\in Z_{\mathrm{H}\Phi}\cap S^1,$ we denote by   
$$N_{(\zeta)}:= \mathrm{ord} \, \mathrm{H}\Phi(\zeta)\in \Bbb N_{\ge 1}$$
 the {\it order of vanishing} of $\mathrm{H}\Phi$ in $\zeta$. Note that $\mathrm{ord} \, \mathrm{H}\Phi$ is constant along $\gamma_{(\zeta)}.$

Similarly, we denote by 
$$T_{(\zeta)}:= \mathrm{ord} \, \Phi(\zeta)\in \Bbb N$$
the {\it order of vanishing} of $\Phi$ in $\zeta$, which may be 0. 

 We shall  also assume in this paper that the following assumptions are satisfied:
 
  \begin{assumptions}\label{TA}
For every real root $\zeta$ of Type $\mathrm{A}$ lying on a coordinate axis, the following hold:
if  $\zeta=(\pm 1,0),$ then 
$$
  T_{(\zeta)}\ka_2\ne 1, \qquad \text{and} \quad (1-T_{(\zeta)})\ka_1+T_{(\zeta)}\ka_2\ne 1,
$$
and if or $\zeta=(0,\pm 1),$ then
$$
  T_{(\zeta)}\ka_1\ne 1, \qquad \text{and} \quad (1-T_{(\zeta)})\ka_2+T_{(\zeta)}\ka_1\ne 1.
$$
\end{assumptions}

We note that {\it these conditions are satisfied for every polynomial function $\Phi$ } (see Lemma \ref{HphinA}), but in general they may fail to hold (cf. Examples  \ref{e6} and \ref{e7}).

\smallskip
When  $\zeta$ is  of Type $\mathrm{B_{T}}$ or  C, then we define  the exponent $M_{(\zeta)}$ as follows:
\begin{equation}\label{MlambdaBC}
	M_{(\zeta)}:=N_{(\zeta)}+2.
\end{equation}

If there exists  a $\zeta\in Z_{\mathrm{H}\Phi}\cap S^1$ of Type $\mathrm{B_{T}},$ then we set
\begin{equation}\label{MBT}
	M_{B_T} :=\max  \{M_{(\zeta)}: \zeta \  \text{is of Type}\  \mathrm{B_{T}}\};
\end{equation}
otherwise, we set  $M_{B_T}:=2.$

Similarly, if there exists  a $\zeta\in Z_{\mathrm{H}\Phi}\cap S^1$  of Type C, then we set
\begin{equation}\label{MC}
	M_{C} :=\max  \{M_{(\zeta)}: \zeta \  \text{is of Type}\  C\};
\end{equation}
otherwise, we set   $M_{C}:=2.$

\noindent We are now in the position to state our main result:
\begin{theorem}\label{theorem1}
 Let $\Phi$  be a real-valued  analytic function on $\Bbb R^2\setminus\{0\}$  which is $\kappa$-homogeneous of degree 1, where  the  weight $\kappa$ satisfies 
$\kappa_1, \kappa_2>0,\ |\kappa|\le 1$ and  $\kappa_1\ne \kappa_2.$ Assume also that the Hessian $\mathrm{H}\Phi$ does not vanish identically, and that $\Phi$ satisfies the Assumptions \ref{TA}  (which applies in particular to any polynomial function of this type). Let then 
$$
p_\Phi:=\max\left\{\frac 32, \frac{2(M_{B_{T}}+1)}{M_{B_{T}} +3}, \frac{2M_{C}}{M_{C} +1}\right\}.
$$   
Then the maximal operator $\mathcal{M}$ defined by (\ref{MS}) is $L^{p} $-bounded for $p>p_\Phi,$  and unbounded 
for $p<p_\Phi,$	i.e., $p_c=p_\Phi.$ Moreover, if $p_c=3/2,$ then $\mathcal{M}$ is unbounded also for $3/2.$ 
\end{theorem}

Observe  that $ \frac{2(M_{B_{T}}+1)}{M_{B_{T}} +3}=6/5<3/2$ when $M_{B_T}=2, $ and $\frac{2M_{C}}{M_{C} +1}=4/3<3/2$ when $M_C=2.$ In particular, if there are no roots of Type $B_T$ or $C,$  then $p_c=3/2.$
\smallskip

Note also that  this theorem in combination with our preceding discussions  allows to determine the critical exponent $p_c$ in a simple way for every mixed-homogeneous polynomial function $\Phi.$

\smallskip  The following remark, which will be a consequence of Lemma \ref{tauinR'}, Remark \ref{TypeT} and Lemma \ref{Ttau=0}
sheds further light on the special roles played by  the roots of Type $\mathrm{B_T}$ or Type C in the theorem:

\begin{remark}\label{transvroot}
A root $\zeta\in Z_{\mathrm{H}\Phi}\cap S^1$  is transversal if and only if it is of Type $\mathrm{B_T}$ or Type C.
\end{remark}

\begin{remark}
	Let us  compare our results in Theorem \ref{theorem1} with known  $L^{p}$-boundedness results for maximal averaging operators
	$\mathcal M $ along mixed-homogeneous hypersurface $S$.
	
	(a)   In \cite{IKM05} it had been shown that $\mathcal M $ is $L^{p}$-bounded for  $p> \max\{2, h\}$. The authors also pointed out that this result is not sharp for  surfaces $S$ passing through the origin.
	Theorem \ref{theorem1} improves on the $L^{p}$-boundedness results in \cite{IKM05}  for such surfaces by covering also values of $p$ below $2.$
	
	(b)    Theorem \ref{theorem1} generalizes some of the  $L^{p}$-estimates obtained in \cite{NSW93} and \cite{ISS}
	from  convex mixed homogeneous polynomial surfaces to non-convex ones. For convex ones our results match with those in these papers. For instance, if $\Phi(x_{1},x_{2}) = x_1^{2} + x_2^{a_{2}}$ with $a_{2} > 2$ an even integer, it was shown in \cite{NSW93} and \cite{ISS} that $\mathcal M $ is bounded on $L^{p}$ for $p > \frac{2a_{2}}{a_{2}+1}$, and this range is sharp up to the endpoint. This result aligns with the conclusions of Theorem \ref{theorem1} and part (c) of  the subsequent Theorem \ref{necessary}. Indeed, here  $\mathrm{H}\Phi=2a_2(a_2-1) x_2^{a_2-2},$ so that there is only one real root $\zeta=(1,0)$ on the unit circle, and this is of Type C, with $M_{(\zeta)} = a_{2}$.
\end{remark}

 \medskip
 The order of vanishing $N_{(\zeta)}:= \mathrm{ord} \, \mathrm{H}\Phi(\zeta)$ has yet another interpretation as the {\it multiplicity of the root } $\gamma_{(\zeta)}.$  We recall to this end  Lemma 3.2 from \cite{IKM05}. We say that a subset $M\subset \Bbb R^2$ is $\kappa$-homogeneous,
if $\delta_r(M)\subset M$ for every $r>0.$ For analytic functions, the proof arguments in \cite{IKM05} yield

\begin{lemma}\label{localform}
Let $F\in C^\omega(\Bbb R^2\setminus\{0\})$ be
a $\kappa$-homogeneous analytic function, where $\kappa_1,\kappa_2>0.$ If
$\zeta=(\zeta_1,\zeta_2)\in\Bbb R^2\setminus\{0\}$ with $\zeta_1\neq 0$, say $\zeta_1>0$, and if
 $1\le N:=\mathrm{ord\,}F(\zeta)<\infty,$ then there exists a $\kappa$-homogeneous 
neighborhood $U$ of $\zeta$ on which $F$ is of the form
\begin{equation}\label{taylor}
F(x)=\bigl(x_2-\tau x_1^{\frac{\kappa_2}{\kappa_1}}\bigr)^N G(x),\quad x\in U,
\end{equation}
where $G$ is a $\kappa$-homogeneous analytic function with $G(\zeta)\neq 0$ and $\tau$
is given by $\tau:=\zeta_2(\zeta_1)^{-\frac{\kappa_2}{\kappa_1}}$. 
If $\zeta_2\ne 0,$ then an analogous statement holds, with the roles of the two 
coordinates interchanged.
\end{lemma}

Note that our assumption that $H\Phi$ does not vanish identically  in Theorem \ref{theorem1} implies that  $N_{(\zeta)}= \mathrm{ord} \, \mathrm{H}\Phi(\zeta)<\infty,$ since $\Phi$ is analytic. Thus by  this lemma, when applied to $F=\mathrm{H}\Phi,$  we obtain that for any   $\zeta\in Z_{\mathrm{H}\Phi}$ with $\zeta_1>0$ there is a $\kappa$-homogeneous neighborhood of the real root  $\gamma_{(\zeta)}$ on which we can  write
\begin{equation}\label{HPhiN}
	\mathrm{H}\Phi(x) = \bigl(x_{2}-\tau(\zeta) x_1^{\frac{\kappa_2}{\kappa_1}}\bigr)^{N_{(\zeta)}}\, {V_\zeta}(x),
\end{equation}
 where $V_\zeta$ is a $\kappa$-homogeneous analytic function which does not vanish on this neighborhood, and where  
$\tau(\zeta):=\zeta_2(\zeta_1)^{-\frac{\kappa_2}{\kappa_1}}.$ 

Note that the mapping $\zeta\mapsto \tau(\zeta)$ is a bijection from the right semi-circle $S^1_+$ where $\zeta_1>0$ onto $\RR$ (even a diffeomorphism), and we shall denote by  $\zeta_\tau\in S^1$ the point corresponding to $\tau\in \RR.$

The subset $R$ of $\Bbb R$ given by 
\[R=\{\tau(\zeta): \zeta\in Z_{\mathrm{H}\Phi}\cap S^1 \ \text{with}\   \zeta_1>0\}\]
then allows for an alternative parametrization of the set of all real roots $\gamma_{(\zeta)}$ of $\mathrm{H}\Phi$ lying in the right half-plane 
$$\mathbb{R}_{+}^{2}:=\{(x_1,x_2)\in \Bbb R^2: x_1>0\}.$$

Analogous factorizations  hold  true on the left  half-plane where $x_1<0$ (just consider here $\Phi(-x_1,x_2)$ in place of $\Phi(x_1,x_2)$), and by swapping $x_1$ and $x_2,$ then also on the upper and lower half-planes where $x_2>0$, respectively $x_2<0$.

\smallskip
As we shall see, it will sometimes be necessary to distinguish between real roots lying on one of the coordinate axis, and the remaining ones. Note that they show a different geometric behavior: roots along the coordinate axis are linear, whereas  the other ones have non-vanishing curvature. 

To this end, as just explained, we may and shall restrict our considerations  to the real roots lying in the right half-plane $x_1>0.$ We then put 
$$R':=R\setminus \{0\},$$
so that $R'$ parametrizes all real roots of $\mathrm{H}\Phi$ in the right half-plane $\mathbb{R}_{+}^{2}$ which are not contained in the $x_1$-axis. 

Indeed, the corresponding set of real roots $\zeta\in S^1$ is  then given by the roots $\zeta_\tau,\,\tau\in R,$ 
and $\tau=0$ parametrizes the root lying on the $x_1$-axis. It will then also be convenient to re-parametrize  
$\gamma_{(\zeta_\tau)}$ by putting, for  $\tau\in R,$ 
\[\gamma_{\tau}:= \bigl\{(x_1,x_2)\in \mathbb{R}^2: x_1>0 \ \text{and}\ x_{2}=\tau x_1^{\frac{\kappa_2}{\kappa_1}}\bigl\},\]
so that $\gamma_{(\zeta)}=\gamma_{\tau(\zeta)}.$
Then, according to \eqref{HPhiN}, we can factorize
\[\mathrm{H}\Phi(x) = \bigl(x_{2}-\tau x_1^{\frac{\kappa_2}{\kappa_1}}\bigr)^{N_\tau}\, {V_\tau}(x)\]
on a $\kappa$-homogeneous neighborhood  of $(1,\tau)$, where $V_\tau$ is a $\kappa$-homogeneous analytic  function which does not vanish on this neighborhood, and where the multiplicity $N_\tau$ of the root $\gamma_\tau$ is given by $N_\tau:=N_{(\zeta_\tau)}.$

Analogously, we can factorize 
\begin{equation}\label{PhiTlambda}
	\Phi(x) = (x_{2}-\tau x_1^{\frac{\kappa_2}{\kappa_1}})^{T_{\tau}}\, {W_\tau}(x),
\end{equation}
on such a neighborhood, on which  $W_\tau$ does not vanish, but where the multiplicity $T_\tau:=T_{(\zeta_\tau)}\in \Bbb N$ of the root $\gamma_\tau$ in $\Phi$ may possibly be $0.$

\smallskip

The size of $T_\tau$ and the Type of  the root $\gamma_\tau$ are closely related, and there are also relations between the multiplicity $N_\tau$ of $\gamma_\tau$ as  a real root of ${\mathrm H}\Phi$ and $T_\tau$  of $\gamma_\tau$ as  a real root of $\Phi.$ These will be examined in the Section \ref{basicp}.

%%%%%%%%%%%%%%%%%%%%%%%%%%%%%%%%%%%%%%%%%%%%%%%%%%%%%%%%%%%%%%%%%%%%%%%%%%%%%%%%%%%%%%%%%%%%%%%%%%%%

\subsection{Examples}\label{examples}

 In the discussions of the following examples it will be convenient to use the following notation: If $\ga_{(\zeta)}$ is a real root of $\mathrm{H}\Phi,$ then for any point $(a,b)\in \ga_{(\zeta)}$ we also write $\ga_{(a,b)}:=\ga_{(\zeta)}.$ 
 
 Note that this root is then non-transversal if and only if $\Phi(a,b)-a\Phi_1(a,b)- b\Phi_2(a,b)=0.$ 
Accordingly, we shall also write $N_{(a,b)}:= N_{(\zeta)}, M_{(a,b)}:= M_{(\zeta)},T_{(a,b)}:= T_{(\zeta)}.$
\color{black}

We shall present here some examples of (mostly polynomial)  functions $\Phi$ which are $\kappa$-homogeneous of degree 1, for  which ${\mathrm H}\Phi$  neither vanishes identically nor is non-vanishing.

\begin{example}[{\bf of Type A}]\label{e1}
Let $\Phi(x)=(x_2-\tau x_1^\ell)^m,$ with $\ell, m\ge 2$ and $\tau\in \Bbb R\setminus\{0\}.$ Here $\kappa=(\frac 1{\ell m}, \frac 1m), \ |\kappa|<1.$
\end{example}

One computes that  
$${\mathrm H}\Phi(x)=-\tau m^2(m-1)\ell (\ell-1) (x_2-\tau x_1^\ell)^{2m-3} x_1^{\ell-2}.$$
For notational convenience, we may assume that $\ell$ is even; the odd case admits a similar discussion and leads to the same conclusion. There are two real roots of ${\mathrm H}\Phi$ lying on the curve $x_2=\tau x_1^{\ell}$, namely $\gamma_{(\pm1,\tau)}$, both of multiplicity $N_{(\pm1,\tau)}=2m-3,$ and clearly $T_{(\pm1,\tau)}=m$. Moreover, $\Phi_1(\pm 1,\tau)=\Phi_2(\pm 1,\tau)=0$, so that  $\Phi_1$ and  $\Phi_2$ vanish along $\gamma_{(\pm1,\tau)}.$ 
Thus both roots  $\ga_{(\pm1,\tau)}$ are of Type A. 

Note that for $\ell\ge 3$,  there are two more roots $\gamma_{(0,\pm1)}$, of multiplicity $N_{(0,\pm1)}=\ell-2,$ lying on the line $x_1=0.$ Here $\Phi_1(0,\pm 1)=0$ and $\Phi_2(0,\pm 1)=m(\pm1)^{m-1}\not=0$. Thus both roots  $\gamma_{(0,\pm1)}$ are of Type C. Then we have $M_{(0,\pm1)}=\ell$ for $\ell\geq 3$.

We also note that for $\ell=2$, there is no root lying on the line $x_1=0$, so we set $M_{(0,\pm1)}:=2=\ell$. Then Theorem \ref{theorem1} shows that 
\[p_{c} = \max\Bigl\{\frac{3}{2}, \frac{2M_{(0,\pm1)}}{M_{(0,\pm1)} + 1} \Bigl\} =\max\Bigl\{\frac{3}{2}, \frac{2\ell}{\ell + 1}\Bigl\}. \]

\smallskip

Here is an example  of Type $B_{T},$  with $T_\tau=1:$
 
\begin{example}[{\bf of Type $\mathrm{B_{T}}$}]\label{e2}
Let $\Phi(x) = 3x^{6}_{1}-15x_{2}x^{4}_{1}+20x^{2}_{2}x^{2}_{1}-8x^{3}_{2}.$ Here $\kappa=(\frac 16,\frac 13).$
\end{example}
One computes that
\[\mathrm{H}\Phi(x)= -1920\, x_{2} (x_{2}-x_{1}^2)^{2}. \]

There are two real roots of $\mathrm{H}\Phi$ lying on the curve $ x_{2}=x^{2}_{1},$ namely $\ga_{(1,1)}$ and $\ga_{(-1,1)},$
both of multiplicity $N_{(\pm1,1)}=2,$ and one  easily  checks that
$$\Phi(x)=(x_2-x_1^2)\Big(x_1^4-4x_1^2 (x_2-x_1^2)-8(x_2-x_1^2)^2\Big),$$
so that $T_{(\pm1,1)}=1.$ Moreover, $\Phi_1(\pm 1,1)=\mp 2$ and $\Phi_2(\pm 1,1)=1,$  so that  $\Phi_1$ and  $\Phi_2$ do not vanish along $\gamma_{(\pm1,1)}.$ Thus both roots  $\ga_{(\pm1,1)}$ are of Type B. Moreover,
$\Phi(\pm 1,1)-\big(\pm \Phi_1(\pm 1,1)+ \Phi_2(\pm 1,1)\big)=1\ne 0,$ so that the roots  $\ga_{(\pm 1,1)}$ are of Type $\mathrm{B_{T}}.$

There are two more roots $\ga_{(\pm 1,0)},$ of multiplicity  $N_{(\pm1,0)}=1,$ lying on the line $x_2=0.$ Here $\Phi_1(\pm 1,0)=\pm 18$ and $\Phi_2(\pm 1,0)=-15,$ and $\Phi(\pm 1,0)-\big(\pm \Phi_1(\pm 1,0)+ 0\cdot \Phi_2(\pm 1,0)\big)=-15\ne 0.$

These are thus also of Type $\mathrm{B_{T}}.$ Since $M_{(\pm1,1)}=4$ and $M_{(\pm1,0)}=3,$ we see that Theorem \ref{theorem1} shows that $p_c=3/2.$

\smallskip
Here is a second example  of Type $B_{T},$ but with $T_\tau=0:$

\begin{example}[{\bf of Type $\mathrm{B_{T}}$}]\label{e2b}
Let $\Phi(x)= x^{6}_{1}-\frac{5}{4}x_{2}x^{4}_{1}+ \frac{5}{18}x^{3}_{2}.$ Here $\kappa=(\frac 16,\frac 13).$
\end{example}
One computes that 
	\[\mathrm{H}\Phi(x)= -25x^{2}_{1}(x_{2}-x^{2}_{1})^{2}.\]
There are  two real roots of $\mathrm{H}\Phi$ lying on the curve $x_2=x_1^2$, namely $\gamma_{(\pm1,1)}$, both of multiplicity $N_{(\pm1,1)}=2.$ And, since 
$$\Phi(x)=\frac 1{36} x_1^6-\frac 5{12} x_1^4(x_2-x_1^2)+\frac 56 x_1^2(x_2-x_1^2)^2+\frac 5{18}(x_2-x_1^2)^3,
$$
we see  that $T_{(\pm1,1)}=0$. Moreover, $\Phi_1(\pm1,1)=\pm1$ and $\Phi_2(\pm1,1)=-\frac{5}{12}$, so that $\Phi_1$ and $\Phi_2$ do not vanish along $\gamma_{(\pm1,1)}.$ Thus both roots $\ga_{(\pm1,1)}$ are of Type B. Moreover, $\Phi(\pm 1,1)-\big(\pm \Phi_1(\pm 1,1)+\Phi_2(\pm 1,1)\big)=-\frac{5}{9}\ne 0,$ so that the roots  $\ga_{(\pm 1,1)}$ are of Type $\mathrm{B_{T}}.$

There are two more roots $\gamma_{(0,\pm1)}$, of multiplicity $N_{(0,\pm1)}=2$, lying on the line $x_1=0$. Here $\Phi_1(0,\pm1)=0$ and $\Phi_2(0,\pm1)=\frac{5}{6}$, so that $\gamma_{(0,\pm1)}$ is of Type C.

Since $M_{(\pm1,1)}=4$ and $M_{(0,\pm1)}=4,$ we see that Theorem \ref{theorem1} shows that 
\[p_{c} = \max\Bigl\{\frac{3}{2},\frac{2(M_{(\pm1,1)} + 1)}{M_{(\pm1,1)} +3}, \frac{2M_{(0,\pm1)}}{M_{(0,\pm1)} + 1} \Bigl\} = \frac{8}{5}. \]

\begin{example}[{\bf of Type $\mathrm{B_{NT}}$}]\label{e3}
Let  $\Phi(x)= 3x^{5}_{1}-8x^{3}_{1}x_{2}+6x_{1}x^{2}_{2}.$ Here $\kappa=(\frac 15,\frac 25).$
\end{example}
One computes that 
	\[\mathrm{H}\Phi(x) =-144\,(x_{2}-x^{2}_{1})(x_{2}+x^{2}_{1}). \]
	
There are  four real roots of $\mathrm{H}\Phi$ lying on the curve $x_2=\pm x_1^2$, namely $\gamma_{(\pm1,1)}$ and $\gamma_{(\pm1,-1)}$, all of multiplicity $N_{(\pm1,1)}=N_{(\pm1,-1)}=1$, and since we can re-write $\Phi$ in the form 
$$
\Phi(x)=x_1\bigl[x_1^4+4x_1^2(x_2-x_1^2)+6(x_2-x_1^2)^2\bigr],
$$
as well as in the form
$$
\Phi(x)=x_1\bigl[17x_1^4-20x_1^2(x_2+x_1^2)+6(x_2+x_1^2)^2\bigr],
$$
we see that $T_{(\pm1,1)}=T_{(\pm1,-1)}=0$. Moreover, $\Phi_1(\pm1,1)=-3$ and $\Phi_2(\pm1,1)=\pm4$, $\Phi_1(\pm1,-1)=45$ and $\Phi_2(\pm1,-1)=\mp20$,
so that $\Phi_1$ and $\Phi_2$ do not vanish along $\gamma_{(\pm1,1)}$ and $\gamma_{(\pm1,-1)}.$ Thus all roots $\ga_{(\pm1,1)}$ and  $\gamma_{(\pm1,-1)}$ are of Type B. Moreover, $\Phi(\pm 1,1)-\big(\pm \Phi_1(\pm 1,1)+\Phi_2(\pm 1,1)\big)= 0$, so that the roots  $\ga_{(\pm 1,1)}$ are of Type $\mathrm{B_{NT}}.$	

On the other hand,   $\Phi(\pm 1,-1)-\big(\pm \Phi_1(\pm 1,-1)-\Phi_2(\pm 1,-1)\big)=\mp48\ne 0$, so that the roots  $\ga_{(\pm 1,-1)}$ are of Type $\mathrm{B_{T}}.$	
	
Since $M_{(\pm1,-1)}=3$, we see that Theorem \ref{theorem1} shows that 
\[p_{c} = \max\Big\{\frac{3}{2}, \frac{2(M_{(\pm1,-1)} + 1)}{M_{(\pm1,-1)} + 3}\Big\} = \frac{3}{2}. \]

	\begin{example}[{\bf of Type C}]\label{e4}
Let $\Phi(x)= x_1^{2m}+(x_2+ x_1^2)^{m}$ with $\ m\ge 3$ odd. Here $\kappa=(\frac 1{2m},\frac 1m).$
\end{example}

One computes that 
\[\mathrm{H}\Phi(x) =2m^2(m-1)(x_2+ x_1^2)^{m-2}\big((2m-1)x_1^{2m-2}+(x_2+ x_1^2)^{m-1}\big). \]

There are  two real roots of $\mathrm{H}\Phi$ lying on the curve $x_2=-x_1^2$, namely $\gamma_{(\pm1,-1)}$, both of multiplicity $N_{(\pm1,1)}=m-2$. It is easy to see that $T_{(\pm1,-1)}=0$. Moreover, $\Phi_1(\pm1,-1)=\pm 2m\not=0$ and $\Phi_2(\pm1,-1)=0$, 
so that $\Phi_1$ does not vanish and $\Phi_2$ vanishes along $\gamma_{(\pm1,-1)}$. So that $\ga_{(\pm1,-1)}$ are of Type C.

Since $M_{(\pm1,-1)}=m$, we see that Theorem \ref{theorem1} shows that 
\[p_{c} = \max\Big\{\frac{3}{2}, \frac{2M_{(\pm1,-1)}}{M_{(\pm1,-1)} + 1}\Big\} =  \frac{2m}{m+ 1}. \]

 \begin{example}[{\bf Monomials}]\label{e5}
Let $\Phi(x)=x_1^A x_2^B,$ where $A,B\in \Bbb N_{\ge 1}.$ Here, we could choose 
$\kappa=(\frac 1{2A},\frac 1{2B}),$ so that  $|\kappa|<1.$
\end{example}
 Then one checks easily that any real root of  $\mathrm{H}\Phi$ is of Type A and lies on one of the coordinate axes. Thus Theorem \ref{theorem1} shows that $p_c=3/2.$

 \begin{example}\label{e7}
Let  $\Phi(x)= x_2^2 \frac{x_1^4}{x_1^4+x_2^2}$ with $\kappa=(\frac 14,\frac 12)$ and $T_0=2.$ 
\end{example}
Here, 
$$
\mathrm{H}\Phi(x)=-x_2^4\,\frac{8x_1^{10}(5x_1^4+9x_2^2)}{(x_1^4+x_2^2)^5},
$$
so that $N_0=4>2T_0-2=2.$  Moreover,  here $ T_0\ka_2=1,$ so that the first condition in Assumptions \ref{TA} is violated.
 
 \begin{example}\label{e6}
Let  $\Phi(x)= x_2^3 \frac{x_1^2}{x_1^4+x_2^2}$ with $\kappa=(\frac 14,\frac 12)$ and $T_0=3.$
\end{example}
Here, $$
\mathrm{H}\Phi(x)=-x_2^6\,\frac{4x_1^2(3x_1^8+12 x_1^4x_2^2 + x_2^4)}{(x_1^4+x_2^2)^5},
$$
so that  $N_0=6> 4=2T_0-2$, and  $ (1-T_0)\ka_1+T_0\ka_2=1,$ so that the second condition in Assumptions \ref{TA} is violated.

\medskip

\textbf{Conventions}: Throughout this article, $\mathrm{H}\Phi$ denotes the determinant of the Hessian matrix of $\Phi$. For convenience, we set  $\partial_i=\frac{\partial}{\partial x_i},$ $\Phi_{i}=\partial_i\Phi$, and  $ \Phi_{ij}=\partial_i\partial_j\Phi, i,j=1, 2,$ etc.
  By $\lfloor \cdot\rfloor$ we mean the floor function.
We shall use the notation $A\ll B$, which means that  there is a sufficiently large constant $G$, which is much larger than $1$ and does not depend on the relevant parameters arising in the context in which
the quantities $A$ and $B$ appear, such that $G A\leq B$.  By
$A\lesssim B$ we mean that $A \le CB $ for some constant $C$ independent of the parameters related to  $A$ and $B$. The notation $A \sim B$ means $A \lesssim B$ and $B \lesssim A$ simultaneously. Moreover, we use the abbreviation $\mathbb{R}_{+}^{2}:=\{(x_1,x_2)\in \Bbb R^2: x_1>0\}.$

%%%%%%%%%%%%%%%%%%%%%%%%%%%%%%%%%%%%%%%%%%%%%%%%%%%%%%%%%%%%%%%%%%%%%%%%%%%%%%%%%%%%%%%%%%%%%%%%%%%%

\section{Basic properties of $\kappa$-homogeneous functions  and  their Hessian}\label{basicp}

\subsection{Multiplicities,  types and transversality of roots of $\mathrm{H}\Phi$} \label{sectionle1} 

Let $\Phi$ be as in Theorem \ref{theorem1}. The next lemma establishes relations between the size of $T_\tau$ and the number of partial derivatives 
{\color{blue} $\Phi_i\, (i=1,2)$ }of $\Phi$  which vanish along $\gamma_{\tau} .$  This information will be important  for our analysis.  
\begin{lemma}\label{tauinR'}
	Assume that $\tau\in R.$ Then the following hold:
	\begin{itemize}
		\item[(a)] If $T_{\tau} \ge 2$, then both  $\Phi_{1}$ and $\Phi_{2}$ vanish along $\gamma_{\tau},$ so that the root $\gamma_\tau$ is of Type A.
		\item[(b)] If $T_{\tau} = 1$, then both $\Phi_{1}$ and $\Phi_{2}$ do not  vanish along $\gamma_{\tau},$ so that 
		$\gamma_\tau$ is of Type B.
 However,  this case can only happen for $\tau\in R',$  not  for $\tau=0,$ 
		\item[(c)] If $T_{\tau} = 0$, then at least one of $\Phi_{1}$ and $\Phi_{2}$ does  not  vanish along $\gamma_{\tau},$ so 
		that $\gamma_\tau$ is of Type B or C.
	\end{itemize}
\end{lemma}
\begin{proof} The proof of (a) is obvious, and (c) follows immediately from  Euler's identity \eqref{Euler} in combination with \eqref{PhiTlambda}. 

It remains to prove (b). So assume that $T_\tau=1.$  By (\ref{PhiTlambda}),  if $(x_1,x_2)\in \gamma_{\tau}$, we get
	$$\Phi_1(x_{1},x_{2})=-\tau (\kappa_{2}/\kappa_{1})x_1^{\kappa_{2}/\kappa_{1}-1}\, {W_\tau}(x_{1},x_{2});\hspace{0.6cm}\Phi_2(x_{1},x_{2})={W_\tau}(x_{1},x_{2}).$$
	 Since $W_\tau$ does not vanish along $\gamma_{\tau}$,  we thus see  that both $\Phi_{1}$ and $\Phi_{2}$ do not  vanish along $\gamma_{\tau}$ if  $\tau \in R^{\prime}$. 
	
\smallskip
Assume next  that $\tau=0.$   Let us here write $W:=W_0.$ From  \eqref{PhiTlambda} one then computes that $0=\mathrm{H}\Phi=-W_1^2$ along $\gamma_0,$ i.e., when $x_2=0,$ so that $ W_1|_{x_2=0}=0.$ But, $W$ is $\kappa$-homogeneous of degree $1-\kappa_2,$ so that the following Euler's identity holds: 
\[
(1-\kappa_2)	W(x)=\kappa_{1}x_{1}W_1(x) + \kappa_{2} x_{2} W_2(x).
\]
For $x_2=0$ this would imply $W=0$  along $\gamma_0,$ contradicting our assumption on $W=W_0.$ 
\end{proof} 

\smallskip

Let us next consider how the Type of a root $\gamma_\tau, \tau \in R$ (which we define as the Type of the root $\gamma_{(\zeta_\tau)}$), influences the magnitude of $T_\tau.$ 

\begin{remark}\label{TypeT}
Lemma \ref{tauinR'} shows the following:
 \begin{itemize}
\item[(a)] If $\gamma_\tau$ is of Type A, then $T_\tau\ge 2.$ Thus,    Euler's identity \eqref{Euler} implies that $S$ is non-transversal along $\gamma_{\tau}$.  

\item[(b)] If $\gamma_\tau$ is of Type B, then $T_\tau\in\{0,1\}$. Moreover, if $\tau\in R\setminus R',$ i.e., $\tau=0$, then $T_\tau=0.$ 

\item[(c)]  When  $\gamma_\tau$ is of Type C, then $T_\tau=0.$  
 \end{itemize}
\end{remark}

\begin{lemma}\label{Ttau=0}
 If $\tau$ is in $R$  and $\gamma_\tau$ is of Type $\mathrm{C}$ (so that $T_\tau=0$), or  if $\tau\in R\setminus R',$ i.e., $\tau =0,$  and $\gamma_\tau$ is of Type $\mathrm{B}$, with   $T_\tau=0$,  then the transversality condition (\ref{Trans})  holds true.
\end{lemma}
 
\begin{proof} 
Assume first that $\gamma_\tau$ is of Type C, and that,  without loss of generality, $\Phi_{1}\neq 0,$ while $\Phi_{2} = 0$ along $\gamma_{\tau} $. By Euler's identity \eqref{Euler}, we then see that  for every $(x_{1},x_{2}) \in \gamma_{\tau} $, we have
\begin{equation}\label{BTrans}
\Phi(x) =  \kappa_{1}x_{1}\Phi_{1}(x),
\end{equation}
and thus
\begin{equation}\label{CTrans}
		\Phi(x) -x_{1}\Phi_{1}(x)-x_{2}\Phi_{2}(x) = (1-\frac{1}{\kappa_{1}}) \Phi(x)
\end{equation}
along $\gamma_{\tau} $. But  $\Phi$ does not vanish along $\gamma_{\tau},$  since $T_{\tau} =0,$ and we are  assuming that  $\kappa_{1} \neq 1.$ Thus  (\ref{Trans}) follows from (\ref{CTrans}). 

Similarly, if $\gamma_\tau$ is Type B with  $\tau=0$ and $T_\tau=0,$ then $x_2=0$ along 
 $\gamma_{0} $, and thus \eqref{BTrans} and \eqref{CTrans} remain true, so that we can conclude as before. 
\end{proof}

\medskip

\begin{remarks}\label{remark1.4}
The results in Theorem \ref{theorem1} depend on the multiplicities $N_{(\zeta)}$ of the real roots $\gamma_{(\zeta)}$ of the Hessian determinant $\mathrm{H}\Phi,$  and these multiplicities $N_{(\zeta)}$ are also linked to the multiplicities $T_{(\zeta)}$ of these roots within $\Phi.$ In what follows, we investigate relationships between these multiplicities, restricting ourselves again to  roots contained in the right half-plane, where $\Phi$ is assumed to satisfy the assumptions   in Theorem \ref{theorem1}.
 \begin{itemize}
\item[(a)]    When $T_{\tau} \ge 2,$ i.e, if $\gamma_\tau$ is of Type A, then the following hold true:

-- If  $\tau \in R^{\prime},$ then $N_\tau=2T_{\tau}-3;$

-- If $\tau \in  R\setminus R^{\prime},$ i.e., if $\tau=0,$ and if $\Phi$ satisfies  Assumptions \ref{TA},  then  $N_0=2T_{0}-2.$ This applies in particular to all $\ka$-homogeneous polynomial functions $\Phi.$ 

\smallskip 

If $\Phi$ is not a polynomial function, but only analytic away from the origin, the last identity for $N_0$ can be false, as Example \ref{e6} shows.  

For the proofs of these assertions we refer to Lemma \ref{HphinA}.

\item[(b)] When $T_{\tau} =1$ and $\tau \in  R^{\prime}$, then we may have $N_{\tau} \ge 2$, as Example \ref{e2} shows.	
	
\item[(c)] When $T_{\tau} = 0$, where $\gamma_{\tau}$ is a root of $\mathrm{H}\Phi$, then $\Phi$ does not vanish along $\gamma_{\tau}.$  Examples of this type are the root $\gamma_0$ in Example \ref{e2} and the root $\gamma_\tau$ in Example \ref{e4}.

\end{itemize}
 \end{remarks}
\smallskip

\begin{remarks}
Assume that  $\gamma_\tau$ is a root of Type B of $\mathrm{H}\Phi$, so that $T_{\tau} =0,$ or $T_{\tau} =1$.
Whether   $\gamma_\tau$ will then be  of Type $\mathrm{B_{T}}$ or not  will  depend on the value  of $T_{\tau}$, and on whether $\tau \in R^{\prime}$, or $\tau \in R\backslash R^{\prime}$, as the following observations show:
 \begin{itemize}
\item[(a)]    If $T_{\tau} =1$ (not even assuming that $\tau$ is of Type B), then $\tau\in R'$ and $\gamma_\tau$ is of Type $\mathrm{B_{T}}$.

\smallskip
	
To see this, recall first from Lemma \ref{tauinR'}(b) that here $\tau\in R'$, and  $x_1\Phi_1\not=0$ and $x_2\Phi_2\not=0$ along $\gamma_{\tau}.$ In particular, $\tau$ is of Type B. Moreover, since $T_{\tau} =1$,  $\Phi$ vanishes along $\gamma_{\tau}$ in this case.

	We will prove our claim by contradiction.  Assuming that  $\gamma_{\tau}$ is not of Type $\mathrm{B_{T}}$, then necessarily also 
	\begin{equation}\label{CaseiEq1}
		x_{1}\Phi_{1}(x_{1},x_{2}) + x_{2}\Phi_{2}(x_{1},x_{2}) =0\qquad \text{along}\quad  \gamma_{\tau}.
	\end{equation}
	Moreover, by  Euler's identity \eqref{Euler}, we have
	\begin{equation}\label{CaseiEq2}
		\kappa_{1}x_{1}\Phi_{1}(x_{1},x_{2}) + \kappa_{2} x_{2}\Phi_{2}(x_{1},x_{2}) =0 \qquad \text{along}\quad  
		\gamma_{\tau}.
	\end{equation}
		(\ref{CaseiEq1}) and (\ref{CaseiEq2}) yield that   $\kappa_{1} = \kappa_{2}$. This contradicts  the fact that $\kappa_{1} \neq \kappa_{2}$.

\item[(b)]   If   $T_{\tau} =0,$ and $\tau \in R\backslash R^{\prime}$, then $\gamma_{\tau}$ is of Type $\mathrm{B_{T}}$.

To see this, let  $\tau=0,$  and assume  to the contrary that  $\gamma_0$ is not of Type $\mathrm{B_{T}}$, i.e.,  $\gamma_0$ is of Type 
	$\mathrm{B_{NT}}.$ Then by definition we must have 
	\begin{equation}\label{CaseiiEq1}
		\Phi(x_{1},0) =x_{1}\Phi_{1}(x_{1},0)+ x_{2}\Phi_{2}(x_{1},0)=x_{1}\Phi_{1}(x_{1},0).
	\end{equation}
	
	The homogeneity of $\Phi$ yields
	\begin{equation}\label{CaseiiEq2}
		\Phi(x_{1},0) = \kappa_{1} x_{1}\Phi_{1}(x_{1},0).
	\end{equation}
	(\ref{CaseiiEq1}) and (\ref{CaseiiEq2}) yield that   $\kappa_{1} = 1$. This contradicts  our assumption that $0< \kappa_{1}<1.$

\item[(c)]  If $T_{\tau} =0$, and $\tau \in R^{\prime}$, then $\gamma_{\tau}$ may be of Type $\mathrm{B_{T}}$, or Type $\mathrm{B_{NT}}$.
	
	Examples of this type are Example \ref{e2b}, and Example \ref{e3}.
	
\end{itemize}
\end{remarks}

%%%%%%%%%%%%%%%%%%%%%%%%%%%%%%%%%%%%%%%%%%%%%%%%%%%%%%%%%%%%%%%%%%%%%%%%%%%%%%%%%%%%%%%%%%%%%%%%%%%%

\subsection{On the Hessian of $\Phi$ near roots of Type A}\label{OnHA}
 Let  $\Phi:\Bbb R^2\setminus \{0\}\to \Bbb R$ be analytic and $\ka$-homogenous of degree 1 such that  $\mathrm{H}\Phi$ does not vanish identically as before, and assume that $\tau\in R$ is such that $\ga_\tau$ is of Type A, i.e, that $T_\tau\ge 2$ (compare Lemma \ref{tauinR'} and Remark \ref{TypeT}).
 
\begin{lemma}\label{HphinA}

(i) If $\tau\in R',$ then $N_\tau=2T_\tau-3,$ so that for $x_1>0$
\[\mathrm{H}\Phi(x) = \bigl(x_{2}-\tau x_1^{\frac{\kappa_2}{\kappa_1}}\bigr)^{2T_\tau-3}\, {V_\tau}(x),\]
where $V_\tau$ does not vanish along $\gamma_\tau.$ 
\smallskip

(ii) If $\tau=0,$ then for $x_1>0$ 
\begin{equation}\label{Phibyka}
\Phi(x)= x_2^{T_0} \,x_1^{\frac{1-T_0\ka_2}{\ka_1}}\, h(x_2 x_1^{-\frac {\ka_2}{\ka_1}})=x_2^{T_0} W_0(x),
\end{equation}
where $h$ is an analytic function on $\Bbb R$ such that $h(0)\ne 0.$ 

In particular, under the Assumptions \ref{TA}, we obtain that 
$|\mathrm{H}\Phi(x)|\sim |x_2|^{2T_0-2}$ when $x_1\sim 1$ and $|x_2|\ll 1$ is sufficiently small, so that $N_0=2T_0-2.$

Similarly, if we put for $0<\sigma\ll1,$ $x_1\sim 1 $ and $ |x_2|\sim 1$  
$$
\Phi_\sigma(x):= x_2^{T_0} \,x_1^{\frac{1-T_0\ka_2}{\ka_1}}\, h(\sigma x_2 x_1^{-\frac {\ka_2}{\ka_1}}),
$$
then $|\mathrm{H}\Phi_\sigma(x)|\sim 1,$ if $\delta$ is sufficiently small. 

\smallskip
Moreover,  the Assumptions \ref{TA} are satisfied when $\Phi$  is a polynomial function.

\end{lemma}

\begin{proof} Let us write $T=T_\tau$ and $W=W_\tau,$ and recall that $T\ge 2.$

 (i) Assume first that $\tau \in   R^{\prime},$ i.e., $\tau\ne 0,$  and write according to \eqref{PhiTlambda}
\[
\Phi(x) = (x_{2}-\tau x_1^{\frac{\kappa_2}{\kappa_1}})^{T}\, W(x)
\]
on a sufficiently narrow $\kappa$-homogeneous neighborhood of the root $\gamma_\tau$ on which the analytic and $\kappa$-homogeneous function $W$ does not vanish. By means of a scaling in $x_2$ we may assume without loss of generality that $\tau=-1.$ Putting then $a:=\kappa_2/\kappa_1$ and $\Delta:=x_2+x_1^a,$ we may then write 
$\Phi=\Delta^T W.$ Then
\begin{eqnarray*}
\Phi_1=T a\Delta^{T-1} x_1^{a-1} W+\Delta^T W_1, \quad, \Phi_2=T\Delta^{T-1} W+\Delta^T W_2.
\end{eqnarray*}
For $x_1\sim 1,$ easy computations then show that 
\begin{eqnarray*}
\Phi_{11}&=&T\Big[ \Delta^{T-2} (T-1) a^2 x_1^{2a-2} W+\Delta^{T-1}a\Big((a-1) x_1^{a-2} W+2 x_1^{a-1} W_1\Big)\Big]+
\mathcal{O} (|\Delta|^T)\\
&=:&T\Big[A+B\Big]+\mathcal{O} (|\Delta|^T),\\
\Phi_{12}&=&T\Big[ \Delta^{T-2} (T-1) a x_1^{a-1} W+\Delta^{T-1}\Big(a x_1^{a-1} W_2+ W_1\Big)\Big]+
\mathcal{O} (|\Delta|^T)\\
&=:&T\Big[C+D\Big]+\mathcal{O} (|\Delta|^T,\\
\Phi_{22}&=&T\Big[\Delta^{T-2} (T-1) W+2\Delta^{T-1} W_2\Big]+\mathcal{O} (|\Delta|^T)\\
&=:&T\Big[E+F\Big]+\mathcal{O} (|\Delta|^T).
\end{eqnarray*}
Then 
$\mathrm{H}\Phi=T^2(I+II+III)+\mathcal{O} (|\Delta|^{2T-2}), $ where we collect terms with common power of $\Delta,$ namely
\begin{eqnarray*}
I&:=&AE-C^2=0\cdot \Delta^{2T-4}=0,\\
II&:=&AF+BE- 2CD=(T-1)a(a-1)x_1^{a-2}W^2 \Delta^{2T-3},\\
III&:=&T^2\Big(BF-D^2\Big)=\mathcal{O} (|\Delta|^{2T-2}).
\end{eqnarray*}
Since $\kappa_1\ne \kappa_2,$ we see that $a\ne 1,$ and hence $|\mathrm{H}\Phi|\sim |\Delta|^{2T-3}$ for $|\Delta|$ sufficiently small, which implies that $N_\tau=2T-3=2T_\tau-3.$

\medskip

(ii) Assume next that  $\tau \in R\backslash R^{\prime}$, i.e., $\tau=0.$ Then $\Phi(x) = x_2^T W(x),$ so that $\Phi(1,y)=y^T W(1,y),$ where 
$h(y):=W(1,y)$ is analytic in $y\in \Bbb R.$ The identity \eqref{Phibyka} then follows immediately by exploiting that  $\Phi$ is $\ka$-homogenous of degree 1.
\smallskip

Put next $A:=\frac{1-T\ka_2}{\ka_1},$  and $\tilde \Phi(x):=x_2^Tx_1^A.$ The Assumptions \ref{TA} imply that $A\ne 0.$ Assume first that also $A\ne 1.$ Then one computes that 
$$
\mathrm{H}\tilde\Phi(x)= TA\left[1-(T+A)\right] x_2^{2T-2} x_1^{2A-2},
$$
where the  Assumptions  \ref{TA} ensures that $TA\left[1-(T+A)\right] \ne 0.$  Similarly, when $A=1,$ then $\mathrm{H}\tilde\Phi(x)=-T^2 x_2^{2T-2}.$ 
In both cases, this easily implies that $|\mathrm{H}\Phi(x)|\sim |x_2|^{2T_0-2}$ when $x_1\sim 1$ and $|x_2|\ll 1$ is sufficiently small, so that $N_0=2T_0-2.$ 

The assertion that  $|\mathrm{H}\Phi_\sigma(x)|\sim 1$ follows analogously.

\smallskip 

Consider next the special case of a polynomial function $\Phi.$ Such a $\ka$-homogeneous  polynomial can be written as a finite sum
\begin{equation}\label{PolyPhi}
\Phi(x) =  \sum\limits_{k=1}^n c_{k}x_1^{A_k} x_2^{B_k},
\end{equation}
with $c_k\ne 0$ and distinct $(A_k,B_k)\in \mathbb{N}^2,$  where  $\kappa_1A_k + \kappa_2B_k =1$. We may  and shall assume  that $B_{k}<B_{k+1},$ hence 
$A_{k}>A_{k+1}$  for $k=1,\dots, n-1$. Then we may factorize $\Phi(x)=x_2^{B_1} W(x),$ where 
$$
W(x)=\sum\limits_{k=1}^n c_kx_1^{A_k}x_2^{B_k-B_1}=c_1x_1^{A_1}+\mathcal{O}(|x_2|),
$$
which shows that here $\frac{1-T_0\ka_2}{\ka_1}=A_1\in \Bbb N$ and $T=B_1$ (more specifically,  since $(A_1-A_k)/(B_k-B_1)=\ka_2/\ka_1,$ here the function $h(y)$ is  explicitly given by the polynomial  $h(y)=\sum\limits_{k=1}^n c_ky^{B_k-B_1}$).

But, if $A_1=0,$ then $W(x)=c_1,$ hence $\Phi(x)=c_1c_2^{B_1}$, so that $\mathrm{H}\tilde\Phi$ would vanish identically, which would contradict our assumptions. Thus $A_1\ge 1,$ so that $T_0\kappa_2\ne 1.$ Moreover, then $T_0+A_1\ge 3,$  so that that  the Assumptions \ref{TA} are satisfied.
\end{proof}

%%%%%%%%%%%%%%%%%%%%%%%%%%%%%%%%%%%%%%%%%%%%%%%%%%%%%%%%%%%%%%%%%%%%%%%%%%%%%%%%%%%%%%%%%%%%%%%%%%%%

\subsection{On the Legendre transform in $x_2$ of $\Phi$ near roots  along which $\partial_2\Phi\ne 0$}\label{Legendre}

Assume that $\Phi$ is as in Theorem \ref{theorem1}, where $|\kappa|\le 1,$  and that $x=(x_1,x_2)$ lies in a sufficiently small $\kappa$-homogeneous neighborhood of a real root  
$\gamma_{\tau}= \bigl\{ x_1>0, \, x_{2}=\tau x_1^{\frac{\kappa_2}{\kappa_1}}\bigl\},  \tau\in R. $  

We shall assume here that $\Phi_2$ does not vanish along $\gamma_\tau.$ This is surely true when $\gamma_\tau$ is of 
Type $\mathrm{B},$ and if $\gamma_\tau$ is of Type $\mathrm{C},$ we may assume this as well after possibly swapping the coordinates $x_1$ and $x_2.$ 

In the proofs in Sections \ref{caseB} and \ref{caseC} we shall study the Fourier transforms of  measures carried by the surface $S,$  and shall first apply the method of stationary phase to the integration in the  variable $x_2.$   The resulting phase function will essentially be the Legendre transform of $\Phi(x_1,x_2)$ in the variable $x_2$ (compare Section 4 in \cite{BIM25}). 

This means that we have to determine a critical point of the phase $\xi_2 x_2+\Phi(x_1,x_2)$ with respect to $x_2,$ i.e., to solve the equation 
\begin{equation}\label{critp}
\xi_2+\Phi_2(x_1,x_2)=0
\end{equation}
in $x_2,$ and then evaluate the phase $\xi_2 x_2+\Phi(x_1,x_2)$ at this critical point.

But, since $\Phi_2$ is $\kappa$-homogenous of degree $1-\kappa_2,$   we may reduce ourselves to the case where $\xi_2=\pm 1.$
Moreover, for $x\in \gamma_\tau,$ we have
\begin{equation}\label{P2hom}
\Phi_2(x_1,\tau x_1^{\frac{\kappa_2}{\kappa_1}})=x_1^{\frac{1-\kappa_2}{\kappa_1}}\Phi_2(1,\tau ),
\end{equation}
where $\Phi_2(1,\tau)\ne 0.$  Thus either   $\Phi_2(x_1,\tau x_1^{\frac{\kappa_2}{\kappa_1}})<0,$ or $\Phi_2(x_1,\tau x_1^{\frac{\kappa_2}{\kappa_1}})>0$  for all  $x_1>0$.

Let us assume without loss of generality that $\Phi_2(x_1,\tau x_1^{\frac{\kappa_2}{\kappa_1}})<0$ for every $x_1>0$.
Then \eqref{critp}, with $\xi_2=\pm 1,$ can have  a solution on $\gamma_\tau$  only if $\xi_2=1.$ And,   if $\xi_2=1,$ then since $0<\kappa_2<1$, \eqref{P2hom} implies that there    exists a unique critical point  $x_0=(x_{0,1},x_{0,2})=(x_{0,1},\tau x_{0,1}^{\frac{\kappa_2}{\kappa_1}})$ on  $\gamma_\tau$ such that
\begin{equation}\label{x10}
	1+\Phi_2(x_0)=0.
\end{equation}

We recall next  \cite[Lemma 3.3]{IKM05}: {\it Assume that $F$ is $\kappa$-homogeneous of degree $1$ and smooth away from 
the origin, let $x\in\Bbb R^2\setminus\{0\}$ and $j\in\{1,2\}.$ If $\kappa_j\neq 1$ and 
$F_j(x)\neq 0$, then either $\mathrm{H} F(x)\ne 0,$  or $F_{jj}(x)\neq 0.$}

Since  $\mathrm{H}\Phi = 0$ along $\gamma_{\tau},$   but $\Phi_2$ is assumed to not vanish  along $\gamma_{\tau},$ this lemma shows that
\begin{equation}\label{phi22ne0}
 \Phi_{22}(x_0)\ne 0.
\end{equation}

The following lemma, which has a similar flavor as Lemma 4.2 in  \cite{BIM25}, will give very precise  information on the Legendre transform  of $\Phi$ in  $x_2$ at $\xi_2=1,$ as well as further information on $\Phi.$
\begin{lemma}\label{Legendret}
Let $\gamma_\tau, \tau \in R,$ be a real root of $\mathrm{H}\Phi$  along which $\Phi_2$ does not vanish, and let $x_0$ be the unique point in $ \gamma_\tau$ satisfying \eqref{x10}. Then  there is a unique, analytic function $x_1\mapsto x_2^c(x_1)$  which is defined in a neighborhood 
$I_\epsilon:=(x_{0,1}-\epsilon, x_{0,1}+\epsilon), \epsilon>0,$ of $x_{0,1}$  on which $x_1>0,$  such that $x_2^c(x_{0,1})=x_{0,2}$ and 
\begin{equation}\label{x2c1}
	1+\Phi_2(x_1,x_2^c(x_1))= 0\qquad \text{for all} \quad x_1\in I_\epsilon.
\end{equation}
a) The function %{\color{red} ( I replaced $\Gamma$ by $\Psi,$ since the cones are later also denoted by $\Gamma$)}
\begin{equation}\label{Gamma1}
\Psi(x_{1}) := x^{c}_{2}(x_{1}) + \Phi(x_{1},x^{c}_{2}(x_{1}))\end{equation}
on $I_\epsilon$ then satisfies 
\begin{equation}\label{Legendra}
	\Psi(x_{1}) = \Psi(x_{0,1}) + \Psi^{\prime}(x_{0,1})(x_{1}-x_{0,1}) + (x_{1}-x_{0,1})^{M_{\tau}}G(x_{1}), 
\end{equation}
where  $G$ is an analytic function  with $G(x_{0,1}) \neq 0,$ and  where $M_\tau:=N_\tau+2,$ if $N_\tau$ denotes the multiplicity of the root $\gamma_\tau$.

Moreover, if $\ga_\tau$ is of type $\rm {B}$, then $\Psi^{\prime}(x_{0,1})\ne 0,$ and if $\ga_\tau$ is of type $\rm {B_{NT}},$
 then 
 \begin{equation}\label{PsiNT}
\Psi(x_{0,1})=x_{0,1}\Psi'(x_{0,1}).
\end{equation}

b) We may decompose 
\begin{equation}\label{phiexpansex}
	\Phi(x)=\Phi_{\rm{lin}}(x)+\Phi_{\rm{nl}}(x),
\end{equation}
where
\begin{eqnarray}\nonumber 
\Phi_{\rm{lin}}(x)&:=&\Psi(x_{0,1})- \Psi^{\prime}(x_{0,1})x_{0,1}+\Psi^{\prime}(x_{0,1})x_{1}-x_2=\Phi(x_0)+\nabla \Phi(x_0)(x-x_0),\\
\Phi_{\rm{nl}}(x)&:=&(x_{1}-x_{0,1})^{M_{\tau}}G(x_1)+(x_2-x_2^c(x_1))^2F(x);\label{psi}
\end{eqnarray}
here $F(x)$ is an analytic  function with $F(x_0)\not=0.$

c) The function $x_2^c$ admits a Taylor expansion 
\begin{equation}\label{x2cnew}
	x_2^c(x_1)=x_{0,2}+  \dot{x}_{0,2}  \cdot(x_1-x_{0,1})+(x_1-x_{0,1})^{M_{\tau}-1}\, \Xi(x_1),
\end{equation}
where $\dot{x}_{0,2}:=(x_2^c)'(x_{0,1})$, and $\Xi$ is an analytic  function of $x_1\in I_\epsilon.$ 
\end{lemma}

The arguments in the next section will  crucially   depend on the results in this lemma too.
\smallskip

\begin{proof} 
Since $\Phi_{22}(x_0)\ne 0$,  the existence  of the interval $I_\epsilon$ and the function  $x_1\mapsto  x_2^c(x_1)$ satisfying \eqref{x2c1} and its  uniqueness follow from  the implicit function theorem.

a) Next, since  $x^{c}_{2}(x_{0,1})= \tau  x_{0,1}^{\frac{\kappa_{2}}{\kappa_{1}}}$, we have
\begin{equation}\label{Gammax10}
\Psi(x_{0,1})  =  \tau  x_{0,1}^{\frac{\kappa_{2}}{\kappa_{1}}} + \Phi(x_{0,1}, \tau  x_{0,1}^{\frac{\kappa_{2}}{\kappa_{1}}}).
\end{equation}
For convenience,  we  shall denote derivatives of a function $g$ of a real variable $s$ also by 
 $\dot{g}(s):=g^{\prime}(s)), \ \ddot{g}(s):=g^{\prime \prime}(s),\dots.$ 
 By differentiating  both sides of (\ref{x2c1}), we see that 
 \begin{equation}\label{Bx2prime}
   \dot{x}^{c}_{2}(x_{1})   = - \frac{\Phi_{12}(x_{1},x^{c}_{2}(x_{1}))} {\Phi_{22}(x_{1},x^{c}_{2}(x_{1}))};
    \end{equation}
in particular,
    \begin{equation}\label{D1}
	\dot{x}_2^c(x_{0,1})=-\frac{\Phi_{21}(x_0)}{\Phi_{22}(x_0)}.
\end{equation}
From \eqref{x2c1}, we  also obtain 
\begin{equation}\label{Gamma'}
\Psi^{\prime}(x_{1})
 = \Phi_{1}(x_{1},x^{c}_{2}(x_{1}));
\end{equation}
in particular, if $\ga_\tau$ is of Type B, then 
\begin{equation}\label{Gammaprimeneq01}
\Psi^{\prime}(x_{0,1}) = \Phi_{1}(x_{0,1},  \tau  x_{0,1}^{\frac{\kappa_{2}}{\kappa_{1}}}) = \Phi_{1}(x_0)\neq 0.
\end{equation}

If $\ga_\tau$ is of type $\rm {B_{NT}},$ then
$$
\Phi(x_0)=x_{0,1}\Phi_1(x_0)+x_{0,2}\Phi_2(x_0),
$$
hence by \eqref{Gammax10} and \eqref{Gammaprimeneq01}
$$
\Psi(x_{0,1})=x_{0,2}+x_{0,1}\Phi_1(x_0)+x_{0,2}\Phi_2(x_0)=x_{0,2}+x_{0,1}\Psi'_1(x_0)+x_{0,2}\Phi_2(x_0).
$$
In combination with \eqref{x10} this implies \eqref{PsiNT}.
\color{black}

For the second  derivative  of $\Psi,$ we obtain from \eqref{Gamma'} and \eqref{Bx2prime} that 
\begin{align}\label{Gamma''}
\Psi^{\prime\prime}(x_{1}) &= \Phi_{11}(x_{1},x^{c}_{2}(x_{1}))+ \dot{x}^{c}_{2}(x_{1})  \Phi_{12}(x_{1},x^{c}_{2}(x_{1}))      =\frac{\mathrm{H}\Phi(x_{1},x^{c}_{2}(x_{1}))} {\Phi_{22}(x_{1},x^{c}_{2}(x_{1}))},
\end{align}
where we note that $ \Phi_{22}(x_{1},x^{c}_{2}(x_{1}))\ne 0$ near $x_{0,1},$ since $\Phi_{22}(x_{0,1}, x_{0,2})\ne 0.$

 By Taylor expansion and $x^{c}_{2}(x_{0,1}) =\tau  x_{0,1}^{\frac{\kappa_{2}}{\kappa_{1}}}$, we have
\begin{equation}\label{Taylorxc2}
x^{c}_{2}(x_{1})=\tau  x_{0,1}^{\frac{\kappa_{2}}{\kappa_{1}}}+ (x_{1}-x_{0,1}) \bigl(  \dot{x}^{c}_{2}(x_{0,1}) + \mathcal{O}(x_{1}-x_{0,1})\bigl).
\end{equation}

Recall next from Lemma \ref{localform} that $\mathrm{H}\Phi$ vanishes of order $N_\tau=N_{(\zeta_\tau)}$ at every point of $\gamma_\tau.$ Thus, by Taylor expansion at the point $(x_1,\tau  x_{1}^{\frac{\kappa_{2}}{\kappa_{1}}}),$ we obtain that 
\[\mathrm{H}\Phi(x_{1},x^{c}_{2}(x_{1})) = [x^{c}_{2}(x_{1})-\tau  x_{1}^{\frac{\kappa_{2}}{\kappa_{1}}}]^{N_{\tau}}Q(x_{1}) \]
 holds near $x_{0,1}$ for an analytic function $Q$ satisfying   $Q(x_{1}) \neq 0$ for $x_{1}>0$ near $x_{0,1}.$ Then (\ref{Gamma''}) and (\ref{Taylorxc2}) yield
 \begin{equation}\label{Ga''}
 \Psi^{\prime \prime}(x_{1}) = (x_{1}-x_{0,1})^{N_{{\tau}}} \bigl(  \dot{x}^{c}_{2}(x_{0,1})-\tau \frac{\kappa_{2}}{\kappa_{1}} x_{0,1}^{\frac{\kappa_{2}}{\kappa_{1}}-1}  + \mathcal{O}(x_{1}-x_{0,1})\bigl)^{N_{\tau}}  \cdot \frac{Q(x_{1},x^{c}_{2}(x_{1}))}{\Phi_{22}(x_{1},x^{c}_{2}(x_{1}))}.
\end{equation}
 
 Observe also that since $\Phi_2$ is $\kappa$-homogeneous of degree $1-\kappa_2,$ the following Euler identity holds for $\Phi_2:$
\[
\kappa_1x_1\Phi_{12}(x_1,x_2)+\kappa_2x_2\Phi_{22}(x_1,x_2)=(1-\kappa_2) \Phi_2(x_1,x_2).
\]
Thus, by \eqref{D1},
\[
\dot{x}_2^c(x_{0,1})=-\frac{\Phi_{12}(x_0)}{\Phi_{22}(x_0)}=-\frac{1-\kappa_{2}}{\kappa_{1}x_{0,1}} \cdot \frac{\Phi_{2}(x_0)}{\Phi_{22}(x_{0})}+ \tau \frac{\kappa_{2}}{\kappa_{1}} x^{\frac{\kappa_{2}}{\kappa_{1}}-1}_{0,1}.
\]
Since $x_{0,2}=\tau x^{\frac{\kappa_{2}}{\kappa_{1}}}_{0,1},$ this implies that
\begin{equation}\label{x2c-ga}
\biggl(x_2^c(x_1)-\tau x_1^{\frac{\kappa_2}{\kappa_1}}\biggl)'\Big|_{x_1=x_{0,1}} =\dot{x}^{c}_{2}(x_{0,1}) - \tau \frac{\kappa_{2}}{\kappa_{1}} x^{\frac{\kappa_{2}}{\kappa_{1}}-1}_{0,1}=-\frac{1-\kappa_{2}}{\kappa_{1}x_{0,1}} \cdot \frac{\Phi_{2}(x_0)}{\Phi_{22}(x_0)} \neq 0.
\end{equation}

Since  $Q(x_{0,1})\ne 0,$ \eqref{Ga''} shows that 
\begin{align}\label{Gammaprimeprime01}
\Psi^{(k)}(x_{0,1})=0 \quad \text{for}\  2 \le k \le N_{\tau}+1, \quad  \Psi^{(N_{\tau}+2)}(x_{0,1}) \neq 0.
\end{align}
 By  Taylor  expansion around $x_{0,1}$ and \eqref{Gammaprimeneq01} this implies \eqref{Legendra}, which proves assertion a).

\smallskip

b) Next, a Taylor expansion of $\Phi$  in $x_2$ near $x_2^c(x_1)$ now leads to 
\begin{eqnarray*}
\Phi(x)&=&\Phi(x_1,x_2^c(x_1))+[x_2-x_2^c(x_1)]\Phi_{2}(x_1,x_2^c(x_1))\\
&& \hskip1cm +[x_2-x_2^c(x_1)]^2\Big[\frac12 \Phi_{22}(x_1,x_2^c(x_1))+\mathcal{O}(x_2-x_2^c(x_1))\Big]\\
&=&\big[x_2^c(x_1)+\Phi(x_1,x_2^c(x_1))\big]-x_2\\
&&\hskip1cm  +[x_2-x_2^c(x_1)]^2\Big[\frac12 \Phi_{22}(x_1,x_2^c(x_1))+\mathcal{O}(x_2-x_2^c(x_1))\Big]\\
&=&\Psi(x_{0,1})- \Psi^{\prime}(x_{0,1})x_{0,1}+\Psi^{\prime}(x_{0,1})x_{1}-x_2+(x_{1}-x_{0,1})^{M_{\tau}}G(x_1)+[x_2-x_2^c(x_1)]^2F(x),
\end{eqnarray*}
where $F(x)$ is an analytic  function with $F(x_0)\not=0.$ This proves assertion b).

c) In order to prove the refined Taylor expansion of $x_2^c$ in \eqref{x2cnew}, we finally  consider the Taylor expansion of  $\ddot x_2^c(x_1)$ at the point $x_{0,1}.$
 
Taking the second derivative of both sides of (\ref{x2c1}), we obtain
\begin{align*}\label{D2}
	 \ddot{x}_2^c(x_1)&=-\frac{1}{\Phi^2_{22}(x_1,x_2^c(x_1))}\biggl[\Phi_{211}(x_1,x_2^c(x_1))\Phi_{22}(x_1,x_2^c(x_1))\\
	&
	 -2\Phi_{212}(x_1,x_2^c(x_1))\Phi_{21}(x_1,x_2^c(x_1))+\Phi^2_{21}(x_1,x_2^c(x_1))\frac{\Phi_{222}(x_1,x_2^c(x_1))}{\Phi_{22}(x_1,x_2^c(x_1))}\biggl]\\
	 &=\frac{1}{\Phi^2_{22}(x_1,x_2^c(x_1))}\biggl[\frac{\Phi_{222}(x_1,x_2^c(x_1)) }{\Phi_{22}(x_1,x_2^c(x_1))}\cdot \mathrm{H}\Phi(x_1,x_2^c(x_1)) -\partial_2\mathrm{H}\Phi(x_1,x_2^c(x_1))\biggl].
\end{align*}

Since $\biggl(x_2^c(x_1)-\tau x_1^{\frac{\kappa_2}{\kappa_1}}\biggl)\big\vert_{x_1=x_{0,1}} =x_{0,2}-\tau x_{0,1}^{\frac{\kappa_2}{\kappa_1}}=0$ and  $\biggl(x_2^c(x_1)-\tau x_1^{\frac{\kappa_2}{\kappa_1}}\biggl)'\Big|_{x_1=x_{0,1}} \neq 0$ by \eqref{x2c-ga}, we have 
\[x_2^c(x_1)-\tau x_1^{\frac{\kappa_2}{\kappa_1}}=(x_1-x_{0,1})\biggl[\biggl(x_2^c(x_1)-\tau x_1^{\frac{\kappa_2}{\kappa_1}}\biggl)'|_{x_1=x_{0,1}}+\mathcal{O}(x_1-x_{0,1})\biggl]=(x_1-x_{0,1})W(x_1),
\]
where $W$ is a smooth function with $W(x_{0,1})\not=0$.  Then  \eqref{HPhiN} implies that
\begin{equation*}
	\mathrm{H}\Phi(x_1,x_2^c(x_1))\sim (x_1-x_{0,1})^{N_{\tau}},\hspace{0.5cm} 	 \partial_2\mathrm{H}\Phi(x_1,x_2^c(x_1))\sim (x_1-x_{0,1})^{N_{\tau}-1},
\end{equation*}
which, together with the identity above for $\ddot{x}_2^c(x_1),$ shows that 
\begin{equation*}
	\bigl|\ddot{x}_2^c(x_1)\bigl| \lesssim |x_1-x_{0,1}|^{N_{\tau}-1}.
\end{equation*}
This implies \eqref{x2cnew}.
\end{proof}

%%%%%%%%%%%%%%%%%%%%%%%%%%%%%%%%%%%%%%%%%%%%%%%%%%%%%%%%%%%%%%%%%%%%%%%%%%%%%%%%%%%%%%%%%%%%%%%%%%%%

\subsection{A local normal form for  $\Phi$ near roots of Type $\mathrm{C}$}
Assume that  $\gamma_{\tau}, \tau \in R,$ is a real root of  $\mathrm{H}\Phi$ of Type C, i.e., exactly one of the functions $\Phi_{1}$ and $\Phi_{2}$ vanishes along $\gamma_{\tau}$. We again put $M_\tau:=N_\tau+2,$ where $N_\tau$ is the multiplicity of the root $\gamma_\tau.$ 
 \smallskip
  
 In the sequel, we may and shall then always assume that   $\Phi_1$ does not vanish  along  $\gamma_{\tau}$, whereas $\Phi_2$ vanishes along  $\gamma_{\tau}.$   Recall also from \cite[Lemma 3.3]{IKM05} that then $\Phi_{11}\ne 0$ along $\gamma_\tau.$
 
To justify this,  note first that for $\tau \neq 0$ we my recur to this case by possibly swapping the roles of the two coordinates $x_1$ and $x_2.$ 
And, if $\tau=0,$  so that $x_2=0$ along $\gamma_\tau,$ then  according to Remark \ref{TypeT}(c) we must have $T_\tau=0,$ i.e., $\Phi$ does not vanish along $\gamma_{\tau}.$  Then Euler's identity \eqref{Euler} shows that  $\Phi_{1} \neq 0,$ and  thus $\Phi_{2} = 0$ along $\gamma_\tau.$

\begin{lemma}\label{TypeCnf}
Let $\gamma_\tau, \tau \in R,$ be a real root of $\mathrm{H}\Phi$  of Type $\mathrm{C},$  and assume  without loss of generality that $\Phi_1$ does not vanish  along  $\gamma_{\tau}$, whereas $\Phi_2$ vanishes along  $\gamma_{\tau}.$    Then, on a sufficiently small $\kappa$-homogeneous neighborhood of $\gamma_\tau,$ $\Phi_{11}$ and $\Phi$ do not vanish,  and $\Phi$ can be written as 
\begin{equation}\label{Phiexpansion}
	\Phi(x_{1},x_{2}) = \Phi(x_{1},\tau x^{\frac{\kappa_{2}}{\kappa_{1}}}_{1}) + (x_{2} -\tau x^{\frac{\kappa_{2}}{\kappa_{1}}}_{1})^{M_{\tau}} V(x_{1},x_{2}),
\end{equation}
where $V$ is analytic, $\kappa$-homogeneous  and does not vanish. 
 \end{lemma}

\begin{proof} 
The lemma will immediately  from the following claim by means of   Taylor  expansion in $x_2$: if $x=(x_{1}, \tau  x_{1}^{\frac{\kappa_{2}}{\kappa_{1}}})\in \gamma_\tau,$ then 
\begin{equation}\label{Phider}
\partial^k_{2}\Phi(x)=0 \quad \text{for}\ k=2,\dots, M_{\tau} -1,\quad \text{and} \quad \partial^{M_{\tau}}_{2}\Phi(x)\ne 0.
\end{equation}

As a first step, we shall  prove that $\Phi_{22}(x) = 0$ if $x\in \gamma_\tau.$ Assume to the contrary that $\Phi_{22}(x) \ne 0.$

Since for $\Phi_j$ is $\kappa$-homogeneous of degree $1-\kappa_j,$ the following Euler's identities hold:
\[
(1-\kappa_j)	\Phi_j(x)=\kappa_{1}x_{1}\Phi_{j1}(x) + \kappa_{2} x_{2}\Phi_{j2}(x), \quad j=1,2.
\]
Note also that we are assuming that $1-\kappa_j>0.$ Moreover, since $\mathrm{H}\Phi(x)=0,$ we may write 
$\Phi_{11}(x)=\Phi_{12}(x)^2/\Phi_{22}(x).$ Thus 
\begin{eqnarray}\label{Phi1Phi2}
(1-\kappa_1)	\Phi_{1}(x)
	&=&\kappa_{1}x_{1} \frac{\Phi^{2}_{12}(x)}{\Phi_{22}(x)} +\kappa_2 x_2 \Phi_{12}(x) =
	 \frac{\Phi_{12}(x)}{\Phi_{22}(x)}\Big( \kappa_1 x_1 \Phi_{21}(x) +\kappa_2 x_2 \Phi_{22}(x)\Big)
	 \nonumber\\
	 &=& \frac{\Phi_{12}(x)}{\Phi_{22}(x)}\, (1-\kappa_2)	\Phi_{2}(x).
\end{eqnarray}
Thus the assumption $\Phi_{2}(x)= 0$ would yield that also    $\Phi_{1}(x) = 0,$  which would contradict our assumption on $\Phi_1.$

Thus along $\gamma_\tau,$  $\Phi_{22}= 0$ and $\Phi_{2}= 0$ by our assumption. In combination with Lemma \ref{localform} applied to $\Phi_2,$ this implies that we  can write 
\[\Phi_{2}(x) = (x_{2}-\tau  x_{1}^{\frac{\kappa_{2}}{\kappa_{1}}})^{\tilde{M}} \tilde{Q}(x), \]
where $\tilde{M} \ge 2$, and where $\tilde{Q}$ is $\kappa$-homogeneous and does not vanish along $\gamma_{\tau}$. 
Then
\begin{align}
	\mathrm{H}\Phi(x) &= \Phi_{11}(x)\biggl[\tilde{M}(x_{2}-\tau  x^{\frac{\kappa_{2}}{\kappa_{1}}}_{1})^{\tilde{M}-1}\tilde{Q}(x) +  (x_{2}-\tau  x^{\frac{\kappa_{2}}{\kappa_{1}}}_{1})^{\tilde{M}}\tilde{Q}_{2}(x) \biggl] \nonumber\\
	&-\biggl[-\tilde{M}\tau  \frac{\kappa_{2}}{\kappa_{1}}x^{\frac{\kappa_{2}}{\kappa_{1}}-1}_{1}(x_{2}-\tau  x^{\frac{\kappa_{2}}{\kappa_{1}}}_{1})^{\tilde{M}-1}\tilde{Q}(x) + (x_{2}-\tau  x^{\frac{\kappa_{2}}{\kappa_{1}}}_{1})^{\tilde{M}}\tilde{Q}_{1}(x) \biggl]^{2} \nonumber\\
	&= (x_{2}-\tau  x^{\frac{\kappa_{2}}{\kappa_{1}}}_{1})^{\tilde{M}-1} Q(x), \nonumber
\end{align}
where
\begin{align}\label{HPhiexpansion}
	Q(x)&= \Phi_{11}(x)\biggl[\tilde{M} \tilde{Q}(x) +  (x_{2}-\tau  x^{\frac{\kappa_{2}}{\kappa_{1}}}_{1}) \tilde{Q}_{2}(x) \biggl]\nonumber\\
	& \quad-(x_{2}-\tau  x^{\frac{\kappa_{2}}{\kappa_{1}}}_{1})^{\tilde{M}-1} \biggl[-\tilde{M}\tau  \frac{\kappa_{2}}{\kappa_{1}}x^{\frac{\kappa_{2}}{\kappa_{1}}-1}_{1}\tilde{Q}(x) + (x_{2}-\tau  x^{\frac{\kappa_{2}}{\kappa_{1}}}_{1}) \tilde{Q}_{1}(x) \biggl]^{2}.
\end{align}
Since   $Q$ does not vanish along $\gamma_{\tau}$,  we see that  $N_{\tau}= \tilde{M}-1$. Thus $M_{\tau}= N_{\tau}+2 $ is the smallest positive integer $k$ such that $ \partial^k_{2}\Phi$ does not vanish along $\gamma_{\tau}.$
\end{proof}

%%%%%%%%%%%%%%%%%%%%%%%%%%%%%%%%%%%%%%%%%%%%%%%%%%%%%%%%%%%%%%%%%%%%%%%%%%%%%%%%%%%%%%%%%%%%%%%%%%%%

\section{Necessary conditions}\label{necessary condition}

In order to prove  the necessity of the condition $p\ge p_\Phi$ in Theorem \ref{theorem1}, after a suitable non-isotropic scaling, it will be sufficient to assume that the maximal operator $\mathcal M$ is given by \eqref{MS}, where the averaging operators $A_t$ are of the form
\begin{equation}\label{AvU}
A_{t}f(y)= \int_{U}f(y-t(x_1,x_2,\Phi(x_1,x_2))\,dx,
\end{equation}
where  $U$ is a suitable open ball centered at the origin. The necessity of the condition $p\ge p_\Phi$ in Theorem \ref{theorem1} will be an immediate consequence of the following result:

\begin{theorem}\label{necessary}
	Under the assumptions of Theorem \ref{theorem1},  let  $\gamma_{(\zeta)}$ with $\zeta\in Z_{\mathrm{H}\Phi}\cap S^1$ be any real root of $H\Phi.$ Then the following hold:
 \begin{itemize}
\item[(a)] 
  $\mathcal{M}$ cannot be $L^p$-bounded unless $p>3/2$; 
\item[(b)] if $\gamma_{(\zeta)}$ is of Type  $\mathrm{B_{T}},$ and assume that  $M_{(\zeta)}\ge 5.$ Then   $\mathcal{M}$ cannot be $L^p$-bounded unless $p\ge \frac{2(M_{(\zeta)}+1)}{M_{(\zeta)} +3};$
\item[(c)] if $\gamma_{(\zeta)}$ is  of Type C, then  $\mathcal{M}$ cannot be $L^p$-bounded unless $p\ge\frac{2M_{(\zeta)}}{M_{(\zeta)} +1}.$
\end{itemize}
\end{theorem}
As for (b), note that $\frac{2(M_{(\zeta)}+1)}{M_{(\zeta)} +3}< 3/2$ when $M_{(\zeta)}\in \{3,4\}.$

The proof of Theorem \ref{necessary} will be based on two  auxiliary results proved in the  next subsection. We observe already here that the first one, Proposition \ref{pnn-1}, immediately implies the assertion (a).

In the proof of (b) and (c), we shall  restrict ourselves to real roots contained in the right half-plane $x_1>0,$ where we shall again parametrize the roots by the parameters $\tau\in R\subset \Bbb R$ and work with the corresponding roots $\gamma_\tau,$ as explained in Subsection \ref{sectionle1}. As seen in this subsection, the discussion of all remaining  roots can easily reduced to this situation.

%%%%%%%%%%%%%%%%%%%%%%%%%%%%%%%%%%%%%%%%%%%%%%%%%%%%%%%%%%%%%%%%%%%%%%%%%%%%%%%%%%%%%%%%%%%%%%%%%%%%

\subsection{Upper and lower bounds for classes of maximal averaging operators along hypersurfaces}

\begin{proposition}\label{pnn-1}
a) Let $\Phi$ be a real-valued non-linear polynomial on   $\Bbb R^{n-1},$ and let 
\begin{equation}\label{Mloc}
\mathcal{M}_{\rm loc} f(y):= \sup_{1/2\le t\le 1 } \Big|\int_{U} f(y-t(x,\Phi(x))\,dx \Big|,
\end{equation}
where $U:=B(0,r)\subset \Bbb R^{n-1}, r>0.$ If   $\mathcal{M}_{\rm loc}$ is bounded on $L^p(\Bbb R^n)$, then $p>n/(n-1).$ 

b) Let  $n=3,$ and let $\Phi:\Bbb R^2\setminus \{0\} \to \Bbb R$ be a non-trivial  analytic function which is $\kappa$-homogeneous of degree 1, where $\kappa_1,\kappa_2>0, \kappa_1\ne \kappa_2$ and $|\kappa|\le 1.$  Furthermore, let 
$U:=B(0,r)\subset \Bbb R^2\setminus \{0\}, r>0,$ and let $\mathcal{M}_{\rm loc}$  be defined as in \eqref{Mloc}.  If   $\mathcal{M}_{\rm loc}$ is bounded on $L^p(\Bbb R^3)$, then $p>3/2.$ 

\end{proposition}

\begin{proof} 
a) It essentially suffices to show that there is a  point $\varrho=(x_0,\Phi(x_0))$  at which the hypersurface $S:= \{(x, \Phi(x)): x\in U\}$ is transversal, i.e., at which $\Phi(x_0)-x_0\cdot \nabla \Phi(x_0)\ne 0.$  Assume to the contrary that
$\Phi(x)-x\cdot \nabla \Phi(x)=0$ for every $x\in U.$ By analyticity,  this would then hold for every $x\in \Bbb R^{n-1}.$    For any monomial $c\, x^\alpha$ contained in $\Phi,$ this would imply that $(1-|\alpha|) c \,x^\alpha \equiv 0,$ hence $|\alpha|=1.$  Consequently, $\Phi$ would be linear, contradicting our assumption. 

But, if $S$ is transversal at $\varrho\in S,$ then, as already explained in the Introduction,  we may  assume that  in a suitable linear coordinate system   $\rho=(0,1),$ and that  a sufficiently small neighborhood of $\varrho$ in $S$  admits a local representation as a graph
\[
 \tilde S = \{(x,1+\phi(x)) : x \in V\},
\]
where $\phi$ is an analytic  function defined on a sufficiently small  open neighborhood $V$ of  $0\in \mathbb{R}^{n-1}$, satisfying 
$\phi(0)=0$ and $\nabla\phi(0)=0.$ From here on, we can argue essentially in the same way as in the proof that Stein's spherical maximal operator can be bounded on $L^p$ only when $p>n/(n-1)$ (compare the argument in \cite{Stbook} , Ch. IX, Remarks 1.1.3). 
\smallskip

b) Assume now that $\Phi:\Bbb R^2\setminus \{0\} \to \Bbb R$ is real-analytic and $\kappa$-homogeneous as assumed in b).  We may assume without loss of generality that $0<\kappa_1<\kappa_2<1.$ As our preceding arguments show, it will suffice to show that the non-transversality condition 
\begin{equation}\label{nt}
\Phi(x)=x_1\Phi_1(x)+x_2\Phi_2(x) 
\end{equation}
cannot hold at every point $x$ in  $\Bbb R^2\setminus \{0\}.$ So assume to the contrary that \eqref{nt} hold everywhere in $\Bbb R^2\setminus \{0\}.$ We shall make use of Euler's identity \eqref{Euler}, i.e., 
\begin{equation}\label{Euler2}
	\Phi(x)=\kappa_1x_1\Phi_1(x) + \kappa_2x_2\Phi_2(x)\qquad \text{for every} \ x\in \Bbb R^2\setminus \{0\}.
\end{equation} 

Applying $\partial_2^j, j\in \Bbb N,$ to both sides of this identity, we get
\[
\partial_2^j \Phi(x) =\kappa_1 x_1\, \partial_2^j\Phi_1(x) +j \kappa_2\partial_2^j\Phi(x)+ \kappa_2 x_2 \partial_2^{j+1}\Phi(x).
\]
Evaluating this at the point $x=(1,0),$ we obtain
\begin{equation}\nonumber%\label{eu1}
\kappa_1\, \partial_2^j\Phi_1(1,0) +(j\kappa_2-1)\partial_2^j\Phi(1,0)=0.
\end{equation}
In a similar way, \eqref{nt} implies 
\begin{equation}\nonumber%\label{nt1}
 \partial_2^j\Phi_1(1,0) +(j-1)\partial_2^j\Phi(1,0)=0.
\end{equation}
 If $A_j$ denotes the matrix
$$A_j:= \left(
\begin{array}{cc}
 \kappa_1 &   j\kappa_2-1   \\
 1 &   j-1  \\
\end{array}
\right),
$$
then these two equations can be expressed by the equation
$A_j \left(
\begin{array}{c}
   \partial_2^j\Phi_1(1,0)  \\
    \partial_2^j\Phi(1,0)
\end{array}
\right)=0.$

For $ j=0,$ $\det A_0=1-\kappa_1>0,$ so that we obtain that $\Phi(1,0)=0.$  Similarly,  if $j=1,$ then $\det A_1=1-\kappa_2>0,$ so that we also find that $\partial_2\Phi(1,0)=0.$  

For $j\ge 2,$ we have $\det A_j=(j-1)\kappa_1-j \kappa_2+1.$ Thus $\det A_j=0$ if and only if $j=J,$ where we have put
$$
J=J_\kappa:= \frac{1-\kappa_1}{\kappa_2-\kappa_1}.
$$
Thus, for all $j\in \Bbb N,$ with the possible exception of $j=J_\kappa$ (provided $J_\kappa\in \Bbb N_{\ge 2}$),  we have $ \partial_2^j\Phi(1,0)=0.$

Thus, if $J_\kappa\notin \Bbb N_{\ge 2},$ the Taylor  expansion of the analytic  function $y\mapsto \Phi(1,y)$ at the point $(1,0)$ shows that $\Phi(1,y)=0$ for every $y\in \Bbb R, $ and then the $\kappa$-homogeneity and analyticity of $\Phi$ imply that $\Phi\equiv 0,$ contradicting our assumptions on $\Phi.$ 
\smallskip

There remains the case where $J_\kappa\in \Bbb N_{\ge 2}.$ Then the Taylor expansion of $\Phi(1,y)$ shows that 
$\Phi(1,y)=c y^J,$ and  by $\kappa$-homogeneity of $\Phi$ this implies that for all $x_1>0$ and $x_2\in \Bbb R$
$$
\Phi(x_1,x_2)=cx_1^{\frac 1{\kappa_1}} \big(x_2 x_1^{-\frac {\ka_2}{\kappa_1}}\Big)^J= c x_1^{\frac {1-J\kappa_2}{\kappa_1}} x_2^J
=c x_1^{-(J-1)} x_2^J,
$$
as one easily computes.  This shows that $\Phi$ would become singular as $x_1\to 0,$ which would violate  our analyticity assumption on $\Phi$ (note however that such a $\Phi$ would satisfy \eqref{nt} in the half-plane $x_1>0$).
\end{proof} 

\medskip

Another  important tool will be a variant of  \cite[Proposition 1.17]{BDIM19}): Suppose that 
\[S:= \{(x, \phi(x)): x\in U\}\]
is a hypersurface in $\mathbb{R}^n$ which is given as the graph of a  smooth function  $\phi:U \to \mathbb{R},$ where $U\subset  \mathbb{R}^{n-1}$ is a bounded  open set.
 
Let $t_0\in [ 1/2,1]$  and $\varepsilon\in (0,1/4),$ and define the corresponding local maximal operator by 
\begin{equation}\label{MSn}
	\mathcal M_{\rm loc}f(y):= \sup_{t_0-\varepsilon\le t\le t_0} |A_{t}f(y)|, \quad y \in \mathbb{R}^{n},
\end{equation}
where the averaging operator $A_{t}$ is defined by
\begin{equation}\label{aver1n}
	A_{t}f(y)= \int_{U} f(y-t(x,\phi(x)))\,dx,\quad y \in \mathbb{R}^n.
\end{equation}
If $A$ is a measurable subset of $\Bbb R^{n}, $ then we put 
\[
|A\cap S|:=\int_{U}\chi_A(x, \phi(x))\, dx.
\]
Since $\mathcal M f(y):= \sup_{t>0} |A_t f (y)|\ge \mathcal M_{\rm loc}f(y),$ the following result gives in particular also necessary conditions for the $L^p$-boundedness of  $\mathcal M:$

\begin{proposition}\label{twopointone}
	Let $K\subset\mathbb{R}^n$ be a symmetric convex body of positive volume  $|K|>0$,  let  $\varepsilon>0$  be as before, and let $\Omega=B(z_0,\varepsilon)\subset U$ be the $\varepsilon$-neighborhood of  a given point  $z_0$ in $\Bbb R^{n-1}.$ Moreover, assume that  $Z:\Omega \to  \mathbb{R}$ is a smooth function whose graph $\Sigma:=\{(z,Z(z)):z\in \Omega\}$ satisfies  a quantitative  transversality condition, i.e.,  there exists a positive constant $c_0>0$ such that
	\begin{equation}\label{TVQ}
		\mathrm{dist} (\varrho+T_{\varrho}\Sigma,0)\geq c_0
	\end{equation}
for all $\varrho=(z,Z(z))\in \Sigma.$ Here, $T_{\varrho}\Sigma$ denotes  the linear  tangent space of $\Sigma$ at the point $\varrho$. Then, for $\varepsilon>0$ sufficiently small,
	\begin{equation}\label{Counterlemma1}
		\|\mathcal{M}_{\rm loc}\|_{L^p\rightarrow L^p}^p\gtrsim c_0 \,\varepsilon^{p(n-1)+1} \int_{\Omega}\frac{|\bigl[(z,Z(z))+K\bigl]\cap S|^p}{|K|}dz,
	\end{equation}
	where the local  maximal operator $\mathcal M_{\rm loc}$ is defined by \eqref{MSn}.
\end{proposition}

\begin{proof}
	Given $\varrho=(z,Z(z))$, $\varrho +T_{\varrho}\Sigma$  is the graph of
	\begin{equation*}
		L_{z}:x\mapsto Z(z)+\nabla Z(z)\cdot (x-z), \hspace{0.5cm}x\in \mathbb{R}^{n-1}.
	\end{equation*}
	By (\ref{TVQ}), we have
	\begin{equation}\label{ETVQ}
		c_0\leq\mathrm{dist} (\varrho+T_{\varrho}\Sigma,0)=\inf_{x\in \mathbb{R}^{n-1}} |(x,L_{z}(x))|\leq |L_{z}(0)|=|Z(z)-\nabla Z(z)\cdot z|.
	\end{equation}
	Let $B=\{t(z,Z(z))\in \mathbb{R}^n: z\in \Omega, t\in [t_0-\varepsilon,t_0]\}.$ Then (\ref{ETVQ}) implies that the Jacobi determinant of the smooth map $\Omega \times [t_0-\varepsilon,t_0]\rightarrow B$, $(z,t)\mapsto t(z,Z(z)),$ has  modulus 
	\begin{equation}\label{ETVQ1}
		t^{n-1}|Z(z)-\nabla Z(z)\cdot z|\geq 4^{-(n-1)}c_0.
	\end{equation}
In particular, this  mapping has a smooth inverse if $\varepsilon$ is sufficiently small.

Furthermore, since $K$ is symmetric and convex, and since $|\Omega|\sim \varepsilon^{n-1},$  it is easy to see that  for any $z\in \Omega$ and  $t\in [t_0-\varepsilon,t_0],$ we have
\[\mathcal{M}_{\rm loc}(\chi_K)(tz,tZ(z)) \geq |A_t(\chi_K)(tz,tZ(z))| \gtrsim \epsilon^{n-1} |[(z,Z(z)+K\bigl]\cap S|.\]
	Using the lower bound \eqref{ETVQ1} for the Jacobian, we conclude that
	\begin{align*}
		\|\mathcal{M}_{\rm loc}(\chi_K)\|_{L^p}^p 
		\geq& \int_B |\mathcal{M}_{\rm loc}(\chi_K)(y)|^p \,dy \\
		\geq& 2^{-(n-1)}c_0  \int_\Omega \int_{t_0-\varepsilon}^{t_0} |\mathcal{M}_{\rm loc}(\chi_K)(tz,tZ(z))|^p  \,dtdz \\
		\gtrsim& c_0 \,{\varepsilon^{p(n-1)+1}}  \int_{\Omega} |\bigl[(z,Z(z))+K\bigl]\cap S|^p \,dz,
	\end{align*}
from which the claim is immediate.
\end{proof}

Note that the  assumption $\Omega\subset U$ was  not needed  for the proof, but the right-hand side of \eqref{Counterlemma1} can only become large if $U$ and $\Omega$ have at least a large intersection.

Applying the above proposition,  we  immediately obtain the following corollary:

\begin{corollary}\label{Counterco}
Let $n=3,$ and let  $S, \, \mathcal{M}_{\rm loc}$ and $\Sigma:=\{(z,Z(z):z\in \Omega=B(z_0,\varepsilon)\}$ be as in Proposition \ref{twopointone}, where $\Sigma $ satisfies  a quantitative transversality condition, and where $\varepsilon>0$ is supposed to be sufficiently small. 

Assume that for every sufficiently small  $\delta>0$, there exists  a symmetric convex body $K_{\delta},$ an open subset $\Omega_{\delta}\subset \Omega,$ and for every $z\in \Omega_{\delta}$ an open subset  $U_\delta(z) \subset U \subset \mathbb{R}^2$, such that every  $z\in \Omega_{\delta}$ and $x\in U_{\delta}(z)$, we have 
	\[(x,\phi(x))\in (z,Z(z))+K_{\delta}=:K_{\delta}(z),\]
where 
	 $|U_{\delta}(z)|\gtrsim \delta^u$ for all  $z\in \Omega_{\delta}$,  $|K_{\delta}|\lesssim \delta^s$ and $|\Omega_{\delta}|\gtrsim\delta^{\omega}$, with  exponents $u>0$ and $s\ge \omega\ge 0.$ 
	 
Then, if $\mathcal{M}_{\rm loc}$ is $L^p$-bounded, necessarily 
	$$p\geq \frac{s-\omega}{u}.$$ 
\end{corollary}

%%%%%%%%%%%%%%%%%%%%%%%%%%%%%%%%%%%%%%%%%%%%%%%%%%%%%%%%%%%%%%%%%%%%%%%%%%%%%%%%%%%%%%%%%%%%%%%%%%%%

\subsection{Conditions enforced by roots of Type  $\mathrm{B_{T}}.$}\label{onCaseBT}

Here we want to prove Theorem \ref{necessary}(b). Assume that  $\mathcal M$ is the maximal operator defined by the averaging operators \eqref{AvU}, and let $\gamma_\tau, \tau\in R,$ be a real root of  $\mathrm{H}\Phi$ of multiplicity $N_\tau$  which is of Type  $\mathrm{B_{T}}.$  We also assume that $M_\tau=N_\tau+2\ge 5.$ We shall make use of the results in Lemma \ref{Legendret} and refer in our arguments to the notation from this lemma. 
By applying a suitable scaling of the form \eqref{Dr}, we may assume that  $U=B(0,r)$ is an open ball   centered at the origin which contains the point $x_0$  given by \eqref{x10}.

Since the boundedness of the maximal operator is stable under linear transformations (though not under  affine linear ones), we may change coordinates by
$$ y_1=x_1, \qquad y_2=x_2-\dot{x}_{0,2}\cdot x_1,$$
with $\dot{x}_{0,2}$ as in \eqref{x2cnew},
 and set
$$\Psi(y):=\Phi(y_1,y_2+\dot{x}_{0,2}\cdot y_1)-\nabla\Phi(x_0)\cdot (y_1,y_2+\dot{x}_{0,2}\cdot y_1).$$
We shall thus study maximal averages along the graph $\tilde S$ of $\Psi$ in place of $S$ (note that $\Psi$ will no longer be $\kappa$-homogenous, but this will not be relevant for the present proof).

Accordingly, we put
$$y_{0,1}:=x_{0,1},\qquad y_{0,2}:=x_{0,2}-\dot{x}_{0,2}\cdot x_{0,1},$$
and note that by \eqref{x2cnew}
\begin{equation}\label{ytwocritical}
y_2^c(y_1):=x_2^c(x_1)-\dot{x}_{0,2}\cdot x_1=y_{0,2}+(y_1-y_{0,1})^{M_\tau-1}\, \tilde \Xi(y_1),
\end{equation}
where  $x_2^c$  is given by \eqref{x2c1}.
It is then easy to compute that
$$\nabla\Psi(y_0)=0,$$
and  that by \eqref{phiexpansex} and \eqref{psi}
\begin{align}
	\Psi(y)&=\Psi(y_0)+\Psi_{\rm{nl}}(y), \\
	\Psi_{\rm{nl}}(y)&:=(y_1-y_{0,1})^{M_\tau}\tilde G (y_1)+(y_2-y_2^c(y_1))^2\tilde F(y).\label{phiexpansexy}
\end{align}

By the transversality condition \eqref{Trans}, we have
$$2c_0:= |\Psi(y_0)-\nabla\Psi(y_0)\cdot y_0|=|\Psi(y_0)|=|\Phi(x_0)-\nabla\Phi(x_0)\cdot x_0|>0.$$
For $0<\epsilon\ll 1$ sufficiently small this transversality remains stable on the set
$$\Omega:=\{y\in U: |y_1-y_{0,1}|<\epsilon, |y_2-y_{0,2}|<\epsilon\};$$
just choose $\epsilon$ so small that for $y\in \Omega$
\begin{align}\label{transverse1}
	|\Psi(y)-\nabla\Psi(y)\cdot y| \geq c_0.
\end{align}
That is, if we choose $Z:=\Psi,$ then the graph $\Sigma:=\{(y,\Psi(y)):y\in \Omega\}$ satisfies a quantitative  transversality condition \eqref{TVQ} as required in Proposition \ref{twopointone} and Corollary \ref{Counterco}.

What we need to prove is that if $\mathcal M$ is $L^p$-bounded and $M_\tau\ge 5,$  then $p\geq\frac{2(M_{\tau}+1)}{M_{\tau}+3}.$ 

We  define for $\delta>0$ sufficiently small
 \begin{align*}
\Omega_{\delta}:=& \{z\in \Omega:|z_1-y_{0,1}|<\delta^{\frac{1}{M_{\tau}}},
	 |z_2-y_{0,2}|<\delta^{\frac12-\frac{1}{2M_{\tau}}}\}, \\
U_{\delta}(z):=&\{y\in \mathbb{R}^2:|y_1-z_1|<\delta^{\frac{1}{M_{\tau}}},
  |y_2-z_2|<\delta^{\frac12+\frac{1}{2M_{\tau}}}\},\\
  K_{\delta}:=&\{w\in \mathbb{R}^3: |w_1|\leq \delta^{\frac{1}{M_{\tau}}}, 	
	|w_2|\leq\delta^{\frac12+\frac{1}{2M_{\tau}}}, |w_3|\lesssim\delta\}. \\
\end{align*}
The exponents of $\delta$ may seem mysterious now, but it will become clear during the proof that their choice is optimal.\\

In order to apply  Corollary \ref{Counterco}, we need to show  that for all $z\in \Omega_{\delta}$ and $ y\in U_{\delta}(z)$, we have 
$$(y-z, \Psi(y)-\Psi(z))\in K_{\delta}.$$
The first two conditions $|y_1-z_1|<\delta^{\frac{1}{M_{\tau}}}$ and  $|y_2-z_2|\leq \delta^{\frac12+\frac{1}{2M_{\tau}}}$ are obvious by the definition of $U_{\delta}(z)$.
\smallskip

Note in particular that  this shows that  $ y\in U_{\delta}(z)$ satisfies  similar conditions as $z\in\Omega_\delta$, i.e., 
\begin{equation}\label{kotz}
	|y_1-y_{0,1}|\lesssim\delta^{\frac{1}{M_{\tau}}},\qquad
	|y_2-y_{0,2}|\lesssim\delta^{\frac12-\frac{1}{2M_{\tau}}}.
\end{equation}

It remains to show that
\begin{equation}\label{Psiyz}
|\Psi(y)-\Psi(z)|=|\Psi_{\rm{nl}}(y)-\Psi_{\rm{nl}}(z)|\lesssim \delta.
\end{equation}
To prove this, let us apply a nonlinear change of coordinates
$$\tilde y_1:=y_1-y_{0,1},\qquad \tilde y_2:=y_2-y_2^c(y_1)$$
(and accordingly $\tilde z_1:=z_1-y_{0,1},\  \tilde z_2:=z_2-y_2^c(z_1)$),
so that the expansion \eqref{phiexpansexy} of $\Psi_{\rm{nl}}$ transforms into
\begin{equation}\label{phiexpansetilde}
	\tilde\Psi(\tilde y):=\Psi_{\rm{nl}}(y)
=\tilde y_1^{M_\tau} \,W(\tilde y_1)+ \tilde y_2^2  \,V(\tilde y)
\end{equation}
for suitable analytic functions $W$ and $V.$
Using the expansion \eqref{ytwocritical} of $y_2^c$, we see that
\begin{equation}\label{wichtig}
	|y_2^c(y_1)-y_{0,2}|\lesssim |y_1-y_{0,1}|^{M_\tau-1}
 \lesssim \delta^{\frac{M_\tau-1}{M_\tau}}
 \lesssim \delta^{\frac12+\frac{1}{2M_{\tau}}}.
\end{equation}
The last inequality holds since $M_\tau\geq 3.$ Since $y\in U_\delta(z),$ we also have the analogous estimate
$$|y_2^c(z_1)-y_{0,2}|\lesssim \delta^{\frac12+\frac{1}{2M_{\tau}}}.$$
This has two implications:
Firstly, using the definition of $\Omega_\delta$ and \eqref{kotz}, we obtain that
$$|\tilde y_1|,|\tilde z_1|\lesssim \delta^{\frac{1}{M_{\tau}}},\qquad
|\tilde y_2|,|\tilde z_2|\lesssim \delta^{\frac12-\frac{1}{2M_{\tau}}}.$$

Moreover,
$$|\tilde y_2-\tilde z_2| \lesssim |y_2-z_2|+|y_2^c(y_1)-y_2^c(z_1)|
\lesssim \delta^{\frac12+\frac{1}{2M_{\tau}}}.$$

It is easy to check from \eqref{phiexpansetilde} that for any $|v_1|\lesssim\delta^{\frac{1}{M_{\tau}}}$,
$|v_2|\lesssim \delta^{\frac12-\frac{1}{2M_{\tau}}}$, the derivatives of $\tilde\Psi$ are bounded by
\begin{align*}
	|\partial_1\tilde\Psi(v)|\lesssim& |v_1|^{M_\tau-1}+|v_2|^2
\leq \delta^{\frac{M_\tau-1}{M_\tau}} + \delta^{1-\frac1{M_\tau} } =2\delta^{\frac{M_\tau-1}{M_\tau}},\\
|\partial_2\tilde\Psi(v)|\lesssim& |v_2|
\lesssim \delta^{\frac12-\frac1{2M_\tau}},
\end{align*}
so that for some $v$ in the convex hull of $\tilde y$, $\tilde z$ we have
\begin{align*}
	|\Psi(y)-\Psi(z)|=&|\tilde\Psi(\tilde y)-\tilde\Psi(\tilde z)| \\
	\leq& |\partial_1\Psi(v)||\tilde y_1-\tilde z_1|+
			|\partial_2\Psi(v)||\tilde y_2-\tilde z_2| \\
	\lesssim& \delta^{\frac{M_\tau-1}{M_\tau}}\delta^{\frac{1}{M_{\tau}}}
	            +\delta^{\frac12-\frac1{2M_\tau}}\delta^{\frac12+\frac{1}{2M_{\tau}}}
	 =2\delta,
\end{align*}
which proves \eqref{Psiyz}. 

Since $|\Omega_\delta|\sim \delta^{\frac12+\frac{1}{2M_{\tau}}},  |U_\delta(z)|\sim \delta^{\frac{1}{M_{\tau}} +\frac12+\frac{1}{2M_{\tau}}}$ and $|K_\delta|\sim \delta^{1+\frac{1}{M_{\tau}} +\frac12+\frac{1}{2M_{\tau}}},$ Corollary \ref{Counterco} thus shows that the condition $p\geq\frac{2(M_{\tau}+1)}{M_{\tau}+3}$ is necessary for $L^p$-boundedness of $\mathcal M.$ 
\smallskip

This finishes the proof of Theorem \ref{necessary}(b).

%%%%%%%%%%%%%%%%%%%%%%%%%%%%%%%%%%%%%%%%%%%%%%%%%%%%%%%%%%%%%%%%%%%%%%%%%%%%%%%%%%%%%%%%%%%%%%%%%%%%

\subsection{Conditions enforced by roots of Type $\mathrm{C}$}\label{onCaseC}

Here we want to prove Theorem \ref{necessary}(c). Assume that  $\mathcal M$ is the maximal operator defined by the averaging operators \eqref{AvU}, and let $\gamma_\tau, \tau\in R,$ be a real root of  $\mathrm{H}\Phi$ of multiplicity $N_\tau$  which is of Type  $\mathrm{C}.$  

We shall make use of the results in Lemma \ref{TypeCnf}, assuming without loss of generality that  $\Phi_{2}$ vanishes and $\Phi_{1}$ does not vanish along $\gamma_{\tau}.$
 Then we need to show that if the maximal operator $\mathcal{M}$ is $L^{p}$-bounded, then  $p \ge \frac{2M_{\tau}}{M_{\tau} +1}.$

\smallskip
First, recall from Lemma \ref{Ttau=0} that $\Phi$ satisfies the transversality condition \eqref{Trans} along $\gamma_\tau.$ 
Thus, if we choose a sufficiently small $0<\epsilon \ll1$ and put 
\[\Omega := \{z\in U: z_{1} \in (\frac{\epsilon}{2},\epsilon), |z_{2} - \tau z^{\frac{\kappa_{2}}{\kappa_{1}}}_{1}| < \epsilon\} \]
and  $Z:=\Psi,$ then the graph $\Sigma:=\{(z,\Psi(z)): z\in \Omega\}$ satisfies a quantitative  transversality condition \eqref{TVQ} as required in Proposition \ref{twopointone} and Corollary \ref{Counterco}.

Next, for  all $0 < \delta \ll \epsilon \ll 1$, we define
 \begin{align*}
 \Omega_{\delta} :=& \{z\in \Omega: z_{1} \in (\frac{\epsilon}{2},\epsilon), |z_{2} - \tau z ^{\frac{\kappa_{2}}{\kappa_{1}}}_{1}| < \delta\},\\
 U_{\delta}(z):=&\{x\in \Omega: |x_{1}-z_{1}| \le \delta^{M_{\tau}}, |x_{2}-z_{2}| \le \delta \},\\
 K_{\delta}:=&[- \delta^{M_{\tau}}, \delta^{M_{\tau}}] \times  [-\delta, \delta] \times [-c\, \delta^{M_{\tau}},  c\, \delta^{M_{\tau}}],
 \end{align*}
with a  constant $c>0$ to be chosen soon. 

Let us show that for all $z\in \Omega_{\delta}$ and all  $ x\in U_{\delta}(z)$ we have   $(x,\Phi(x))- (z,\Phi(z))\in K_{\delta}$, provided $c>0$ is chosen sufficiently large.

Indeed, clearly $|x_1-z_1|\le \delta^{M_{\tau}},$ and $|x_2-z_2|\le \delta.$  As for the  third component,   we decompose 
\[\Phi(x)-\Phi(z)=\big (\Phi(x_1,x_2)-\Phi(z_1,x_2)\big) + \big (\Phi(z_1,x_2)-\Phi(z_1,z_2)\big),\]
and make use of the normal form \eqref{Phiexpansion} for $\Phi,$ i.e., 
\[\Phi(x_{1},x_{2}) = \Phi(x_{1},\tau x^{\frac{\kappa_{2}}{\kappa_{1}}}_{1}) + (x_{2} -\tau x^{\frac{\kappa_{2}}{\kappa_{1}}}_{1})^{M_{\tau}}V(x_1,x_2).\]
From this, we easily deduce that $\Phi(x_1,x_2)-\Phi(z_1,x_2)=\mathcal{O}(\delta^{M_{\tau}}),$ and 
$$\Phi(z_1,x_2)-\Phi(z_1,z_2)=\mathcal{O}(\delta^{M_{\tau}-1}\delta)=\mathcal{O}(\delta^{M_{\tau}}),
$$
so that indeed $(x,\Phi(x))- (z,\Phi(z))\in K_{\delta}$ if $c>0$ is chosen sufficiently large.

Thus, we can apply  Corollary \ref{Counterco}.  Note that here we have $s=2M_{\tau}+1$, $\omega=1$, and $u=M_{\tau}+1$, so that we see that  $\mathcal{M}$ cannot be $L^{p}$-bounded unless $p \ge \frac{2M_{\tau}}{M_{\tau}+1}$.

We have thus also finished the proof of Theorem \ref{necessary}(c), hence the proof of Theorem \ref{necessary}. \qed

%%%%%%%%%%%%%%%%%%%%%%%%%%%%%%%%%%%%%%%%%%%%%%%%%%%%%%%%%%%%%%%%%%%%%%%%%%%%%%%%%%%%%%%%%%%%%%%%%%%%

\section{Sufficiency of the assumptions in Theorem \ref{theorem1}: Reduction to narrow $\kappa$-homogeneous neighborhoods  of real roots of  $\mathrm{H}\Phi$}\label{suff1}

\subsection{The case where  $\mathrm{H}\Phi$ has no real root}\label{noroot}

In this case, $\mathrm{H}\Phi$  does not vanish away from the the origin, and Theorem \ref{theorem1}  is then confirmed by

 \begin{proposition}\label{remarkGammaj}
 Assume that {\color{blue} $|\kappa|\le 1,$ }and that $\mathrm{H}\Phi$ is non-vanishing away from the origin. Then $\mathcal{M}$ defined by (\ref{MS}) is $L^{p}$-bounded if $p>3/2$, and unbounded for $p\le 3/2,$  so that $p_c=p_\Phi.$
\end{proposition}

\begin{proof} 
The necessity of the condition $p>3/2$ has already been shown in Proposition \ref{pnn-1}. The proof of the sufficiency of this condition is standard and will only be sketched  briefly.

Choose a smooth cut-off $\chi_0\in C_0^\infty(\Bbb R^2)$ supported where $|x|\le 2$ and such that $\chi_0(x)=1$ for $|x|\le 1,$ and put $\chi_1(x):=\chi_0(x)-\chi_0(\delta_2 (x)),$ where $\delta_r$ denotes again our dilation $\delta_r(x):=(r^{\kappa_1}x_1,r^{\kappa_2}x_2), \, r>0.$ Then 
 \begin{equation}\label{dyadicdecomposition}
 \sum_{j\in \mathbb{Z}}\chi_1(\delta_{2^j}(x)) =1\quad \text{for every }\quad x \in \mathbb{R}^{2} \backslash \{0\}.
 \end{equation}
 For any given $j\in \Bbb Z$ we define the maximal operator $\mathcal M_j$ by 
 \[
 \mathcal M_j f(y):=\sup\limits_{t>0}\Big|\int_{\mathbb{R}^2}f(y-t(x,\Phi(x)))\, \psi(x)\,\chi_1(\delta_{2^j}(x))\, dx\Big|,\qquad y \in \mathbb{R}^3.
 \]
 Assuming without loss of generality that $\psi$ is supported in a sufficiently small neighborhood of the origin, we may then estimate 
\begin{equation}\label{Mdecomp}
 \|\mathcal{M} \|_{L^{p}\rightarrow L^{p}} \le \sum_{j=0}^\infty\|\mathcal{M}_j\|_{L^{p}\rightarrow L^{p}}.
\end{equation}

Moreover, by exploiting also the scalings $D_r$ of $\Bbb R^3$ from \eqref{Dr}, one easily sees that 
\begin{equation}\label{Mscale}
 \|\mathcal{M}_j\|_{L^{p}\rightarrow L^{p}}=2^{-|\kappa| j} \,  \|\tilde{\mathcal{M}}_{j}\|_{L^{p}\rightarrow L^{p}},
  \end{equation}
 where
\[
\tilde{\mathcal{M}}_{j} f(y):=\sup\limits_{t>0}\Big|\int_{\mathbb{R}^2} f(y-t(x,\Phi(x)))\, \psi(\delta_{2^{-j}}(x))\,\chi_1(x)\, dx\Big|.
\]

Since $\mathrm{H}\Phi$ is non-vanishing on the support of $\chi_1,$ the corresponding part of the hypersurface $S$ 
has non-vanishing Gaussian curvature.  Thus, by \cite{F86} and \cite{Gr81}, the  maximal operators  $\tilde{\mathcal{M}}_{j}$ satisfy uniform  $L^{p}$-estimates for every given $p>3/2$.  In combination with \eqref{Mdecomp} and \eqref{Mscale} this implies that $ \mathcal M$ is bounded on $L^p$ for every $p>3/2.$
\end{proof} 
\color{black}

An example for this situation is $\Phi(x_{1},x_{2})= x^{2}_{1}x_{2}+x_{1}x^{3}_{2}$, where $\mathrm{H}\Phi(x_{1},x_{2}) = -4x^{2}_{1}-9x_{2}^{4}$.

%%%%%%%%%%%%%%%%%%%%%%%%%%%%%%%%%%%%%%%%%%%%%%%%%%%%%%%%%%%%%%%%%%%%%%%%%%%%%%%%%%%%%%%%%%%%%%%%%%%%

\subsection{Reduction to narrow neighborhoods of real roots of $\mathrm{H}\Phi$ }\label{reducttoroots}
Assume now that $Z_{\mathrm{H}\Phi}\ne \emptyset.$  For every root $\zeta \in  Z_{\mathrm{H}\Phi}\cap S^1,$ we define a cut-off $\rho_\zeta$ which is smooth away from the origin, $\kappa$-homogeneous of degree 0 and localizes to a sufficiently narrow $\kappa$-homogeneous neighborhood of the root $\gamma_{(\zeta)}.$ 

More specifically, suppose that $\gamma_{(\zeta)}$ is contained in the right half-plane $x_1>0,$ so that  we may parametrize it by 
$\gamma_{\tau}$ as in Subsection \ref{main result}, with $\tau=\tau(\zeta)\in R,$ where 

\[\gamma_{\tau}= \bigl\{(x_1,x_2)\in \mathbb{R}^2: x_1>0 \ \text{and}\ x_{2}=\tau x_1^{\frac{\kappa_2}{\kappa_1}}\bigl\}.\]

In order to defray the notation, we shall in the sequel usually abbreviate $\kappa_2/\kappa_1$ by
$$
a=a_\kappa:=\kappa_2/\kappa_1.
$$
Let  $\chi_0$ be the smooth cut-off from the preceding subsection, and put
\begin{equation}\label{rhotau}
\rho_{(\zeta)}(x)=\rho_\tau(x):=\chi_0\Big(\frac{x_2-\tau x_1^{a}}{\ve x_1^{a}}\Big), \qquad x_1>0,
\end{equation}
where $\ve>0$ is supposed to be sufficiently small. For roots contained in the left half-plane $x_1<0$ we can define the cut-off $\rho_{(\zeta)}$ analogously, and for roots lying on the $x_2$-axis by swapping the roles of $x_1$ and $x_2.$ 

By choosing $\epsilon$ sufficiently small, we may assume that the cut-offs $\rho_{(\zeta)}, \zeta \in  Z_{\mathrm{H}\Phi}\cap S^1,$ have disjoint supports. We then decompose 
\[
1=\rho_{\rm G} +\sum\limits_{\zeta \in  Z_{\mathrm{H}\Phi}\cap S^1}\rho_{(\zeta)}  \qquad \text{on} \quad \Bbb R^2\setminus\{0\},
\]
and define corresponding maximal operators
\[
 \mathcal M_{(\zeta)} f(y):=\sup\limits_{t>0}\Big|\int_{\mathbb{R}^2}f(y-t(x,\Phi(x)))\, \psi(x)\,\rho_{(\zeta)}(x)\, dx\Big|,\qquad  
  \zeta \in  Z_{\mathrm{H}\Phi}\cap S^1,
 \]
 and 
 \[
 \mathcal M_{\rm G} f(y):=\sup\limits_{t>0}\Big|\int_{\mathbb{R}^2}f(y-t(x,\Phi(x)))\, \psi(x)\,\rho_{\rm G}(x)\, dx\Big|.
 \]

As for  $\mathcal M_{\rm G},$ observe that $\mathrm{H}\Phi$ does not vanish on $\supp \chi_1\cap \supp \rho_G,$ where $\chi_1$ is as in the preceding subsection. Thus we may apply the same arguments to $\mathcal M_{\rm G}$ that we had applied in the proof of Proposition \ref{remarkGammaj} to conclude that
\begin{equation}\label{M0est}
 \|\mathcal{M}_{\rm G}\|_{L^{p}\rightarrow L^{p}}<\infty \qquad \text{for every}\quad p>3/2.
  \end{equation}
  
  This leaves us with the estimation of the maximal operators  $ \mathcal M_{(\zeta)}.$ To this end, we shall again restrict ourselves to roots contained in the right half-plane $x_1>0,$ where we parametrize the roots by the real parameters  $\tau\in R,$ and shall accordingly put $\rho_\tau=\rho_{(\zeta_\tau)}$ and $\mathcal{M}_\tau=\mathcal M_{(\zeta_\tau)},$ i.e., 
   \[
 \mathcal M_\tau f(y):=\sup\limits_{t>0}\Big|\int_{\mathbb{R}^2}f(y-t(x,\Phi(x)))\, \psi(x)\,\rho_\tau(x)\, dx\Big|,
 \]
 with $\rho_\tau$ as defined in \eqref{rhotau}.
 
Again, by means of the dyadic decomposition \eqref{dyadicdecomposition}, we can then bound 
$$
\|\mathcal{M_\tau}\|_{L^{p}\rightarrow L^{p}}\le C_p\, \sup\limits_{j\ge 0} \|\tilde{\mathcal{M}}_{\tau,j} \|_{L^{p}\rightarrow L^{p}},
$$
where 
\[
\tilde{\mathcal{M}}_{\tau,j} f(y):=\sup\limits_{t>0}\Big|\int_{\mathbb{R}^2} f(y-t(x,\Phi(x)))\, \psi(\delta_{2^{-j}}(x))\, \rho_\tau(x) \,\chi_1(x)\, dx\Big|.
\]
Since $x_1\sim 1$ on the support of $\rho_\tau \chi_1,$ we may reduce these estimates to uniform estimates 
\begin{equation}\label{estMtau}
\|\tilde{\mathcal{M}}_\tau\|_{L^{p}\rightarrow L^{p}}\le C_{p,C_k}
\end{equation}
of maximal operators of the form
\begin{equation}\label{tMtau}
\tilde{\mathcal{M}}_{\tau} f(y):=\sup\limits_{t>0}\Big|\int_{\mathbb{R}^2} f(y-t(x,\Phi(x)))\, \tilde \psi(x)\, \rho_\tau(x) \,\chi_1(x_1)\, dx\Big|,
\end{equation}
where here $\chi_1$ denotes a one-variate smooth cut-off function supported in $[1/2,2],$ and where $\tilde\psi$ is any  smooth function on $\Bbb R^2$ whose $C^k$-norm is bounded by $C_k$, provided $k$ is chosen sufficiently large.

More precisely, we shall prove

\begin{theorem}\label{theorem2.1}
 Let $\Phi$ be as in Theorem \ref{theorem1}, and let $\tau\in R.$  Then the following hold, provided $\epsilon>0$ in the definition of $\rho_\tau$ is chosen sufficiently small.
  \begin{itemize}
\item[(a)] If  $\gamma_{\tau}$ is of Type A, then estimate \eqref{estMtau} holds true for every $p > 3/2.$
\item[(b)] If  $\gamma_{\tau}$ is of Type $\mathrm{B_{T}}$, then estimate \eqref{estMtau} holds true for every $p >\max\big\{\frac{3}{2}, \frac{2(M_{\tau}+1)}{M_{\tau}+3}\big\}$; and if  $\gamma_{\tau}$ is of Type $\mathrm{B_{NT}}$, then estimate \eqref{estMtau} holds true for every $p>3/2.$
\item[(c)] If  $\gamma_{\tau}$ is of Type C, then estimate \eqref{estMtau} holds true for every $p >\max\big\{\frac{3}{2}, \frac{2M_{\tau}}{M_{\tau}+1}\big\}$.
 \end{itemize}
\end{theorem}
The estimates in this theorem will complete the proof of Theorem \ref{theorem1}.

 We denote by 
 \begin{equation}\label{dmiudef}
 \tilde m_\tau(\xi)=\widehat{d\tilde \mu_\tau}(\xi):=\int_{\mathbb{R}^{2}}e^{-i\bigl(\xi_1 x_1+\xi_2 x_2+\xi_3\Phi(x)\bigl)} \tilde \psi(x)\, \rho_\tau(x) \,\chi_1(x_1)\, dx
 \end{equation}
 the Fourier transform of the part $d\tilde \mu_\tau$ of the surface measure of $S$ corresponding to the density $ \tilde \psi(x)\, \rho_\tau(x) \,\chi_1(x_1).$

%%%%%%%%%%%%%%%%%%%%%%%%%%%%%%%%%%%%%%%%%%%%%%%%%%%%%%%%%%%%%%%%%%%%%%%%%%%%%%%%%%%%%%%%%%%%%%%%%%%%

\section{Auxiliary results}\label{auxiliary}
Given any smooth and bounded Fourier multiplier $m$ on $\mathbb{R}^{n}$, we shall denote by $T_m$ the associated Fourier multiplier operator defined by
\[\widehat{T_{m}f}:=m\hat{f}, \quad \quad f\in \mathcal{S}(\mathbb{R}^{n}).\]
 We shall denote by $\mathcal{M}_{m}$  the corresponding maximal operator, i.e.,
\[\mathcal{M}_{m}f(x):=\sup_{t>0}|T_{m(t\,\cdot)}f(x)|, \quad \quad f\in \mathcal{S}(\mathbb{R}^{n}).\]
By $\mathcal{F}^{-1}$ we shall denote the inverse Fourier transform. Note that the  maximal operator defined in \eqref{tMtau} can then be written as 
\[
\tilde{\mathcal{M}}_{\tau}=\mathcal{M}_{\tilde m_\tau},
\]
with $\tilde m_\tau$ defined in \eqref{dmiudef}.

The following easily proven properties will be useful: For every invertible linear transformation $A\in \rm{GL}(n,\Bbb R),$ we have
\begin{eqnarray}\label{mGLn}
\| \mathcal{F}^{-1}(m \circ A)\|_{L^1(\Bbb R^n)} &=&\| \mathcal{F}^{-1}m\|_{L^1(\Bbb R^n)};\\
\|\mathcal {M}_{m\circ A}\|_{L^{p}(\mathbb{R}^{n}) \rightarrow L^{p}(\mathbb{R}^{n}) } 
&=&\|\mathcal {M}_{m}\|_{L^{p}(\mathbb{R}^{n}) \rightarrow L^{p}(\mathbb{R}^n) } \qquad (1\le p\le \infty). \label{mGLn2}
\end{eqnarray}

%%%%%%%%%%%%%%%%%%%%%%%%%%%%%%%%%%%%%%%%%%%%%%%%%%%%%%%%%%%%%%%%%%%%%%%%%%%%%%%%%%%%%%%%%%%%%%%%%%%%

\subsection{Reduction of maximal estimates to multiplier estimates}\label{Mtomult}

The maximal estimates in Theorem \ref{theorem2.1} will be reduced to estimates of Fourier multipliers by means of \cite[Theorem 9.2]{BDIM25}:

\begin{theorem}\label{multipliertomaximal}
Let $m$ be a smooth Fourier multiplier supported in the annulus $1\le |\xi| \le 2$, and denote by $\dot{m}$ its radial derivative, i.e., $\dot{m}(\xi):=\xi\cdot\nabla m(\xi)=\frac {d}{dt}m(t\xi)|_{t=1}$. Let also $1<p<\infty$, $\epsilon>0$ and $\lambda \gg 1$ be given, and put $\tilde{m}:=\lambda^{-1} \dot{m}$.

Assume that the following two estimates hold true for $m$, and for $\tilde{m}$ in place of $m$ as well:
\[
\|T_{m}\|_{L^{p}(\mathbb{R}^{n}) \rightarrow L^{p}(\mathbb{R}^{n}) } \le A, \leqno{(i)}
\]
\[
\int_{|x|>r}|\mathcal{F}^{-1}m(x)|dx \le B (1+r)^{-\epsilon} \quad \quad \text{for every} \quad \quad r>0. \leqno{(ii)}
\]
Then
\[
\|\mathcal{M}_{m}\|_{L^{p}(\mathbb{R}^{n}) \rightarrow L^{p}(\mathbb{R}^{n}) } \le C_{p} \lambda^{\frac{1}{p}}A[\log(2+B/A)]^{|\frac{1}{p}-\frac{1}{2}|}.
\]
In particular, if $A \le 1/2$ and $B \ge 2$, then, for any $\delta>0$,
\[
\|\mathcal{M}_{m}\|_{L^{p}(\mathbb{R}^{n}) \rightarrow L^{p}(\mathbb{R}^{n}) } \le C_{p,\delta} \lambda^{\frac{1}{p}}(A^{1-\delta} +A\log B).
\]
\end{theorem}

Since the constant $B$  in condition (ii) of the theorem  enters only logarithmically in the bound for the maximal operator, the following simple lemma gives in many cases an easily checkable  sufficient criterion  for verifying condition (ii).

\begin{lemma}\label{generalkernelestimate}
	Let $R>0, $ and let $\{m_\lambda\}_{\la\ge 1}$ be a family of smooth functions on $\Bbb R^n$ supported in the ball $B(0,R)$ and satisfying growth estimates 	 
\begin{equation}\label{mgrowth}
 |\partial_\xi^\alpha m_\lambda(\xi)| \leq C_\alpha \lambda^{C(1+|\alpha|)}
\end{equation}
	for every  $\alpha\in\mathbb{N}^n$, $\xi\in\mathbb{R}^n$ and $\lambda\gg 1$.  Then there exist constants $C_1,C_2>0$ such that
	$$ \int_{|x|>r} |\mathcal{F}^{-1}{m}_{\lambda}(x)|dx \leq C_2 \lambda^{C_1}(1+r)^{-1}$$
	for all $r>0$, $\lambda\gg 1$.
\end{lemma}

\begin{proof}
Estimate 	\eqref{mgrowth} with  $\alpha=0$ immediately implies that 
$$
\int_{B(0,1)} |\mathcal{F}^{-1}{m}_{\lambda}(x)|\, dx \lesssim \lambda^C,
$$
so that we may assume without loss of generality that $r\geq1$.

We claim that for any $N>0$, there exists a constant $C_N>0$ such that
	\begin{equation}\label{5may26}
		|\mathcal{F}^{-1}{m}_{\lambda}(x)|\leq C_N R^n \lambda^C \left(1+\frac{|x|}{\lambda^C}\right)^{-N}.
	\end{equation}
	For $|x|\leq \lambda^C$, this is easy to see, so let us assume $|x|> \lambda^C$, and assume, say, $|x_1|\sim |x|$. Integrating by parts in $\xi_1$, we obtain
	\begin{eqnarray*}
		|\mathcal{F}^{-1}{m}_{\lambda}(x)|&=& c_n\biggl| \int_{B(0,R)} e^{ix\xi}m_\lambda(\xi) \,d\xi \biggl| \\
		&\lesssim& \int_{B(0,R)} |x_1|^{-N} |\partial_{\xi_1}^N m_\lambda(\xi)|\,d\xi \\
		&\lesssim& |x|^{-N} R^n\lambda^{C(1+N)} = R^n \lambda^C \left(\frac{|x|}{\lambda^C}\right)^{-N},
	\end{eqnarray*}
	from which the claim \eqref{5may26} follows. For $N=n+1$, \eqref{5may26} gives
\begin{eqnarray*}
		\int_{|x|>r} |\mathcal{F}^{-1}{m}_{\lambda}(x)|\,dx
		&\lesssim& R^n\lambda^{C(n+2)} \int_{|x|>r} |x|^{-n-1}\,dx 
		\lesssim R^n \lambda^{C(n+2)} r^{-1}.
\end{eqnarray*}
\end{proof}

%%%%%%%%%%%%%%%%%%%%%%%%%%%%%%%%%%%%%%%%%%%%%%%%%%%%%%%%%%%%%%%%%%%%%%%%%%%%%%%%%%%%

\subsection{$L^1$-estimates for the inverse Fourier transform of a class of oscillatory multipliers}\label{oscillm}
\begin{lemma}\label{lemma3.2}
Let\begin{align}
 m(\eta)&=  e^{i\lambda \eta_{3}\bigl(c_{1}  (\frac{\eta_{1}}{\eta_{3}})^{\alpha_{1}} +c_{2} \sigma^{\rho} (\frac{\eta_{1}}{\eta_{3}})^{\alpha_{2}}\frac{\eta_{2}}{\eta_{3}} + \sigma \Psi (\frac{\eta_{1}}{\eta_{3}},\frac{\eta_{2}}{\eta_{3}})\bigl)}
 a( \eta)\beta(|\eta|)\chi_{1}(\frac{\eta_{1}}{\eta_{3}})\chi_{2}(\frac{\eta_{2}}{\eta_{3}}), \qquad \eta \in \Bbb R^3. \nonumber
\end{align}
Here, we assume that $a$ is smooth and $\beta$, $\chi_{1}$ and $\chi_{2}$ are smooth cut-off functions with  $\supp \beta \subset [1,2]$, \supp $\chi_{1} \subset [1,2]$  and $\supp \chi_{2} \subset [-1, 1],$ and the real-valued function  $\Psi$ is assumed to be  smooth  on  $ \supp \chi_{1}\times  \supp\chi_{2}.$ In addition, we assume that $c_{1},c_{2},\alpha_{1},\alpha_{2},\rho$ are real numbers. The parameters $\lambda$, $\sigma$ and $\rho$ are assumed to satisfy $\lambda \gg 1$,  $\sigma \ll 1$,   $\lambda \sigma  \ge 1$ and  $\rho\ge \frac{1}{2}. $ Then we can estimate
\[  \| \mathcal{F}^{-1}m\|_{L^1(\Bbb R^3)} \le C \lambda^{\frac{1}{2}}(\lambda \sigma)^{\frac{1}{2}}.\]
holds true, where the constant $C$ does not depend on $\la, \sigma $ and $\rho.$

\end{lemma}

\begin{proof} 
We first observe that for $\eta\in \supp m$ we have $|\eta_3|\sim 1.$ 

By means of a smooth partition  of unity we can  decompose $m$ as $\sum_{(v_{1},v_{2})}m_{v_{1},v_{2}}$, where each  $m_{v_{1},v_{2}}$ is of the form
$$
m_{v_{1},v_{2}}(\eta)=\chi_0 \Big(\la^{\frac 12}\big(\frac{\eta_1}{\eta_3}-v_1\big) \Big) \chi_0 \Big((\la\sigma)^{\frac 12}\big(\frac{\eta_2}{\eta_3}-v_2\big) \Big) m(\eta),\qquad (v_{1},v_{2}) \in [1,2] \times [-1,1],
$$
with a suitable smooth cut-off  function $\chi_0$  supported in $[-1,1],$ so that $m_{v_1,v_2}$
is smooth and supported in $\{\eta: 1 \le |\eta| \le 2, |\frac{\eta_{1}}{\eta_{3}}-v_{1}| \le \lambda^{- \frac{1}{2}}, |\frac{\eta_{2}}{\eta_{3}}-v_{2}| \le (\lambda \sigma)^{-\frac{1}{2}}\}.$
Since the number  of such pairs  $(v_{1},v_{2})$ is bounded by a constant multiple of $\lambda^{\frac{1}{2}} (\lambda \sigma)^{\frac{1}{2}}$, Lemma \ref{lemma3.2} will follow  if we can prove that for each $(v_{1},v_{2})$,
\begin{equation}\label{mv1v2}
\|\mathcal{F}^{-1}m_{v_{1},v_{2}}\|_{L^{1}} \le C.
\end{equation}

For a fixed $m_{v_{1},v_{2}}$, we perform the linear change of variables $\eta_{1}=\lambda^{-\frac{1}{2}} \eta_{1}^{\prime} + v_{1}\eta^{\prime}_{3}$, $\eta_{2}=(\lambda \sigma)^{-\frac{1}{2}} \eta_{2}^{\prime} + v_{2}\eta^{\prime}_{3}$, $\eta_{3} = \eta^{\prime}_{3}$, where  $\eta^{\prime}=(\eta_{1}^{\prime},\eta^{\prime}_{2},\eta^{\prime}_{3})$, so that
$\chi_0 \Big(\la^{\frac 12}\big(\frac{\eta_1}{\eta_3}-v_1\big) \Big) \chi_0 \Big((\la\sigma)^{\frac 12}\big(\frac{\eta_2}{\eta_3}-v_2\big)\Big)=\chi_0(\eta'_1)\chi_0(\eta'_2).$ 

Let $\tilde m_{v_1,v_2}$ express $m_{v_1,v_2}$ in these new coordinates, i.e., $\tilde m_{v_1,v_2}(\eta')=m_{v_1,v_2}(\eta). $  Then by
\eqref{mGLn}, we have $\|\mathcal{F}^{-1}m_{v_{1},v_{2}}\|_{L^{1}}=\|\mathcal{F}^{-1}\tilde m_{v_{1},v_{2}}\|_{L^{1}}.$

For convenience, putting $s_1:=\eta_1/\eta_3, \, s_2:=\eta_2/\eta_3,$ we also denote by
\[
F(s_1,s_2) =   c_{1}  s_1^{\alpha_{1}} +c_{2}\sigma^{\rho} s_1^{\alpha_{2}}s_2 + \sigma \Psi (s_1,s_2)
\]
the main factor of the  phase in $m(\eta).$ By Taylor expansion of $F$ at $(v_{1},v_{2})$, we obtain that  in the new coordinates $\eta',$ the complete phase in $\tilde m$ can written as 
\begin{align}
&\lambda \eta^{\prime}_{3}F\biggl(\lambda^{-\frac{1}{2}}\frac{\eta^{\prime}_{1}}{\eta^{\prime}_{3}} +v_{1},(\lambda \sigma)^{-\frac{1}{2}}\frac{\eta^{\prime}_{2}}{\eta^{\prime}_{3}}+v_{2}\biggl) \nonumber\\
&=\lambda \eta^{\prime}_{3} F(v_{1},v_{2})+ \lambda \cdot \lambda^{-\frac{1}{2}} \eta^{\prime}_{1} F_{1}(v_{1},v_{2}) + \lambda \cdot (\lambda\sigma)^{-\frac{1}{2}} \eta^{\prime}_{2} F_{2}(v_{1},v_{2}) + \lambda \cdot \lambda^{-1} \frac{(\eta^{\prime}_{1})^{2}}{2\eta^{\prime}_{3}} F_{11}(v_{1},v_{2}) \nonumber\\
&\quad \quad +2\lambda \cdot \lambda^{-\frac{1}{2}} \cdot (\lambda \sigma)^{-\frac{1}{2}}\frac{\eta^{\prime}_{1}\eta^{\prime}_{2}}{2 \eta^{\prime}_{3}} F_{12}(v_{1},v_{2})+ \lambda \cdot (\lambda \sigma)^{-1} \frac{(\eta^{\prime}_{2})^{2}}{2\eta^{\prime}_{3}} F_{22}(v_{1},v_{2})+\lambda \eta^{\prime}_{3}R(\frac{\eta^{\prime}_{1}}{\eta^{\prime}_{3}}, \frac{\eta^{\prime}_{1}}{\eta^{\prime}_{3}}), \nonumber
\end{align}

where we have collected all the remainder terms in $R$.  Note  that 
\[|\lambda \cdot \lambda^{-\frac{1}{2}} (\lambda \sigma)^{-\frac{1}{2}} F_{12}(v_{1},v_{2})| \lesssim \lambda \cdot \lambda^{-\frac{1}{2}} (\lambda \sigma)^{-\frac{1}{2}}(\sigma^{\rho} +\sigma)= \sigma^{\rho-\frac{1}{2}} +\sigma^{\frac 12}\lesssim 1,  \]
\[|\lambda \cdot (\lambda \sigma)^{-1} F_{22}(v_{1},v_{2})| \lesssim \lambda \cdot   (\lambda \sigma)^{-1}\sigma  \lesssim   1, \]
 and that the remainder term $\la R$ only picks up higher order terms which are even smaller.

Given this, it is now easily seen that  for each multi-index $\alpha$,
\begin{align}
&\biggl|\partial_{\eta^{\prime}}^{\alpha} \biggl(\lambda \cdot \lambda^{-1} \frac{(\eta^{\prime}_{1})^{2}}{2\eta^{\prime}_{3}} F_{11}(v_{1},v_{2})  +2\lambda \cdot \lambda^{-\frac{1}{2}} \cdot (\lambda \sigma)^{-\frac{1}{2}}\frac{\eta^{\prime}_{1}\eta^{\prime}_{2}}{2 \eta^{\prime}_{3}} F_{12}(v_{1},v_{2})\nonumber\\
&\quad \quad + \lambda \cdot (\lambda \sigma)^{-1} \frac{(\eta^{\prime}_{2})^{2}}{2\eta^{\prime}_{3}} F_{22}(v_{1},v_{2}) +\lambda \eta^{\prime}_{3}R(\frac{\eta^{\prime}_{1}}{\eta^{\prime}_{3}}, \frac{\eta^{\prime}_{1}}{\eta^{\prime}_{3}}) \biggl)\biggl| \lesssim 1.\nonumber
\end{align}

Then integrations by parts with respect to the $\eta^{\prime}$-variable in $\tilde m_{v_{1},v_{2}}$ yield that, for each positive integer $N$, $|\mathcal{F}^{-1}\tilde m_{v_{1},v_{2}}(x)|$ can be dominated by
\[C_N\biggl(1+\bigl|x_{1} +\lambda^{\frac 12} F_{1}(v_{1},v_{2})\bigl|+\bigl|x_{2}+\la(\lambda \sigma)^{-\frac{1}{2}}F_{2}(v_{1},v_{2})\bigl| + \bigl|x_{3}+\lambda F(v_{1},v_{2})\bigl|\biggl)^{-N},
\]
which implies that $ \|\mathcal{F}^{-1}\tilde m_{v_{1},v_{2}}\|_{L^{1}} \le C,$ hence \eqref{mv1v2}.
\end{proof}

%%%%%%%%%%%%%%%%%%%%%%%%%%%%%%%%%%%%%%%%%%%%%%%%%%%%%%%%%%%%%%%%%%%%%%%%%%%%%%%%%%%%%%%%%%%%%%%%%%%%

 \section{Type A:  Both $\partial_1\Phi$ and $\partial_2\Phi$ vanish along $\gamma_{\tau}$}\label{caseAi}

In this section, we shall prove Theorem \ref{theorem2.1}(a). So assume here that $\tau\in R$ parametrizes a fixed root $\gamma_\tau$ of Type A, and let $\tilde m_\tau$ be the associated Fourier multiplier defined in \eqref{dmiudef}, for which $\tilde{\mathcal{M}}_{\tau}=\mathcal{M}_{\tilde m_\tau}.$ We shall also abbreviate the multiplicity $T_\tau$ of this root in $\Phi$ by $T:=T_\tau,$ and recall  from Remark \ref{TypeT}(a) that $T\ge 2.$

We shall again  use the abbreviation $a:=\ka_2/\ka_1,$ and note that $a\ne 1,$ since we are assuming that $\ka_1\ne \ka_2.$

Since $\chi_1$ is supported in $[1/2,2],$ so that $x_1\sim 1$ on the support of  the amplitude  $\tilde \psi(x)\, \rho_\tau(x) \,\chi_1(x_1)$ of  the oscillatory integral defining $\tilde m_\tau,$  and since 
$|x_2-\tau x_1^a|\lesssim \ve$ (compare \eqref{rhotau}, where $a=\ka_2/\ka_1$), in a first step we perform a dyadic decomposition 
$$
\tilde \psi(x)\, \rho_\tau(x) \,\chi_1(x_1)=\sum\limits_{l\ge l_0} \tilde \chi_1\big(2^{l}(x_{2}-\tau  x_{1}^{a})\big)\tilde \psi(x)\, \chi_1(x_1),
$$
 where $\tilde \chi_1$ is a suitable smooth cut-off function  supported in $[-2,-1]\cup[1,2]$, and where we may assume that the positive integer $l_0$ is sufficiently large by assuming $\epsilon$ to be sufficiently small.
  Accordingly, we decompose
\begin{eqnarray*}
\tilde m_\tau(\xi)&=& \sum_{l \ge l_0} \int_{\mathbb{R}^2} e^{-i\big(\xi_1x_1+\xi_2x_2 + \xi_3\Phi(x)\big)} \tilde\chi_1\big(2^{l}(x_{2}-\tau  x_{1}^{a})\big) \tilde \psi(x) \chi_1(x_1)\, dx \\
&=& \sum_{l \ge l_{0}} 2^{-l}\int_{\mathbb{R}^2} e^{-i\big(\xi_1x_1+\xi_2\tau x_1^{a} +\xi_2 2^{-l} x_2 + \xi_3\Phi(x_1,2^{-l}x_2+\tau  x_1^{a})\big)}  \psi(x_{1},2^{-l}x_{2}) \chi_1(x_1)\tilde\chi_1(x_2)  \, dx,\nonumber
\end{eqnarray*}
for a smooth function $\psi.$ For $0<\sigma\ll 1$ let us put 
$$
\tilde m_{\tau,\sigma}(\xi):= \int_{\mathbb{R}^2} e^{-i\big(\xi_1x_1+\xi_2\tau x_1^{a} +\xi_2 \sigma x_2 + \xi_3\Phi(x_1,\sigma x_2+\tau  x_1^{a})\big)}  \psi(x_1,\sigma x_2) \chi_1(x_1)\tilde\chi_1(x_2)  \, dx.
$$
Then  $\tilde m_\tau(\xi)= \sum_{l \ge l_0} 2^{-l}\tilde m_{\tau,2^{-l}}(\xi),$ hence
\begin{equation}\label{dyadicl}
\|\mathcal M_{\tilde m_{\tau}}\|_{L^{p}\rightarrow L^{p}}\le  \sum_{l \ge l_0} 2^{-l}\|\mathcal M_{\tilde m_{\tau,2^{-l}}}\|_{L^{p}\rightarrow L^{p}}.
\end{equation}

%%%%%%%%%%%%%%%%%%%%%%%%%%%%%%%%%%%%%%%%%%%%%%%%%%%%%%%%%%%%%%%%%%%%%%%%%%%%%%%%%%%%%%%%%%%%%%%%%%%%
 
 \subsection{$L^p$-bounds for $\tilde{\mathcal{M}}_\tau$ when $\tau=0$ }\label{TypeAlambda=0}
 
  If $\tau=0$ and $\sigma=2^{-l},$  then $\Phi(x_1,\sigma x_2+\tau  x_1^{a})=\sigma^{T_0} \Phi_\sigma(x),$ with $\Phi_\sigma$ as defined in Lemma \ref{HphinA}. Let us then look at the {\it re-scaled} Fourier multiplier
  $$
 m_{\tau,\sigma}(\xi):=\tilde m_{\tau,\sigma}(\xi_1, \sigma^{-1}\xi_2,\sigma^{-T_0}\xi_3)= \int_{\mathbb{R}^2} e^{-i\big(\xi_1x_1+\xi_2 x_2 + \xi_3\Phi_\sigma(x)\big)}  \psi(x_1,\sigma x_2) \chi_1(x_1)\tilde\chi_1(x_2)  \, dx,
$$
and note that  this is essentially the Fourier transform of the surface measure of the graph of $\Phi_\sigma,$ multiplied with the density defined by 
the compactly supported amplitude  $\psi(x_1,\sigma x_2) \chi_1(x_1)\tilde\chi_1(x_2).$ Since  $|\mathrm{H}\Phi_\sigma(x)|\sim 1$  by Lemma \ref{HphinA}(ii), this surface has non-vanishing Gaussian curvature. Thus,  again by  \cite{Gr81,F86}, we see that for any $p>3/2$
$$
\|\mathcal M_{m_{\tau,\sigma}}\|_{L^{p}\rightarrow L^{p}}\lesssim 1,
$$
uniformly in $\sigma$ for  $\sigma$ sufficiently small, which we can guarantee for $\sigma=2^{-l}$ with $l\ge l_0$ by choosing $l_0$ sufficiently large.  But $\|\mathcal M_{\tilde m_{\tau,\sigma}}\|_{L^{p}\rightarrow L^{p}}=\|\mathcal M_{ m_{\tau,\sigma}}\|_{L^{p}\rightarrow L^{p}}$ by \eqref{mGLn2}, and thus in combination with \eqref{dyadicl} we see that for $\tau=0$
$$
\| \tilde {\mathcal{M}}_{\tau}\|_{L^p\rightarrow L^p}=\|\mathcal M_{\tilde m_{\tau}}\|_{L^p\rightarrow L^p}<\infty
$$ for every $p>3/2.$

%%%%%%%%%%%%%%%%%%%%%%%%%%%%%%%%%%%%%%%%%%%%%%%%%%%%%%%%%%%%%%%%%%%%%%%%%%%%%%%%%%%%%%%%%%%%%%%%%%%%

\subsection{$L^p$-bounds for $\tilde{\mathcal{M}}_\tau$ when  $\tau \in R^{\prime}$ }
Assume now that $\tau\ne 0.$ By \eqref{PhiTlambda}, we have 
$$
\Phi(x_1,\sigma x_2+\tau  x_1^{a})=\sigma^Tx_2^TW_{\tau}(x_1,\sigma x_2+\tau  x_1^{a}):= \sigma^Tx_2^T \,\widetilde{W}(x_1,\sigma x_2),
$$
so that we may write
$$
\tilde m_{\tau,\sigma}(\xi)= \int_{\mathbb{R}^2} e^{-i\big(\xi_1x_1+\xi_2\tau x_1^{a} +\sigma\xi_2  x_2 +\sigma^T\xi_3  x_2^T \,\widetilde{W}(x_1,\sigma x_2)\big)}  \psi_\sigma(x) \, dx,
$$
with amplitude
$$
\psi_\sigma(x):=\psi(x_1,\sigma x_2) \chi_1(x_1)\tilde\chi_1(x_2),
$$
where we  recall that $\chi_1$ is supported in $[1/2,2]$ and $\tilde \chi_1$ in  $[-2,-1/2]\cup[1/2,2].$

In view of  \eqref{mGLn2}, it will more convenient to estimate the maximal operator associated to the {\it re-scaled }multiplier
$m_{\sigma}(\xi_1,\xi_2,\xi_3):=m_{\sigma}(\xi_1,\xi_2,\sigma^{-T}\xi_3),$ i.e., 
\begin{equation}\label{Admiul}
m_{\sigma}(\xi)= \int_{\mathbb{R}^2} e^{-i\big(\xi_1x_1+\xi_2\tau x_1^{a} +\sigma\xi_2  x_2 +\xi_3  x_2^T \,\widetilde{W}(x_1,\sigma x_2)\big)}  \psi_\sigma(x) \, dx,
\end{equation}
since 
$$
\|\mathcal M_{\tilde m_{\tau,\sigma}}\|_{L^p\rightarrow L^p}=\|\mathcal M_{m_{\sigma}}\|_{L^p\rightarrow L^p}
$$
(note that we have dropped here the index $\tau$ in order to defray the notation).
Observe also that this multiplier is well-defined as well for any sufficiently small $\sigma\in \Bbb R.$

Let $\beta$ be a smooth cut-off function supported in $[1/2, 2]$ such that
\begin{equation}\label{LPDecoposition}
\sum_{j \in \mathbb{Z} }\beta(2^{-j}r)=1 \hspace{0.5cm}\textmd{for every} \hspace{0.5cm} r>0.
\end{equation}
For $j_0\gg 1$, we put $\beta_{0}(r)= \sum_{j < j_0}\beta(2^{-j}r),$ and decompose 
\begin{equation}\label{tmsigdecomp}
m_{\sigma}(\xi)= m_{\sigma}^0(\xi)+\sum\limits_{j\ge j_0}  m_{\sigma}^{2^j}(\xi),
\end{equation}
where
$$
 m_{\sigma}^0(\xi):=\beta_0(|\xi|)  m_{\sigma}(\xi) \quad \text{and} \quad   m_{\sigma}^{2^j}(\xi):=\beta(2^{-j}|\xi|)  m_{\sigma}(\xi).
$$

We first estimate $\mathcal{M}_{ m_{\sigma}^0}:$ Integrations by parts  yield
\[
|\mathcal {F}^{-1}m_{\sigma}^0(y)|=\biggl|\int_{\mathbb{R}^{3}}e^{i y \cdot \xi}\beta_0(|\xi|)  m_{\sigma}(\xi)\, d\xi\biggl| \le C_{N}(1+|y|)^{-N}
\]
for any integer $N$. 

Thus  $|\mathcal{M}_{ m_{\sigma}^0}f(y)| \lesssim  M_{\rm HL} f(y),$ where $M_{\rm HL}$ denotes   Hardy-Littlewood's maximal operator, which is $L^p$-bounded for every $p>1.$ It follows  that  $\|\mathcal{M}_{m_{\sigma}^0}\|_{L^{p} \rightarrow L^{p}} \lesssim 1$ for every $p>1,$ uniformly for $\sigma\ll 1.$ 
 
  \medskip

 We shall thus concentrate on estimating  the maximal operators associated to the multipliers $ m_{\sigma}^{2^j}, \, j\ge j_0.$ To this end, fix a $j\ge j_0,$ and put 
   \begin{equation}\label{defmlasi}
\la:=2^j\gg 1 \quad\text{and} \quad m_{\sigma}^\la(\xi)=\beta(\la^{-1}|\xi|)  m_{\sigma}(\xi).
\end{equation}
Let us  denote the full phase function in (\ref{Admiul}) by
\begin{equation}\label{phixi}
\phi_{\xi}(x) :=\xi_1x_1+\xi_2\tau x_1^{a} +\sigma\xi_2  x_2 +\xi_3  x_2^T \,\widetilde{W}(x_1,\sigma x_2).
\end{equation}

Then we have
\begin{eqnarray*}
\partial_{x_1}\phi_{\xi}(x) &=&  \xi_{1}+ \xi_{2}\tau a x_{1}^{a-1}  + \xi_{3}x_{2}^{T}\widetilde{W}_{1}(x_{1},\sigma x_{2})
\\
\partial_{x_2}\phi_{\xi}(x) &=&   \sigma \xi_{2} + \xi_3\bigl[Tx_{2}^{T-1}\widetilde{W}(x_{1},\sigma x_{2}) + \sigma x_{2}^{T}\widetilde{W}_{2}(x_{1},\sigma x_{2})\bigl].
\end{eqnarray*}

 We notice that  $x_{1} \sim |x_{2}| \sim 1$ on the support of $\psi_\sigma$, $\sigma\ll1,$  and
\[|\widetilde{W}(x_1,\sigma x_2)|\sim 1, \quad \quad |\widetilde{W}_i(x_1,\sigma x_2)| \lesssim 1\quad \text{ for } i =1,2.\]
Thus  $\nabla \phi_{\xi}(x_{1},x_{2}) = (0,0)$ only if $|\xi_{1}| \sim |\xi_{2}|   \gg |\xi_{3}|, $
which then also  implies that
\begin{equation}\label{curv12}
|\partial^2_{x_1}\phi_{\xi}(x)|\sim |\xi_2|\sim \la,  \quad \text{and} \quad |\partial^2_{x_2}\phi_{\xi}(x)|\sim |\xi_3| ,
\end{equation}

since $\sigma\ll 1$ and $a\ne 1.$

More specifically,  we will distinguish  the following  cases: 

\noindent \textbf{Case a.} $|\xi_1|\gg |\xi_2|+|\xi_3|.$ Then 
$|\partial_{x_{1}} \phi_{\xi}(x)| \gtrsim |\xi_1|\sim |\xi|,$ so that 
integrations by parts with respect to the $x_1$-variable  in (\ref{Admiul}) imply that for any $\al\in \Bbb N^3$ and $N\in\Bbb N$,
\begin{equation}\label{1remainderdecay}
\biggl|\partial^{\alpha}_{\xi}\tilde m_{\sigma}^{\la}(\xi)\biggl| \lesssim \frac{C_{N,\alpha}}{(1+ |\xi|)^{N}}.
\end{equation}

\noindent \textbf{Case b.} $|\xi_2|\gg |\xi_1|+|\xi_3|.$  Then $|\partial_{x_1}\phi_{\xi}(x)|  \gtrsim |\xi_2|\sim |\xi|,$ so that again integrations by parts with respect to the $x_1$-variable in \eqref{Admiul}  imply \eqref{1remainderdecay}.

\noindent \textbf{Case c.} $|\xi_3|\gtrsim |\xi_1|+|\xi_2|.$  Assuming that $\sigma$ is sufficiently small, we  see that  $|\partial_{x_2}\phi_{\xi}(x)|  \gtrsim |\xi_3|\sim |\xi|,$ so that here integrations by parts with respect to the $x_2$-variable in \eqref{Admiul}  imply \eqref{1remainderdecay}.

It is now easily seen that what remains is 

\noindent \textbf{Case d.} $|\xi_3|\ll |\xi_1|\sim |\xi_2|.$  In this case, the function   $ \phi_{\xi}$ may  have a critical point.

\smallskip

 We have thus reduced to the situation where  $|\xi_3/\xi_2|\ll 1,$ and where $|\xi_1/\xi_2|$  lies in an interval $[c_1,c_2]$ with $0<c_1<c_2.$ By decomposing this interval into a finite number of sufficiently short subintervals and then applying a suitable scaling in $\xi_1$ for each of these subintervals we can easily reduce to the case where  $[c_1,c_2]=[1/4,1]$ (with a slight modification of the cut-off $\beta(\la^{-1}|\xi|)$ which localizes to $|\xi|\sim \la$ in \eqref{defmlasi} for the contribution of each subinterval).
 
 We then decompose 
\begin{eqnarray}\nonumber
m_{\sigma}^{\la}(\xi) &=& \tilde{\chi}_1\bigl(\frac{\xi_{1}}{\xi_{2}}\bigl) \chi_{0,\epsilon'}\big(\frac{\xi_{3}}{\xi_{2}}\big)\, m_{\sigma}^{\la}(\xi)+ \biggl(1- \tilde{\chi}_1\bigl(\frac{\xi_{1}}{\xi_{2}}\bigl) \chi_{0,\epsilon'}\bigl(\frac{\xi_{3}}{\xi_{2}}\bigl) \biggl)  m_{\sigma}^{\la}(\xi) \\
&=:&  m_{\sigma,1}^{\la}(\xi)+ m_{\sigma,0}^{\la}(\xi), \label{domination1}
\end{eqnarray}
where $\tilde{\chi}_1$ and $\chi_{0,\epsilon'}$ are non-negative smooth cut-off functions,  where  $\supp\tilde{\chi}_1 \subset [-2,-1/2] \cup [1/2,2]$ and $\tilde{\chi}_1(s) =1 $ near  $|s| =1$,   whereas $\chi_{0,\epsilon'}$ is supported in a small $\epsilon'$-neighborhood of 0 such that  $\chi_{0,\epsilon'}\equiv 1$ near 0.

 Then the maximal operator  $ \mathcal{M}_{m_{\sigma,0}^{\la}}$ will be a remainder term. 

Indeed, according to the analysis in Cases a -- c,  estimate \eqref{1remainderdecay} holds true   when $\xi$ is restricted  to the support of
$1- \tilde{\chi}_1\bigl(\frac{\xi_{1}}{\xi_{2}}\bigl) \chi_{0,\epsilon'}\bigl(\frac{\xi_{3}}{\xi_{2}}\bigl) $.
Then integrations by parts with respect to the $\xi$-variables yield that
\[ 
\mathcal{F}^{-1} m_{\sigma,0}^{\la}(y)=\int_{\mathbb{R}^{3}} e^{i y \cdot \xi} \biggl(1-\tilde{\chi}_1\bigl(\frac{\xi_{1}}{\xi_{2}} \bigl)\chi_{0,\epsilon'}\bigl(\frac{\xi_{3}}{\xi_{2}}\bigl) \biggl)   m_{\sigma}(\xi)\beta(\la^{-1}|\xi|)\, d\xi =\mathcal{O}\Big(\lambda^{-N} (1+|y|)^{-N}\Big)
\]
for every  $  N \in \mathbb{N}.$  This shows that 
 $| \mathcal{M}_{m_{\sigma,0}^{\la}}f(y)| \lesssim C_{N}\lambda^{-N} M_{\rm HL}f(y),$ so that 
 \begin{equation}\label{Msi0}
\|\mathcal{M}_{m_{\sigma,0}^{\la}}\|_{L^{p}\rightarrow L^{p}} \lesssim \lambda^{-N}
\end{equation}
for any $p>1$ and any positive integer $N$.

\smallskip

We are thus left with the estimation of $\mathcal{M}_{m_{\sigma,1}^{\la}},$  where the multiplier $m_{\sigma,1}^{\la}$ is defined by the first summand in \eqref{domination1}. Note that here  the $\xi$-variable is restricted to the region described by Case d, i.e.,
\[\lambda \sim |\xi_{1}| \sim |\xi_{2}| \gg |\xi_{3}|, \quad \quad \lambda \gg 1. \]

In the sequel, it will be convenient to put
\[s_1:= \frac{\xi_1}{\xi_2}  \qquad s_3:= \frac{\xi_3}{\xi_2}; \quad\text{note that}\ |s_1|\sim 1,|s_3|\ll 1. \]
Accordingly, we write $\phi_{\xi}(x)=\xi_2 \,\tilde \phi(x, s_1,s_3),$ with 
\begin{equation}\label{phis}
\tilde \phi(x, s_1,s_3)=s_1x_1+\tau x_1^{a} +\sigma x_2 +s_3  x_2^T \,\widetilde{W}(x_1,\sigma x_2),
\end{equation}
so that 
\begin{equation}\label{mlasi1}
m_{\sigma,1}^\la(\xi)= \tilde{\chi}_1\bigl(\frac{\xi_{1}}{\xi_{2}}\bigl) \chi_{0,\epsilon'}\big(\frac{\xi_{3}}{\xi_{2}}\big)\,\beta(\la^{-1}|\xi|)\,\int_{\mathbb{R}^2} e^{-i\xi_2\tilde \phi(x, s_1,s_3)}  \psi_\sigma(x) \, dx.
\end{equation}

\medskip

\noindent {\bf I)  Assume now first that $\sigma\la\ll 1.$ }

Then $|\sigma\xi_2|\ll 1,$ so that we may integrate by parts with respect to $x_2$ in \eqref{Admiul} to gain factors of order 
$(1+|\xi_3|)^{-N}$ for any $N\in \Bbb N.$  Moreover, in view of \eqref{curv12}, we see that with respect to the integration in the variable $x_1,$ the phase $\phi_\xi$ has either a non-degenerate critical point, or we can apply integrations by parts. Thus,  in both cases, we can estimate
$$
|m_{\sigma,1}^{\la}(\xi)|\le C_N\la^{-1/2} (1+|\xi_3|)^{-N}.
$$

Furthermore,  if we put $X(x):=\big(x_1, \sigma x_2+\tau x_1^a,x_2^T \,\widetilde{W}(x_1,\sigma x_2)\big),$ then 
$$
\mathcal{F}^{-1} m_{\sigma,1}^{\la}(y)=c\int_{\mathbb{R}^{2}} \int_{\mathbb{R}^{3}} e^{i\xi\cdot(y-X(x))} \beta(\la^{-1}|\xi|)\tilde{\chi}_1\big(\frac{\xi_{1}}{\xi_{2}} \big)\chi_{0,\epsilon'}\big(\frac{\xi_{3}}{\xi_{2}}\big)\, d\xi\, \psi_\sigma(x) \, dx.
$$
For $|y|\gg 1,$ integrations by parts in $\xi$ then allow to re-write, for any $N\in\Bbb N_{\ge 1},$
$$
\mathcal{F}^{-1} m_{\sigma,1}^{\la}(y)=c\int_{\mathbb{R}^{2}} \int_{\mathbb{R}^{3}} e^{i\xi\cdot(y-X(x))} \beta(\la^{-1}|\xi|)\tilde{\chi}_1\big(\frac{\xi_{1}}{\xi_{2}} \big)\chi_{0,\epsilon'}\big(\frac{\xi_{3}}{\xi_{2}}\big)\, b_N(\xi) d\xi\, \psi_\sigma(x) q_N(y,x)\, dx,
$$
where any  derivative $\partial_\xi^\alpha  b_N(\xi)$ of $b_N$ is bounded by a constant  $C_\alpha,$ and where $q_N(y,x)$ is smooth for $|y|\gg 1$ and $x\in \supp \psi_\sigma$ and satisfies an estimate
$$
|q_N(y,x)|\le C_N|y|^{-N}.
$$
If we then exploit the oscillations in the integrations  in $x_1$ and $x_2$ as before, and note that for $N\ge 2$
$$
\int_{\mathbb{R}^{3}}\Big| \beta(\la^{-1}|\xi|)\tilde{\chi}_1\big(\frac{\xi_{1}}{\xi_{2}} \big)\chi_{0,\epsilon'}\big(\frac{\xi_{3}}{\xi_{2}}\big)\, b_N(\xi) (1+|\xi_3|)^{-N}\Big|d\xi\lesssim \la^2,
$$
we see that for any $N\in\Bbb N$, 
$$
|\mathcal{F}^{-1} m_{\sigma,1}^{\la}(y)|\le C_N \la^2 \,\la^{-1/2} \, (1+|y|)^{-N}.
$$
This implies that
$| \mathcal{M}_{m_{\sigma,1}^{\la}}f(y)| \lesssim C_{N}\lambda^2\, \lambda^{-1/2 } M_{\rm HL}f(y),$ 
hence
$$
\|\mathcal{M}_{m_{\sigma,1}^{\la}}\|_{L^{p}\rightarrow L^{p}} \le C_p  \lambda^{-1/2}\,\la^2= \lambda^{3/2}\qquad \text{for any }\quad p>1.
$$

Furthermore, since  $\|m_{\sigma,1}^{\la}(\xi)\|_\infty \lesssim \la^{-1/2}$ by Plancherel,  and since 
$|\partial_tm_{\sigma,1}^{\la}(t\xi)\big|_{t=1}=\la\, \widetilde {m_{\sigma,1}^{\la}},$   where $\widetilde{ m_{\sigma,1}^{\la}}$ behaves essentially like  $m_{\sigma,1}^{\la},$ classical  methods based on Littlewood-Paley theory and a variant of Sobolev's  embedding theorem (see for instance in the proof of  \cite[Proposition 4.1]{BDIM19})  show that
$$
\|\mathcal{M}_{m_{\sigma,1}^{\la}}\|_{L^{2}\rightarrow L^{2}} \le C \lambda^{1/2}\la^{-1/2}=C.
$$
By interpolation between these two estimates, we obtain that for $1<p<2$ and any $\ve >0,$ 

$$
\|\mathcal{M}_{m_{\sigma,1}^{\la}}\|_{L^{p}\rightarrow L^{p}} \le C_{p,\ve}  \lambda^{\frac 32(\frac 2p-1)+\ve}.
$$
For $p>6/5,$ by choosing $\ve$ sufficiently small, we may then estimate
$$
\sigma \| \mathcal{M}_{m_{\sigma,1}^{\la}}\|_{L^{p}\rightarrow L^{p}} \lesssim 
\sigma \la^{1-\ve}\la^{\frac 3p-\frac 52+2\ve}\le \sigma^{\ve}\la^{-\ve}.
$$
These estimates can be summed over all dyadic $\la=2^j\ge 1$ and all dyadic $\sigma=2^{-l}, l\ge l_0,$ so that we can apply  inequality \eqref{dyadicl}. Note that here even the weaker assumption $p>6/5$ in place of $p>3/2$ suffices. 

\medskip
\noindent {\bf II)  Assume from here on   that $\sigma\la\gtrsim 1.$}

Note first that
$$
\partial_{x_1}\tilde\phi(x,s_1,s_3)=  s_1+ \tau a x_{1}^{a-1}  + s_3x_{2}^{T}\widetilde{W}_{1}(x_{1},\sigma x_{2}).
$$

If $s_1$ and $\tau$ have the same signs, this shows that we can again integrate by parts in $x_1$ and obtain the analogue of estimates \eqref{1remainderdecay} for $m^\la_\sigma.$ Accordingly, by the previous arguments we obtain a remainder term estimate 
$\|\mathcal{M}_{m_{\sigma,1}^{\la}}\|_{L^{p}\rightarrow L^{p}} \lesssim \lambda^{-N}$ for every $p>1$ (compare \eqref{Msi0}).
\smallskip

So assume from here on that $s_1$ and $\tau$ have opposite signs. We may and shall also assume without loss of generality that $\tau<0,$ hence $s_1>0.$

If $s_3=0,$ the function $x_1\mapsto \partial_{x_1}\tilde\phi(x,s_1,s_3)$  has then a unique zero 
$$
x_1^0(s_1):=\big(\frac {-s_1}{a\tau}\big)^{\frac 1{a-1}}>0\qquad \text{ for every } s_1>0.
$$

And,  since 
$|\partial^2_{x_1}\tilde\phi(x,s_1,0)|\sim 1$ for $x_1\sim 1,$ we see by the implicit function theorem that for any sufficiently small $s_3$ and $s_1\sim 1$ the function 
$x_1\mapsto\tilde\phi(x,s_1,s_3)$ has a unique critical point $x_1^c(x_2,s_1,s_3,\sigma)$, i.e., 
$$
\partial_{x_1}\tilde\phi\big((x_1^c(x_2,s_1,s_3,\sigma), x_2),s_1,s_3\big)=0,
$$ 
depending smoothly on $x_2,s_1,s_3,\sigma$ and such that $x_1^c(x_2,s_1,s_3,\sigma)=x_1^0(s_1)$ when $s_3=0.$ 
Moreover,  since $|\partial^2_{x_1}\tilde\phi(x,s_1,s_3)|\sim 1$ and 
 $|\partial_{x_2}\partial_{x_1}\tilde\phi(x,s_1,s_3)|\lesssim |s_3|$ for $x\in \supp \psi_\sigma, \, s_1\sim 1$ and $s_3$ sufficiently small, we have that 
 \begin{equation}\label{d2x1c}
|\partial_{x_2}x_1^c(x_2,s_1,s_3,\sigma)|\lesssim |s_3|.
\end{equation}
Note that  all this applies for  $s_3\in \supp \chi_{0,\epsilon'},$ provided we have chosen $\epsilon'$ sufficiently small, what we shall assume henceforth.

Applying the method of stationary phase to the integration in $x_1,$ keeping in mind that $|\partial^2_{x_1}\tilde\phi(x,s_1,s_3)|\sim 1$  uniformly in $x,s_1,s_2$ for $x$ in the support of the amplitude $\psi_\sigma$ when $s_1\sim 1,|s_3|\ll 1,$ we see that we can then write (with $s_1=\xi_1/\xi_2$ and  $s_3=\xi_3/\xi_2$)
\begin{equation}\label{mlasi12}
 m_{\sigma,1}^{\la}(\xi) =\la^{-\frac 12} \beta(\la^{-1}|\xi|) \,\chi_1\big(\frac{\xi_{1}}{\xi_{2}}\big) \chi_{0,\epsilon'}\big(\frac{\xi_{3}}{\xi_{2}}\big)\,\int_{\mathbb{R}} e^{-i\xi_2 \breve \phi(x_2,s_1,s_3)}  A_{0,\la}(\xi_2; x_2, s_1,s_3)\, \tilde\chi_1(x_2)\,dx_2,
\end{equation}
with the phase 
\begin{eqnarray}\nonumber
\breve \phi(x_2,s_1,s_3)&:= &\sigma  x_2 +s_3  x_2^T \,\widetilde{W}\big( x_1^c(x_2,s_1,s_3,\sigma),\sigma x_2\big)\\
&&\hskip3cm + s_1x_1^c(x_2,s_1,s_3,\sigma)+\tau  x_1^c(x_2,s_1,s_3,\sigma)^{a}. \label{bphi}
\end{eqnarray}
Here, $A_{0,\la}(\xi_2; x_2, s_1,s_3)=(\xi_2/\la)^{-1/2}A_0(\xi_2; x_2, s_1,s_3),$  where $A_0$ is a symbol of order 0 in $\xi_2$ which depends smoothly on the parameters
$(x_2,s_1,s_3),$ with symbol norms uniformly bounded in these parameters. Note also that $|\xi_2/\la|\sim 1.$

We still need to gain more from the integration in $x_2.$ We first consider the cases where $\sigma\gg |s_3|,$ or 
$|s_3|\gg \sigma.$ In the first case we could gain  directly by integrations by parts in $x_2$  in \eqref{mlasi12},  but in the second case this is not so clear.

However, going back to \eqref{mlasi1}, we note that 
$$
\partial_{x_2}\tilde\phi(x,s_1,s_3)=  \sigma  +s_3 T x_2^{T-1} \,\widetilde{W}(x_1,\sigma x_2)+\mathcal{O}(\sigma|s_3|).
$$
Thus,  if $\sigma\gg |s_3|,$ then $|\partial_{x_2}\tilde\phi(x,s_1,s_3)|\sim \sigma,$ and if $|s_3|\gg \sigma, $ then 
$|\partial_{x_2}\tilde\phi(x,s_1,s_3)|\sim |s_3|\gg \sigma.$ 

Thus, given any  $M\in\Bbb N,$ we can integrate by parts $M$- times  in $x_2$ in the oscillatory integral \eqref{mlasi1}, which leads to an oscillatory integral with the same phase, but an amplitude which is of order $\mathcal{O}((\sigma\la)^{-M})$ with respect to any  any $C^k$-norm (actually it is a symbol of order $-M$ in $\sigma \xi_2$). Applying afterwards the method of stationary phase to the $x_1$-integration in the same way as before, we then find that 
\begin{equation}\label{mlaest3}
|m_{\sigma,1}^{\la}(\xi)|\le C_M \la^{-\frac 12}  (\la\sigma)^{-M} \qquad \text{for any}\  M\in\Bbb N.
\end{equation}

We therefore now decompose 
\begin{eqnarray}\nonumber
m_{\sigma,1}^{\la}(\xi) &=& \chi_{1,\ve}\big(\frac{\xi_3}{\sigma \xi_2}\big)m_{\sigma,1}^{\la}(\xi)+ 
 \Big(1-\chi_{1,\ve}\big(\frac{\xi_3}{\sigma \xi_2}\big)\Big) m_{\sigma,1}^{\la}(\xi)\\
&=:&  m_{\sigma,2}^{\la}(\xi)+ m_{\sigma,3}^{\la}(\xi), \label{mladecomp}
\end{eqnarray}
where $\chi_{1,\ve}$ is a suitable smooth cut-off function supported in $[-\ve^{-1}, \ve]\cup [\ve, \ve^{-1}]$ for some sufficiently small $\ve >0.$ 
\smallskip

As for $m_{\sigma,3}^{\la}:$  By choosing $M=1$  in the previous estimate and applying similar arguments as before we see that for any $N\in \Bbb N,$
$$
|\mathcal{F}^{-1} m_{\sigma,3}^{\la}(y)|\le C_N \la^3 \,\la^{-\frac 12}  (\la\sigma)^{-1}  \, (1+|y|)^{-N}.
$$
By comparison  with Hardy-Littlewood's maximal operator, this implies that 
$$
\|\mathcal{M}_{m_{\sigma,3}^{\la}}\|_{L^{p}\rightarrow L^{p}} \le C_{p}  \la^3 \,\la^{-\frac 12}  (\la\sigma)^{-1}\qquad \text{for any }\quad p>1.
$$
Moreover, here we have
$$
\|\mathcal{M}_{m_{\sigma,3}^{\la}}\|_{L^{2}\rightarrow L^{2}} \le C \la^{\frac 12}\la^{-\frac 12}  (\la\sigma)^{-1} =(\la\sigma)^{-1}.
$$
By interpolation, this yields that for $1<p<2$ and $\ve >0$ sufficiently small
$$
\sigma \|\mathcal{M}_{m_{\sigma,3}^{\la}}\|_{L^{p}\rightarrow L^{p}} \le C_{p,\ve} \sigma(\la\sigma)^{-1} \lambda^{\frac 52(\frac 2p-1)+\ve}=\la^{-\ve}\lambda^{\frac 5p-\frac 72+2\ve}\le \sigma^\ve \la^{-\ve},
$$
provided $p>10/7.$ Thus for this $p$-range these estimates can by summed over all dyadic $\la=2^j\ge 1$ and all dyadic $\sigma=2^{-l}, l\ge l_0,$ so that we can apply  inequality \eqref{dyadicl}. Note that here even the weaker assumption $p>10/7$ in place of $p>3/2$ suffices. 

\medskip

\noindent{\bf We are thus left with $m_{\sigma,2}^{\la}.$} It is here useful to consider the {\it re-scaled} multiplier
$$\breve m_{\sigma,2}^{\la}(\xi_1,\xi_2,\xi_3):= m_{\sigma,2}^{\la}(\xi_1,\xi_2, \sigma \xi_3)$$ 
in place of $m_{\sigma,2}^{\la}.$  By \eqref{mlasi12} and \eqref{mladecomp}, for all means and purposes it can be written as 
\begin{equation}\label{mlasi2}
\breve m_{\sigma,2}^{\la}(\xi_1,\xi_2,\xi_3)=\la^{-\frac 12}  \chi_1\big(\la^{-1}\xi_2) \,\chi_1\big(\frac{\xi_{1}}{\xi_{2}}\big)  \chi_{1,\ve}\big(\frac{\xi_3}{ \xi_2}\big)\,\int_{\mathbb{R}} e^{-i\xi_2 \breve \phi(x_2,s_1,\sigma s_3)}  A_{0,\la}\big(\xi_2; x_2, s_1,\sigma s_3\big)\, \tilde\chi_1(x_2)\,dx_2,
\end{equation}
where now $|s_3|=\big|\frac {\xi_3}{\xi_2}\big|\sim 1.$  Moreover,
\begin{eqnarray}\nonumber
\breve \phi(x_2,s_1,\sigma s_3)&= &\sigma  x_2 +\sigma s_3  x_2^T \,\widetilde{W}\big( x_1^c(x_2,s_1,\sigma s_3,\sigma),\sigma x_2\big)\\
&&\hskip3cm + s_1x_1^c(x_2,s_1,\sigma s_3,\sigma)+\tau  x_1^c(x_2,s_1,\sigma s_3,\sigma)^{a}. \label{bphi2}
\end{eqnarray}

Then 
$$\partial_{x_2} \breve \phi(x_2,s_1,\sigma s_3)=\sigma F(x_2,s_1,s_3,\sigma),
$$
 where by \eqref{d2x1c}
$$
 F(x_2,s_1,s_3,\sigma )=1 +s_3 T x_2^{T-1} \widetilde{W}\big( x_1^c(x_2,s_1,\sigma s_3,\sigma),\sigma x_2\big)
 +\mathcal O(|\sigma s_3|).
$$
This function  makes sense for any  sufficiently small $\sigma\in \Bbb R.$  

Assume first that $\sigma=0.$ Then $\sigma s_3=0,$ and we recall that then
$x_1^c(x_2,s_1,\sigma s_3,\sigma)|_{\sigma =0}=x_1^0(s_1)=\big(\frac {-s_1}{a\tau}\big)^{\frac 1{a-1}}.$
Moreover, by  \eqref{d2x1c} we have $\partial_{x_2}x_1^c(x_2,s_1,\sigma s_3,\sigma)= \mathcal{O}(\sigma ).$

Thus, by Taylor expansion of $x_1^c(x_2,s_1,\sigma s_3,\sigma)$ at  $\sigma =0$ we obtain
\begin{equation}\label{taylx1c}
x_1^c(x_2,s_1,\sigma s_3,\sigma)=\big(\frac {-s_1}{a\tau}\big)^{\frac 1{a-1}} +\mathcal{O}(\sigma).
\end{equation}
 Moreover, since $\widetilde{W}(x_1,\sigma x_2)=W_{\tau}(x_1,\sigma x_2+\tau  x_1^{a}), $ we have for  $\sigma=0$  that 
 $\widetilde{W}(x_1,\sigma x_2)|_{\sigma=0}=W_{\tau}(x_1,\tau  x_1^{a}).$ But, $W_{\tau}$ is $\kappa$-homogeneous of degree $1-T\kappa_2,$ so that we have 
 $$ 
 \widetilde{W}(x_1,\sigma x_2)|_{\sigma=0}=cx_1^b, \qquad \text{with} \ b:=\frac {1-T\kappa_2}{\kappa_1}.
$$
Thus, 
$$
\widetilde{W}\big( x_1^c(x_2,s_1,\sigma s_3,\sigma),\sigma x_2\big)\big|_{\sigma=0}=c\big(\frac {-s_1}{a\tau}\big)^{\frac{b}{a-1}},
$$
where $c\ne 0$ . 

Thus
$$
 F(x_2,s_1,s_3,\sigma)|_{\sigma=0}=1 +s_3 T x_2^{T-1}c\big(\frac {-s_1}{a\tau}\big)^{\frac{b}{a-1}}.
 $$

We may assume that the signs of $s_3,c$ and $x_2$ are so that this equation 
$F(x_2,s_1,\sigma s_3)|_{\sigma=0}=0$ has a  unique solution $x_2^0(s_1,s_3),$ namely
$$
x_2^0(s_1,s_3):=\Big(-\big(s_3T c \big(\frac {-s_1}{a\tau}\big)^{\frac{b}{a-1}}\big)^{-1}\Big)^{\frac 1{T-1}},
$$ lying in a compact interval in $(0,\infty)$ for all relevant values of $s_1\sim 1$ and $|s_3|\sim 1.$ Moreover, 
$|\partial_{x_2} F(x_2,s_1,s_3,\sigma)|_{\sigma=0}\sim 1.$ For simplicity, assume also that $x_2>0.$

Since $|s_3|\sim 1,$ we can  now apply the implicit function theorem to show that for $\sigma$ sufficiently small, there is a unique function $x_2^c(s_1,s_3,\sigma)$  such that $x_2^c(s_1,s_3,0)=x_2^0(s_1,s_3)$ and 
$$
 F\big(x_2^c(s_1,s_3,\sigma),s_1,s_3,\sigma\big)\equiv 0
$$
for all $s_1,s_3$ such that $s_1\sim 1$ and $|s_3|\sim 1.$

Then, by applying the method of  stationary phase to the $x_2$-integration, we obtain from \eqref{mlasi2}
\begin{equation}\label{mlasi22}
\breve m_{\sigma,2}^{\la}(\xi_1,\xi_2,\xi_3)=\la^{-\frac 12} (\la\sigma)^{-\frac 12}\, \chi_1\big(\la^{-1}\xi_2) \,\chi_1\big(\frac{\xi_{1}}{\xi_{2}}\big)  \chi_{1,\ve}\big(\frac{\xi_3}{ \xi_2}\big)\,
  \breve A_{0,\la}\big(\xi_2; s_1,\sigma s_3\big) \,e^{-i\xi_2 \breve \phi(x^c_2,s_1,\sigma s_3)}.
\end{equation}
Here, $\breve A_{0,\la}$ is of the form
$$
\breve A_{0,\la}(\xi_2;  s_1,\sigma  s_3)= \big(\frac {\xi_2}{\la}\big)^{-1} \breve A_0(\xi_2;s_1,\sigma s_3),
$$
where $\breve A_0(\xi_2;s_1,\sigma s_3)$ is a symbol of order  0  in $\xi_2$ not depending on $\la$, uniformly in the parameters $ s_1$ and $\sigma s_3$.

Note also that, by \eqref{bphi2} and  \eqref{taylx1c}, the final phase has the form
\begin{eqnarray*}
\xi_2 \breve \phi\big(x_2^c(s_1,s_3,\sigma),s_1,\sigma s_3\big)&=&\xi_2\Big[ s_1\big(\frac {-s_1}{a\tau}\big)^{\frac 1{a-1}}
+\tau \big(\frac {-s_1}{a\tau}\big)^{\frac a{a-1}}+ \sigma \Psi(s_1,s_3)\Big]\\
&=&\xi_2\Big[  (1-a)\tau \big(\frac {-s_1}{a\tau}\big)^{\frac a{a-1}}+ \sigma \Psi(s_1,s_3)\Big],
\end{eqnarray*}
where $\Psi$ is a smooth function.

We finally observe that  {\it re-scaled }multiplier 
$$\mathring{m}_{\sigma,2}^\la(\xi):= \breve m_{\sigma,2}^{\la}(\la \xi),$$
for which according to \eqref{mGLn2} we have $\|\mathcal {M}_{\breve m_{\sigma,2}^{\la}}\|_{L^p\to L^p}=\|\mathcal {M}_{\mathring{m}_{\sigma,2}^\la}\|_{L^p\to L^p},$ 
 then takes the form
\begin{equation}\label{mla3}
\mathring{m}_{\sigma,2}^\la(\xi)= \la^{-\frac 12} (\la\sigma)^{-\frac 12}\,\,e^{-i\la \xi_2\Big[  (1-a)\tau \big(\frac {-s_1}{a\tau}\big)^{\frac a{a-1}}+ \sigma \Psi(s_1,s_3)\Big]}\, 
\chi_1\big(\xi_2) \,\chi_1\big(\frac{\xi_{1}}{\xi_{2}}\big)  \chi_{1,\ve}\big(\frac{\xi_3}{ \xi_2}\big)\,
  \breve A_{0,\la}(\la \xi_2;  s_1,\sigma s_3).
  \end{equation}
  Note that $\breve A_{0,\la}(\la \xi_2;  s_1,\sigma s_3)$ is a smooth function in $\xi_2$ with uniform bounds 
  $|\partial_{\xi_2}^k  \breve A_{0,\la}|\le C_k,$ with constants $C_k$ which do not depend on $\la.$

  We can thus apply Lemma \ref{lemma3.2} (with variables  $(\eta_1,\eta_2,\eta_3)=(\xi_1, \xi_3,\xi_2)$) and conclude  that $\|\mathcal{F}^{-1}\mathring{m}_{\sigma,2}^\la\|_{L^1}\lesssim 1,$ so that 
\begin{equation*}\label{mlasi31}
 \|T_{\mathring{m}_{\sigma,2}^\la}\|_{L^p\to L^p}\lesssim 1 \qquad \text{for every} \quad p\ge 1.
   \end{equation*}
  Moreover, by Plancherel's theorem, we have 
   \begin{equation*}\label{mlasi32}
 \|T_{\mathring{m}_{\sigma,2}^\la}\|_{L^2\to L^2}\lesssim \la^{-\frac 12} (\la\sigma)^{-\frac 12}.
\end{equation*}
  By interpolation, these estimates lead to 
\begin{equation}\label{mlasi3p}
 \|T_{\mathring{m}_{\sigma,2}^\la}\|_{L^p\to L^p}\lesssim   \big(\la^2\sigma\big)^{-(1-\frac 1p)},\qquad 1\le p\le 2.
  \end{equation}

We want to apply Theorem \ref{multipliertomaximal} in order to also bound the maximal operator 
$\mathcal{M}_{\mathring{m}_{\sigma,2}^\la}.$ 

To this end, we first  observe  that  by \eqref{mla3} it is easy  to see that the multiplier 
$$
\dot m_{\sigma,2}^\la:=\frac {d}{dt} \mathring{m}_{\sigma,2}^\la(t\xi)\big|_{t=1},
$$ 
as defined in Theorem \ref{multipliertomaximal}, is of the form $\la  \widetilde{\mathring{m}_{\sigma,2}^\la},$ where 
$\widetilde{\mathring{m}_{\sigma,2}^\la}$ is of the same form \eqref{mla3} as  $\mathring{m}_{\sigma,2}^\la,$ only with a slightly modified  ``amplitude''. 

Thus, it will remain to check condition (ii) in this theorem. For our purposes, it will suffice to prove  that there is a constant $C>0$ such that the following estimate holds for every $r>0:$
 \begin{equation}\label{kernelA2}
 \int_{|x|>r}\bigl|\mathcal{F}^{-1}{\mathring{m}_{\sigma,2}^\la}(x)\bigl|dx \lesssim \lambda^{C}(1+r)^{-1}.
  \end{equation}
  To prove this, note that from  \eqref{mla3} we can easily derive that  for each multi-index $\alpha\in\Bbb N^3,$  
\[
\bigl|\partial ^{\alpha}\mathring{m}_{\sigma,2}^\la(\xi)\bigl| \lesssim \lambda^{|\alpha|}.
\]
Estimate \eqref{kernelA2} follows then  immediately from  Lemma \ref{generalkernelestimate}.

\smallskip
We are now in a position to apply  Theorem \ref{multipliertomaximal}. In view of   \eqref{mlasi3p},  this theorem   implies  that for $1\le p\le 2$ and  any small $\delta>0,$ 
 $$
 \|\mathcal {M}_{\mathring{m}_{\sigma,2}^\la}\|_{L^p\to L^p}\lesssim  \la^{\frac 1p} \big(\la^2\sigma\big)^{-(1-\frac 1p)}
 \big((\la^2\sigma)^{\delta}+\log\la\big),
$$
  hence 
  \begin{equation}\label{Mlasi3p}
 \sigma \|\mathcal {M}_{\mathring{m}_{\sigma,2}^\la}\|_{L^p\to L^p}\lesssim  \sigma^{\frac 1p+\delta}\la^{-(2-\frac 3p)+2\delta}.  \end{equation}
 For any $p>3/2$ sufficiently close to $3/2$ these estimates can  be summed over all dyadic $\la=2^j\ge 1$ and all dyadic $\sigma=2^{-l}, l\ge l_0.$ 
 
 This will then also complete the proof of Theorem \ref{theorem2.1}(a)  by interpolation with the trivial $L^\infty$-bounds.
 \qed

\medskip

%%%%%%%%%%%%%%%%%%%%%%%%%%%%%%%%%%%%%%%%%%%%%%%%%%%%%%%%%%%%%%%%%%%%%%%%%%%%%%%%%%%%%%%%%%%%%%%%%%%%

\section{Type $\mathrm{B}$: Both $\partial_1\Phi$ and $\partial_2\Phi$  do not vanish along $\gamma_{\tau}$}\label{caseB}
In this section, we shall prove Theorem \ref{theorem2.1}(b). So assume here that $\tau\in R$ parametrizes a fixed real root $\gamma_\tau$ of Type $\mathrm{B}$, and let $\tilde m_\tau$ be the associated Fourier multiplier defined in \eqref{dmiudef}, for which $\tilde{\mathcal{M}}_{\tau}=\mathcal{M}_{\tilde m_\tau},$ i.e., 
$$
 \tilde m_\tau(\xi)=\int_{\mathbb{R}^{2}}e^{-i\bigl(\xi_1 x_1+\xi_2 x_2+\xi_3\Phi(x)\bigl)} \tilde \psi(x)\, \rho_\tau(x) \,\chi_1(x_1)\, dx,
$$
where we recall from \eqref{rhotau} that 
$\rho_\tau(x):=\chi_0\Big(\frac{x_2-\tau x_1^{a}}{\epsilon x_1^{a}}\Big).$
We shall again use the abbreviations such as  $a=\ka_2/\ka_1$, etc.  used in the previous section.

%%%%%%%%%%%%%%%%%%%%%%%%%%%%%%%%%%%%%%%%%%%%%%%%%%%%%%%%%%%%%%%%%%%%%%%%%%%%%%%%%%%%%%%%%%%%%%%%%%%%

\subsection{Reduction to FIO-cone multipliers}\label{FIOcone}

In a first step, we perform here the  analogue of the non-homogenous dyadic decomposition \eqref{tmsigdecomp}
directly with $\tilde m_\tau,$ i.e.,
$$
 \tilde m_{\tau}(\xi)=  \tilde m_{\tau}^0(\xi)+\sum\limits_{j\ge j_0} \tilde m_{\tau}^{2^j}(\xi),
$$
where
$$
\tilde m_{\tau}^0(\xi):=\beta_0(|\xi|) \tilde m_{\tau}(\xi) \quad \text{and} \quad   \tilde m_{\tau}^{2^j}(\xi):=\beta(2^{-j}|\xi|)  \tilde m_{\tau}(\xi).
$$
Here, $\beta$ and $\beta_0$ are a smooth cut-off functions supported in $[1/2, 2], $ respectively in $[0, 2^{j_0}].$ 

The maximal operator  $\mathcal{M}_{ \tilde m_{\tau}^0}$ can be estimated in the same way as $\mathcal{M}_{ m_{\sigma}^0}$ in the previous section, so that $\|\mathcal{M}_{ \tilde m_{\tau}^0}\|_{L^{p} \rightarrow L^{p}} \lesssim 1$ for every $p>1.$ 

\smallskip
We shall therefore concentrate  from now on on the maximal operators $\mathcal{M}_{ \tilde m_{\tau}^\la},$ where 
$$
\la:=2^j\gg 1 \quad\text{and} \quad \tilde m_{\tau}^\la(\xi)=\beta(\la^{-1}|\xi|) \,\tilde m_{\tau}(\xi).
$$
The full phase $\phi_\xi$ is here simply $\phi_{\xi}(x) :=\xi_1x_1+\xi_2 x_2+\xi_3\Phi(x),$ and we have 
$$
\partial_{x_1}\phi_{\xi}(x) =  \xi_1+\xi_3\Phi_1(x), \qquad \partial_{x_2}\phi_{\xi}(x) =  \xi_2+\xi_3\Phi_2(x),
$$
where we may assume that on the  support of $\rho_\tau(x) \,\chi_1(x_1)$ we have $|\Phi_1(x)|\sim 1\sim |\Phi_2(x)|,$ since the root $\gamma_\tau$ is of Type $\mathrm{B}.$

In a first step, we shall show that for all $\alpha\in \Bbb N^3$ and $N \in \mathbb{N},$
\begin{equation}\label{Brapiddecay}
|\partial_{\xi}^{\alpha}\tilde m_\tau^\la(\xi)| \lesssim \frac{C_{N,\alpha}}{(1+|\xi|)^{N}}=\mathcal {O}(\la^{-N}),
\end{equation}
unless 
\begin{equation}\label{allsim}
|\xi_1|\sim |\xi_2|\sim |\xi_3|.
\end{equation}
\smallskip

To this end, let $\xi=(\xi_1,\xi_2,\xi_3)\in\Bbb R^3$ with $|\xi|\sim \la$ be given. Choose then a permutation $(i_1,i_2,i_3)$ of $(1,2,3)$ such that 
$$
|\xi_{i_1}|\le |\xi_{i_2}|\le |\xi_{i_3}|, 
$$  so that in particular $|\xi|\sim |\xi_{i_3}|.$ 

Assume  first that  $i_3=1.$ If $i_1=3,$ and if we assume that $|\xi_{i_3}|\gg |\xi_{i_1}|,$ then $|\partial_{x_{i_3}}\phi_{\xi}(x)| \sim |\xi_{i_3}|\sim |\xi|,$ and thus \eqref{Brapiddecay} follows by integrations by parts in $x_{i_3}.$ We are the reduced to the case where $|\xi_{i_1}|\le |\xi_{i_2}|\le |\xi_{i_3}|\lesssim |\xi_{i_1}|,$ i.e., to \eqref{allsim}.

Next, if $i_3=1$ and $i_1=2,$  then $i_2=3,$ and  as before integrations by parts in $x_{i_3}$  show  \eqref{Brapiddecay}  holds unless 
$|\xi_{i_3}|\sim |\xi_{i_2}|.$ In the latter case, assume we had  $|\xi_{i_1}|\ll |\xi_{i_2}|.$ Then $|\partial_{x_{i_1}}\phi_{\xi}(x)| \sim |\xi_{i_2}\Phi_{i_1}|\sim |\xi|,$ so \eqref{Brapiddecay}  would follow  by integrations by parts in $x_{i_2},$ since $\Phi_{i_1}=\Phi_2\ne 0.$ So, again we are reduced to \eqref{allsim}.
\smallskip

The case where $i_3=2$ can be handled in the same way, by interchanging the roles of $x_1$ and $x_2.$
\smallskip

Assume finally that $i_3=3.$ Here integrations by parts in $\xi_{i_1}$ shows that \eqref{Brapiddecay}  would hold for 
$|\xi_{i_1}|\ll |\xi_{i_3}|,$ so we may reduce to the case where $|\xi_{i_3}|\lesssim  |\xi_{i_1}|,$ in which \eqref{allsim} again holds.
\medskip 

Assume now that \eqref{allsim} holds.  We shall here make use of the results in Lemma \ref{Legendret}. As in this lemma, we shall assume without loss of generality that $ \Phi_2<0$ along $\gamma_\tau,$ i.e., 
 \begin{equation}\label{Phi2le0}
 \Phi_{2}(x_{1},\tau  x^{a}_{1})<0\qquad \text{for all} \quad x_1>0
 \end{equation}
 (we may even assume that $\Phi_2(x)<0$ for all $x$ in the support of $\rho_\tau(x) \,\chi_1(x_1)$ by choosing $\ve$ in the definition of $\rho_\tau$ sufficiently small). 
 
 If then $\xi_2/\xi_3<0,$ we would again find that $|\partial_{x_2}\phi_{\xi}(x)| =  |\xi_2+\xi_3\Phi_2(x)|\sim |\xi|,$ so that  integrations by parts in $x_2$  show that again  \eqref{Brapiddecay} holds, which allows  us to reduce to the case where $\xi_2/\xi_3>0.$
 
By means of similar arguments as in Section \ref{caseAi}, we now decompose 
\begin{eqnarray}\nonumber
\tilde m_{\tau}^{\la}(\xi) &=& \tilde{\chi}_1\bigl(\frac{\xi_1}{\xi_3}\bigl) \chi_1\big(\frac{\xi_2}{\xi_3}\big)\, \tilde m_{\tau}^{\la}(\xi)+ \biggl(1- \tilde{\chi}_1\bigl(\frac{\xi_1}{\xi_3}\bigl) \chi_1\big(\frac{\xi_2}{\xi_3}\big)\biggl)  \, \tilde m_{\tau}^{\la}(\xi) \\
&=:&  \tilde m_{\tau,1}^{\la}(\xi)+ \tilde m_{\tau,0}^{\la}(\xi), \label{domination2}
\end{eqnarray}
where $\tilde{\chi}_1$ and $\chi_1$ are non-negative smooth cut-off functions,  where  $\supp\tilde{\chi}_1 \subset [-2,-1/2] \cup [1/2,2]$ and $\tilde{\chi}_1(t) =1 $ near  $|t| =1$,   whereas $\chi_1$ is supported in $[1/2,2]$ and equals 1 near 1.

By means of \eqref{Brapiddecay}, we can estimate the maximal operator  $ \mathcal{M}_{\tilde m_{\tau,0}^{\la}}$  in the same way as we did this for $\mathcal{M}_{m_{\sigma,0}^{\la}}$ in \eqref {Msi0}, i.e., 
$$
\|\mathcal{M}_{\tilde m_{\tau,0}^{\la}}\|_{L^{p}\rightarrow L^{p}} \lesssim \lambda^{-N}
$$
for any $p>1$ and any positive integer $N$.

\smallskip

We shall therefore from here on concentrate on the maximal operator $ \mathcal{M}_{\tilde m_{\tau,1}^{\la}},$ keeping in mind  that $\tilde m_{\tau,1}^{\la}$ is supported where 
$\lambda \sim |\xi_{1}| \sim |\xi_{2}| \sim |\xi_{3}|$ and $\xi_{2}/\xi_{3} > 0. $

\smallskip 
From here on, it will be more convenient to work with the {\it re-scaled} multiplier 
$$
m_{\tau,1}^{\la}(\xi):= \tilde m_{\tau,1}^{\la}(\la \xi).
$$
Using the short-hand notation
\[s_1:= \frac{\xi_1}{\xi_3},  \ s_2:= \frac{\xi_2}{\xi_3}, \ s:=(s_1,s_2)
 \quad\text{(note that}\ |s_1|\sim 1\sim s_2),
  \]
we may assume for simplicity that this multiplier then takes the form
\begin{equation}\label{mla1B}
 m_{\tau,1}^{\la}(\xi)=\tilde{\chi}_1(s_1)\chi_1(s_2)\tilde {\tilde{\chi}}_1(\xi_3) \, \int_{\mathbb{R}^{2}}
 e^{-i\la\xi_3\big(s_1 x_1+s_2 x_2+\Phi(x)\big)}\,  \tilde \psi(x)\,\rho_\tau(x)\,\chi_1(x_1)\, dx,
\end{equation}
wherein  also $|\xi_3|\sim 1.$ We shall also use the abbreviation
$$
\phi_{s}(x):=s_1 x_1+s_2 x_2+\Phi(x).
$$

Let  $x_0=(x_{0,1},x_{0,2})$ be the unique solution in $\gamma_\tau$ of the equation \eqref{x10} (cf. Lemma \ref{Legendret}), i.e., 
$$
x\in \gamma_\tau \ \text{and}\  \partial_{x_2}\big( x_2+\Phi(x)\big)=1+\Phi_2(x)=0, \qquad  \text{if and only if} \quad x=x_0.$$
This shows that for $s_2=1$ there is exactly one critical  point of the phase $\phi_s$ along $\gamma_\tau,$ namely $x_0.$

\smallskip For general  $s_2>0$ of size $s_2\sim 1,$ we can exploit  that $\Phi$ is $\ka$-homogeneous of degree $1$ to reduce to the case $s_2=1:$  
If we put  $r(s_2):=s_2^\frac{1}{1-\ka_2},$ homogeneity  implies that 
$$
s_2\big(r(s_2)^{\ka_2}x_2\big)+\Phi(\delta_{r(s_2)} (x))=s_2^\frac{1}{1-\ka_2}\big(x_2+\Phi(x)),
$$
where  $x\in \gamma_\tau$ if and only if  $\delta_{r(s_2)} (x)\in \gamma_\tau.$ 
\smallskip

We therefore change variables $x\mapsto \delta_{s_2^\frac{1}{1-\ka_2}}(x)$ in the integration in \eqref{mla1B}, so that
\begin{equation}\label{mla1B2}
 m_{\tau,1}^{\la}(\xi)=\tilde{\chi}_1(s_1)\chi_1(s_2)\tilde {\tilde{\chi}}_1(\xi_3) \, \int_{\mathbb{R}^{2}}e^{-i\la \xi_3\big[(s_1s_2^{\frac{\ka_1}{1-\ka_2}} x_1+s_2^{\frac{1}{1-\ka_2}}\big(x_2+\Phi(x)\big)\big]} \, \psi_{s}(x)\,\rho_\tau(x)\,\chi_{1,s_2}(x_1)\, dx_2\, dx_1,
\end{equation}
where 
$$
\psi_{s}(x):=\tilde\psi\big(\delta_{s_2^\frac{1}{1-\ka_2}}(x)\big), \quad \chi_{1,s_2}(x_1):=\chi_1\big(s_2^{\frac{\ka_1}{1-\ka_2}}x_1).
$$
We may then even assume that  $x_0$ lies in the support of $ \rho_\tau(x)\,\chi_{1,s_2}(x_1).$ Otherwise, we would have 
$|\partial_{x_2}\big( x_2+\Phi(x)\big)|\gtrsim 1$ along $\ga_\tau$ when $x_1\sim 1,$ and the same type of estimate would even hold on the support of  $\rho_\tau(x)\,\chi_{1,s_2}(x_1),$ provided we choose $\ve $ sufficiently small, so that integrations by parts in $x_2$ would again show that the multiplier $m_{\tau,1}^{\la}$ only contributes an error term.

\smallskip
Then, by Lemma \ref{Legendret}, there is a unique, analytic function $x_1\mapsto x_2^c(x_1)$  which is defined in a neighborhood 
$I_\epsilon:=(x_{0,1}-\epsilon, x_{0,1}+\epsilon), \epsilon>0,$ of $x_{0,1}$  on which $x_1>0,$  such that $x_2^c(x_{0,1})=x_{0,2}$ and 
$$
\partial_{x_2}\big(x_2+\Phi(x_1,x_2)\big)\vert_{x_2=x_2^c(x_1)}=	1+\Phi_2(x_1,x_2^c(x_1))= 0\qquad \text{for all} \quad x_1\in I_\epsilon.
$$

We therefore choose a smooth cut-off function $x_1\mapsto \chi_0(\frac{x_1-x_{0,1}}\epsilon)$ supported in $I_{2\epsilon}$ and identically 1 on $I_\epsilon,$  and decompose
$$
m_{\tau,1}^{\la}=m_{\tau,2}^{\la}+m_{\tau,3}^{\la},
$$
where $m_{\tau,2}^{\la}$ is defined as  $m_{\tau,1}^{\la},$ only with the amplitude $\psi_{s}(x)\,\rho_\tau(x)\,\chi_{1,s_2}(x_1)$ replaced by the amplitude  $\psi_{s}(x)\,\rho_\tau(x)\,\chi_{1,s_2}(x_1)\chi_0(\frac{x_1-x_{0,1}}\epsilon).$

 Note that for $x\in \gamma_\tau$ with $x_1\ne x_{0,1}$ we have that $\partial_{x_2}\big(x_2+\Phi(x)\big)=1+\Phi_2(x)\ne 0.$  This shows that  in the oscillatory integral defining $m_{\tau,3}^{\la}(\xi),$ in which $|x_1-x_{0,1}|\ge \epsilon,$  we can integrate by parts in $x_2$ (provided again we choose $\ve$ sufficiently small), and see again that the multiplier $m_{\tau,3}^{\la}$ only provides an error term satisfying estimates of the form \eqref{Brapiddecay}.
\smallskip

We shall therefore from here on concentrate on the multiplier $ m_{\tau,2}^{\la}.$ 
Recall  also from  \eqref{phi22ne0} that then  $|\Phi_{22}|\sim 1 $ in the integral defining  $ m_{\tau,2}^{\la}.$ We may then  apply the method of stationary phase to the $x_2$-integration in this integral and  conclude that 
\begin{eqnarray}\nonumber
 m_{\tau,2}^{\la}(\xi)&=&\la^{-\frac 12} \tilde{\chi}_1(s_1)\chi_1(s_2)\tilde {\tilde{\chi}}_1(\xi_3) \, \int_{\mathbb{R}}\, e^{-i\la \xi_3\big[(s_1s_2^{\frac{\ka_1}{1-\ka_2}} x_1+s_2^{\frac{1}{1-\ka_2}}\Psi(x_1)\big]} \\
&&\hskip6cm  \times A_\la(\xi_3,s_2;x_1)\chi_{1,s_2}(x_1)\chi_0(\frac{x_1-x_{0,1}}\epsilon)\, dx_1,\label{mla1B3}
\end{eqnarray}
where 
$$
\Psi(x_1) := x^{c}_{2}(x_{1}) + \Phi(x_{1},x^{c}_{2}(x_{1})), 
$$
and where 
$$A_\la(\xi_3,s_2;x_1)=\Big(\xi_3 s_2^{\frac{1}{1-\ka_2}}\Big)^{-\frac 12} A_0(\la\xi_3s_2^{\frac{1}{1-\ka_2}};x_1),$$
with a symbol $A_0$ of order 0 in $\la\xi_3s_2^{\frac{1}{1-\ka_2}}$  depending smoothly on $x_1$, so that 
  $A_\la(\xi_3,s_2;x_1)$ is smooth and every partial derivative of order $k$ of $A_\la$ is uniformly bounded by a constant $C_k$ not depending on $\la$ for $|s_2|\sim1\sim|\xi_3|.$
Moreover, by Lemma \ref{Legendret}, 
\begin{equation*}%\label{Legendra}
	\Psi(x_{1}) = \Psi(x_{0,1}) + \Psi^{\prime}(x_{0,1})(x_{1}-x_{0,1}) + (x_{1}-x_{0,1})^{M}G(x_{1}), 
\end{equation*}
where $\Psi^{\prime}(x_{0,1})\ne 0$, since $\ga_\tau$ is of Type B,  $G$ is an analytic function  with $G(x_{0,1}) \neq 0,$ and  where $M=M_\tau:=N_\tau+2\ge 3,$ if $N_\tau$ denotes the multiplicity of the root $\gamma_\tau$.  

Applying the coordinate change $y=x_1-x_{0,1},$ this can be re-written as
\begin{eqnarray*}\nonumber
 m_{\tau,2}^{\la}(\xi)&=&\la^{-\frac 12} \tilde{\chi}_1(s_1)\chi_1(s_2)\tilde {\tilde{\chi}}_1(\xi_3) \, \int_{\mathbb{R}}\, e^{-i\la \xi_3\big[(s_1s_2^{\frac{\ka_1}{1-\ka_2}} (x_{0,1}+y)+s_2^{\frac{1}{1-\ka_2}}\tilde \Psi(y)\big]} \\
&&\hskip6cm  \times \tilde A_\la(\xi_3,s_2;y) \chi_0(\frac{y}\epsilon)\, dy,
\end{eqnarray*}
where $\tilde A_\la(\xi_3,s_2;y):=\chi_{1,s_2}(y+x_{0,1})\,A_\la(\xi_3,s_2;x_{1,0}+y),$ and 
$$
\tilde \Psi(y) = \Psi(x_{0,1}) + \Psi^{\prime}(x_{0,1}) y + y^{M}\tilde G(y), 
$$
with $\tilde G(0)\ne 0$ and $\Psi^{\prime}(x_{0,1})\ne 0.$
Thus 
\begin{equation}\label{mla1B6}
m_{\tau,2}^{\la}(\xi)=\la^{-\frac 12} \tilde{\chi}_1(s_1)\chi_1(s_2)\tilde {\tilde{\chi}}_1(\xi_3)\, e^{-i\la\xi_3\Big[s_1 s_2^{\frac{\ka_1}{1-\ka_2}}x_{0,1}+s_2^{\frac{1}{1-\ka_2}} \Psi(x_{0,1})\Big]} \times  J(\la,\xi_3,s),
\end{equation}
where 
$$
J(\la,\xi_3,s): =\int_{\mathbb{R}}\, e^{-i\la \xi_3s_2^{\frac{1}{1-\ka_2}}\Big[\big(s_1s_2^{\frac{\ka_1-1}{1-\ka_2}}+\Psi'(x_{0,1})\big)y+y^M\tilde G(y)\Big]}\, 
\tilde A_\la(\xi_3,s_2;y)\chi_0(\frac{y}\epsilon)\, dy.
$$

From here on, we can closely follow the steps of  proof of Theorem 9.1 in \cite{BDIM25}, by adapting in each of its steps  the arguments to the phase functions arising in our context.  

 We first observe that
$J(\la,\xi_3,s)$ is an ``Airy type integral" as   studied in  \cite[Lemma 2.2]{IM16} (compare also  \cite{BDIM25}).
Put 
\[u:=-\Big(s_{1}s_{2}^{\frac{\kappa_{1}-1}{1-\kappa_{2}}}+ \Psi^{\prime}(x_{0,1})\Big). \]
Then $J(\la,\xi_3,s)$ can be re-written by
  \begin{equation}\label{Iju}
  J(\la,\xi_3,s) =  \int_{\mathbb{R}}\, e^{-i\la \xi_3s_2^{\frac{1}{1-\ka_2}}\big[y^M\tilde G(y)-uy\big]}\, 
\tilde A_\la(\xi_3,s_2;y)\chi_0(\frac{y}\epsilon)\, dy.
 \end{equation}
Note that $\tilde A_\la(\xi_3,s_2;y)$ is independent of   $s_1$.  By  \cite[Lemma 2.2]{IM16} (taking into account also the slight correction of this lemma outlined in \cite{BDIM25}), we then obtain the following:

  \smallskip

\textbf{(a)} If $\lambda^{\frac{M-1}{M}}|u| \lesssim 1$, then
\[J(\la,\xi_3,s) = \lambda^{-\frac{1}{M}}g\biggl(\lambda^{\frac{M-1}{M}}u,\xi_3, s_{2}\biggl),\]
 where $g(v,\xi_3,s_2)$ is a smooth function whose derivatives of any fixed order are uniformly bounded for $|v| \lesssim 1$ and  $|\xi_3| \sim|s_2| \sim 1 $.

\smallskip

\textbf{(b)} If $\lambda^{\frac{M-1}{M}}|u| \gg 1$, let us first assume  that $M$ is odd and  that $u $ and $\tilde G$ have the same sign. Then 
\begin{eqnarray*}
J(\la,\xi_3,s)&=&\lambda^{-\frac{1}{2}} |u|^{-\frac{M-2}{2M-2}} \chi_{0} \bigl(\frac{u}{\epsilon} \bigl) a_{+}\biggl(\lambda \xi_{3}s_{2}^{\frac{1}{1-\kappa_{2}}} |u|^{\frac{M}{M-1}}, |u|^{\frac{1}{M-1}},\xi_3, s_{2} \biggl)e^{-i\lambda \xi_3s_{2}^{\frac{1}{1-\kappa_{2}}}|u|^{\frac{M}{M-1}}q_{+}(|u|^{\frac{1}{M-1}})}  \nonumber\\
&+&\lambda^{-\frac{1}{2}} |u|^{-\frac{M-2}{2M-2}} \chi_{0} \bigl(\frac{u}{\epsilon} \bigl) a_{-}\biggl(\lambda  
\xi_3s_{2}^{\frac{1}{1-\kappa_{2}}} |u|^{\frac{M}{M-1}}, |u|^{\frac{1}{M-1}}, \xi_3,s_{2} \biggl)e^{-i\lambda  \xi_3s_{2}^{\frac{1}{1-\kappa_{2}}}|u|^{\frac{M}{M-1}}q_{-}(|u|^{\frac{1}{M-1}})} \\
&+& \lambda^{-1} \biggl|\xi_{3}s_{2}^{\frac{1}{1-\kappa_{2}}} u \biggl|^{-1} E\biggl(\lambda s_{3}s_{2}^{\frac{1}{1-\kappa_{2}}} |u|^{\frac{M}{M-1}}, |u|^{\frac{1}{M-1}},\xi_3, s_{2} \biggl),  \nonumber
\end{eqnarray*}
where, $q_{\pm}$   are smooth (here even analytic) functions and  non-vanishing at the origin, and $a_{\pm}$   are smooth functions which are symbols of order $0$ with respect to $\mu = \lambda \xi_3s_{2}^{\frac{1}{1-\kappa_{2}}} |u|^{\frac{M}{M-1}}$ so that 
$$
\big|\partial^{\alpha}_{\mu}\partial^{\beta}_{v} \partial^{\gamma}_{\xi_3}\partial^{\rho}_{s_{2}}a_{\pm}(\mu,v,\xi_3,s_{2})\big| \le C_{\alpha, \beta, \gamma,\rho}\, |\mu|^{-\alpha}, \quad \quad \forall \alpha, \beta, \gamma,\rho \in  \mathbb{N}.
$$
Moreover, $E$ is smooth and satisfies
$$
|\partial^{\alpha}_{\mu}\partial^{\beta}_{v} \partial^{\gamma}_{\xi_3}\partial^{\rho}_{s_{2}}E(\mu,v,\xi_3,s_{2})| \le C_{N, \alpha, \beta, \gamma,\rho}|v|^{-\beta}\,  |\mu|^{-N}, \quad \quad \forall N, \alpha, \beta, \gamma,\rho\in  \mathbb{N}.
$$

If $M$ is odd and  $u $ and $\tilde G$ have the same sign, then the same formula remains valid, even with $a_{+}\equiv 0$, $a_{-} \equiv0$.

Finally, if $M$ is even, we do have a similar result but without the presence of the term containing $a_{-}$.

\smallskip

We can simplify the notation by setting
$$
A_\la^\pm(|u|, \xi_3,s_2):=a_{\pm}\biggl(\lambda \xi_{3}s_{2}^{\frac{1}{1-\kappa_{2}}} |u|^{\frac{M}{M-1}}, |u|^{\frac{1}{M-1}},\xi_3, s_{2} \biggl),
$$
and
$$
E_\la(|u|, \xi_3,s_2):=\biggl|\xi_{3}s_{2}^{\frac{1}{1-\kappa_{2}}}\biggl|^{-1}E\biggl(\lambda s_{3}s_{2}^{\frac{1}{1-\kappa_{2}}} |u|^{\frac{M}{M-1}}, |u|^{\frac{1}{M-1}},\xi_3, s_{2} \biggl).
$$

Note that these functions  then satisfy estimates of the following form, uniformly in $\la\gg 1:$
\begin{equation}\label{symbol1}
\big|\partial^{\alpha}_{|u|}\partial^{\beta}_{\xi_3}\partial^{\gamma}_{s_2}  A_\la^\pm(|u|, \xi_3,s_2)\big|
\le C_{\alpha, \beta, \gamma}\, |u|^{-\alpha}, \quad \quad \forall \alpha, \beta, \gamma\in  \mathbb{N}
\end{equation}
and 
\begin{equation}\label{symbol2}
\big|\partial^{\alpha}_{|u|}\partial^{\beta}_{\xi_3}\partial^{\gamma}_{s_2}  E_\la(|u|, \xi_3,s_2)\big|
\le C_{N,\alpha, \beta, \gamma}\, \big(\la |u|^{\frac{M}{M-1}}\big)^{-N}  |u|^{-\alpha}, \quad \quad \forall  N, \alpha, \beta, \gamma\in  \mathbb{N},
\end{equation}
and we may re-write \eqref{Iju} as 
\begin{eqnarray}
J(\la,\xi_3,s)
&=&\lambda^{-\frac{1}{2}} |u|^{-\frac{M-2}{2M-2}} \chi_{0} \bigl(\frac{u}{\epsilon} \bigl) \,A^+_\la( |u|, \xi_3,s_2)\, e^{-i\lambda \xi_{3}s_{2}^{\frac{1}{1-\kappa_{2}}}|u|^{\frac{M}{M-1}}q_{+}(|u|^{\frac{1}{M-1}})}  \nonumber\\
&+&\lambda^{-\frac{1}{2}} |u|^{-\frac{M-2}{2M-2}} \chi_{0} \bigl(\frac{u}{\epsilon} \bigl) \, A^{-}_\la(|u|, \xi_3,s_2) \,e^{-i\lambda  \xi_{3}s_{2}^{\frac{1}{1-\kappa_{2}}}|u|^{\frac{M}{M-1}}q_{-}(|u|^{\frac{1}{M-1}})} \label{Iju2}\\
&+&\lambda^{-1}|u|^{-1} E_\la(|u|, \xi_3,s_2).  \nonumber
\end{eqnarray}
In view of the singularities of these terms  in $u$ at $u=0$ we next perform a dyadic decomposition in $u.$ Actually, it will  be more convenient to do this by means of the simpler expression 
\begin{equation}\label{formulaF}
	F(\xi)=\tilde F(s):= s_1+s_2^{\frac{1-\kappa_{1}}{1-\kappa_{2}}}\Psi^{\prime}(x_{0,1})=\frac{\xi_{1}}{\xi_{3}} + \Psi^{\prime}(x_{0,1}) \big(\frac{\xi_{2}}{\xi_{3}}\big)^{\frac{1-\kappa_{1}}{1-\kappa_{2}}}
	\end{equation}
in place of $u.$  Note that indeed 
\begin{equation}\label{usimF}
u =-s_2^{\frac{\ka_1-1}{1-\ka_2}}F(\xi),
\end{equation}
so that  $|u|\sim |F(\xi)|,$ since $s_2\sim 1.$ The function $F$ will play a crucial role in our analysis - we shall come back to this soon.

In view of  the above results in (a) and (b), by choosing $k(\la)\in\Bbb N$ suitably so that $2^{k(\la)}\sim \la^{\frac {M-1}M},$ we decompose
\begin{equation}\label{decomposeIju}
J(\la,\xi_3,s)=J_0(\la,\xi_3,s)
+ \sum^{k(\lambda)}_{k=1} J^{+}_{k}(\la,\xi_3,s) +\sum^{k(\lambda)}_{k=1} J^{-}_{k}(\la,\xi_3,s) + \sum_{k=1}^{k(\lambda)}J^{E}_{k}(\la,\xi_3,s),
\end{equation}
where 
\begin{equation}\label{I0}
J_0(\la,\xi_3,s):=  \lambda^{-\frac{1}{M}} \chi_{0}\biggl(\lambda^{\frac{M-1}{M}}|F(\xi)| \biggl)g\biggl(\lambda^{\frac{M-1}{M}}u,\xi_3, s_{2}\biggl),
\end{equation}
\begin{align}\label{Ik}
J^{\pm}_{k}(\la,\xi_3,s) &:= \lambda^{-\frac{1}{2}} \chi_{1}\biggl(\lambda^{\frac{M-1}{M}}2^{-k}|F(\xi)| \biggl)  |u|^{-\frac{M-2}{2M-2}}\,  A^\pm_{\la}(|u|, \xi_3,s_2) \nonumber\\
& \quad \quad \times e^{i\lambda \xi_3s_{2}^{\frac{1}{1-\kappa_{2}}}|u|^{\frac{M}{M-1}}q_{\pm}( |u|^{\frac{1}{M-1}} )}
\end{align}
and
\begin{align}\label{IE}
J^{E}_{k}(\la,\xi_3,s) = \lambda^{-1} \chi_{1}\biggl(\lambda^{\frac{M-1}{M}} 2^{-k}|F(\xi)| \biggl)  |u|^{-1}  E_\la(|u|, \xi_3,s_2),
\end{align}
with suitably chosen smooth cut-off function $\chi_0$ and $\chi_1.$ 
We note that $J^{\pm}_{k}(\la,\xi_3,s)$ and $J^{E}_{k}(\la,\xi_3,s)$   localize  to $|u| \sim 2^{k} \lambda^{-\frac{M-1}{M}}$, while  $J_{0}(\la,\xi_3,s)$ localizes to $|u| \lesssim \lambda^{-\frac{M-1}{M}}$.

Accordingly, we  decompose the multiplier $m_{\tau,2}^{\la}$ in \eqref{mla1B6} as 
\begin{equation}\label{m2taudecomp}
m_{\tau,2}^{\la}=m^{0}_{\lambda}+\sum_{\pm}\sum^{k(\lambda)}_{k=1}m^{\pm}_{\lambda,k}+\sum^{k(\lambda)}_{k=1}m^{E}_{\lambda,k},
\end{equation}
with
\begin{eqnarray}
m^{0}_{\lambda}(\xi)&:=& \la^{-\frac 12} \tilde{\chi}_1(s_1)\chi_1(s_2)\tilde {\tilde{\chi}}_1(\xi_3)\, e^{-i\la\xi_3\Big[s_1 s_2^{\frac{\ka_1}{1-\ka_2}}x_{0,1}+s_2^{\frac{1}{1-\ka_2}} \Psi(x_{0,1})\Big]} \times  J_0(\la,\xi_3,s),\label{m0k} \\
m^{\pm}_{\lambda,k}(\xi)&:=& \la^{-\frac 12} \tilde{\chi}_1(s_1)\chi_1(s_2)\tilde {\tilde{\chi}}_1(\xi_3)\, e^{-i\la\xi_3\Big[s_1 s_2^{\frac{\ka_1}{1-\ka_2}}x_{0,1}+s_2^{\frac{1}{1-\ka_2}} \Psi(x_{0,1})\Big]} \times  J^\pm_k(\la,\xi_3,s),\label{mpmlk}\\
m^{E}_{\lambda,k}(\xi)&:=&\la^{-\frac 12} \tilde{\chi}_1(s_1)\chi_1(s_2)\tilde {\tilde{\chi}}_1(\xi_3)\, e^{-i\la\xi_3\Big[s_1 s_2^{\frac{\ka_1}{1-\ka_2}}x_{0,1}+s_2^{\frac{1}{1-\ka_2}} \Psi(x_{0,1})\Big]} \times  J_k^{E}(\la,\xi_3,s).\label{mEk}
\end{eqnarray}
Then 
\begin{equation}\label{M2taudecomp}
\|\mathcal{M}_{m_{\tau,2}^{\la}}\|_{L^{p} \rightarrow L^{p}} \le \| \mathcal{M}_{m_{\lambda}^{0}}\|_{L^{p} \rightarrow L^{p}} +\sum_{\pm}\sum_{k=1}^{k(\lambda)} \| \mathcal{M}_{m^{\pm}_{\lambda,k}}\|_{L^{p} \rightarrow L^{p}} + \sum_{k=1}^{k(\lambda)} \| \mathcal{M}_{m^{E}_{\lambda,k}}\|_{L^{p} \rightarrow L^{p}}.
\end{equation}

We shall separately estimate the maximal operators corresponding to each of these Fourier multipliers, focusing, however, mostly on the main terms given by the  multipliers $m^{\pm}_{\lambda,k}.$ By means of 
 Theorem \ref{multipliertomaximal}, we can again essentially reduce the $L^p$- estimation  of  the maximal operators  $\mathcal{M}_{m_{\lambda,k}^{\pm}}$ to estimating  the Fourier multiplier operators  $T_{m_{\lambda,k}^{\pm}}.$ 
\medskip

To this end {\it let us  fix $k$  and the corresponding multiplier $m^{\pm}_{\lambda,k},$} and introduce the following abbreviations: We put 
\begin{equation}\label{ga+R}
\gamma := M/(M-1) ,\qquad R=R_k(\la):=2^{-k}\lambda^{\frac{M-1}{M}}.
\end{equation}
Since  $M\ge 3,$  we may assume here  that 
\begin{equation}\label{abbrevs}
1<\ga <2, \quad R\gg 1 \quad  \text{and} \quad \la\ge R^\ga.
\end{equation}

Then $m^{\pm}_{\lambda,k}$ is supported where $|\xi_i|\sim 1$ for $i=1,2,3$ (more precisely where  $|\xi_3|\sim 1, |s_1|\sim 1$ and $s_2\sim 1$),  and where $R|F(\xi)|\in \supp \chi_1.$

In particular, we see that 
\begin{equation}\label{Fsize}
|u|\sim |F(\xi)|\sim R^{-1}\qquad \text{on }\ \supp m^{\pm}_{\lambda,k},
\end{equation}
hence  $|u|^{-\frac{M-2}{2M-2}}\sim R^{\frac{M-2}{2M-2}}$ in \eqref{Ik}.

From  \eqref{mpmlk}, \eqref{Ik} and the estimates \eqref{symbol1}, we can then easily deduce that   we may write
\begin{equation}\label{mpmlk2}
m^{\pm}_{\lambda,k}=\la^{-1} R^{\frac{M-2}{2M-2}}\, \tilde{\chi}_1(s_1)\chi_1(s_2)\tilde {\tilde{\chi}}_1(\xi_3)\, 
 \chi_1\big(R|F(\xi)|\big)\, a_{\la,R}^\pm\big(F(\xi),\xi_3, s_2\big)\, e^{-i\la \xi_3\phi^\pm(\xi)},
\end{equation}
with the phase function 
\begin{equation}\label{phipm}
\phi^\pm(\xi):=\Big[s_1 s_2^{\frac{\ka_1}{1-\ka_2}}x_{0,1}+s_2^{\frac{1}{1-\ka_2}} \Psi(x_{0,1})\Big]+
s_{2}^{\frac{\ka_1}{1-\kappa_{2}}}|F(\xi)|^{\gamma} q_{\pm}\big(s_2^{\frac{(\gamma-1)(\ka_1-1)}{1-\ka_2}}|F(\xi) |^{\gamma-1}\big),
\end{equation}
and a smooth  amplitude  $a_{\la,R}^\pm\big(F(\xi), \xi_3,s_2\big)$, where $a_{\la,R}^\pm$ satisfies estimates of the form 
\begin{equation}\label{ampest}
|\partial^\alpha_v \partial^\beta_{\xi_3} \partial^\rho_{s_2} a_{\la,R}^\pm(v, \xi_3,s_2)|\le C_{k,\beta} R^\alpha 
\quad \quad \forall \alpha, \beta, \rho\in  \mathbb{N}.
\end{equation}

Observe now that the level sets  $\{F=c\}$ of $F$ are cones in $\Bbb R^3,$ which in the projective coordinates 
$s_1=\xi_1/\xi_3,s_2=\xi_2/\xi_3$ are given by the curves $s_1=c-s_2^{\frac{1-\kappa_{1}}{1-\kappa_{2}}}\Psi^{\prime}(x_{0,1}).$ These have curvature of modulus $\sim 1,$ since $s_2\sim 1.$  This indicates the close connection with the  maximal operator studied in \cite[Theorem 9.1]{BDIM25}, where the corresponding cone had been the light cone. 

Note also that by \eqref{Fsize}  the multiplier $m^{\pm}_{\lambda,k}$ localizes to a $1/R$-neighborhood of the cone $\{F=1/R\}, $ but apart from the ``cone multiplier''  $\chi_1\big(R|F(\xi)|\big)$ which localizes to this conic region, $m^{\pm}_{\lambda,k}$ has in addition an  oscillatory factor with a phase function at the  scale $\la\gg 1,$ thus merging this  cone  multiplier with a translation invariant Fourier integral operator (FIO). 

\medskip

Such classes of ``FIO-cone multiplier operators''  have been studied in \cite{BDIM25} for the case of the light cone 
$\mathbf \Gamma_{\rm lc}:=\{\xi\in\Bbb R^3: \xi_1^2+\xi_2^2=\xi_3^2\}.$ 

What will be important to us is that the  theory developed in \cite{BDIM25} can actually be easily extended to more general ``curved '' cones  than the light-cone.

This extension and the $L^p$-estimates for FIO-cone multipliers that we obtain by means of this extension  will be explained in the next subsection.

%%%%%%%%%%%%%%%%%%%%%%%%%%%%%%%%%%%%%%%%%%%%%%%%%%%%%%%%%%%%%%%%%%%%%%%%%%%%%%%%%%%%%%%%%%%%%%%%%%%%

\subsection{On FIO-cone multipliers}\label{FIOcone}
Let $\mathbf \Gamma$ be the cone
\begin{equation}\label{ccone}
\mathbf \Gamma=\{(\eta,h(\eta)): \eta\in \Omega\}\subset \RR^3\setminus\{0\},
\end{equation}
where   $\Omega$ is a  conic sector in $\Bbb R^2\setminus \{0\}$, and where $h:\Omega\to \RR\setminus \{0\}$ is a smooth  function which is homogeneous of degree 1 such that $\mathrm{rank}\,D_\eta^2 h=1$ on $\Omega$. We shall denote such cones briefly as {\it curved cones}.

Furthermore, assume that $F:\Omega\times(\RR\setminus \{0\}) \to \RR$ is a smooth, 0-homogeneous level function such that 
\begin{equation}\label{FGa}
\mathbf \Gamma=\{\xi\in \Omega\times(\RR\setminus \{0\}): F(\xi)=0\}
\end{equation} 
and $\nabla F(\xi)\ne 0$ on $\mathbf \Gamma$. 

For instance, if  $h(\eta)\ne 0$ for all $\eta\in \Omega,$ so that
$\mathbf \Gamma\subset \RR^3_\times:=\{\xi\in\RR^3: \xi_3\ne 0 \},$ then we could choose 
$$
F(\xi):=1-h\big(\frac {\xi_1}{\xi_3},\frac {\xi_2}{\xi_3}\big)=1-\frac 1{\xi_3}h(\xi_1,\xi_2).
$$

For $R\gg1,$ we then define the conic shells of thickness $\sim 1/R$ by 
\[ 
\Gamma^{\pm}_{R}
:= \big\{\xi\in \Omega\times(\RR\setminus \{0\}): \frac{1}{2R} < \pm F(\xi)  < \frac{1}{R}, \, 1<\xi_{3}<2\big\}.
\]

An {\it  FIO-cone multiplier}  associated to the cone $\mathbf \Gamma$ is then  a Fourier multiplier of the form
\begin{equation}\label{FIOconemult}
m_{\la,R}(\xi)=a_{\la,R}(\xi)m_R(\xi) e^{-i\lambda\phi(\xi)} ,
\end{equation}
where $m_R$ is  a smooth cut-off function localizing to  the  conic shell  $\Gamma^{\pm}_{R}$  of thickness $\sim 1/R,$ i.e., a {\it cone multiplier},  $a_{\la,R}$ is a smooth   amplitude  whose derivatives satisfy similar estimates as  the cone multiplier $m_R,$  where $\la\gg1,$ and  where $\phi$ is a phase lying in a certain class $\mathcal F^{\kappa_1,\kappa_2,\gamma}$ of real-valued homogeneous phase functions.

More precisely: Given any point $\xi_0$ in  $\Gamma^{\pm}_{R},$ let us denote by  $\mathbf \Gamma(\xi_0)$ the cone 
$$
\mathbf \Gamma(\xi_0):=\{\xi\in \Omega\times(\RR\setminus \{0\}): F(\xi)=F(\xi_0)\}
$$ 
passing through $\xi_0.$ 
Note that $\xi_0$ is tangential to $\mathbf \Gamma(\xi_0)$ at the point $\xi_0.$ Choose a second tangential vector $t(\xi_0)$ and  a vector $n(\xi_0)$ which is normal  to $\mathbf \Gamma(\xi_0)$ at  $\xi_0,$ both of length $\sim 1,$ such that $n(\xi_0),t(\xi_0),\xi_0$ form an {\it orthogonal  frame} at $\xi_0$.

  Then the  smooth function $a_{\la,R}$ on $\Gamma^{\pm}_{R}$ is  called an {\it admissible amplitude}, if for every $\xi_0\in\Gamma^{\pm}_{R}$
\begin{equation}\label{admisamp}
 |\partial_{n(\xi_0)}^{\alpha_1}\partial_{t(\xi_0)}^{\alpha_2} \partial_{\xi_0}^{\alpha_3} a_{\lambda,R}(\xi_0)|\leq B_{\alpha} R^{\alpha_1+\alpha_2/2}
 \qquad \text{for all } \al\in \NN^3.
\end{equation}

Let us next introduce classes of phase  functions in analogy to those  introduced in  \cite{BDIM25}:

Let $\kappa_1,\kappa_2\ge 0$ and $\gamma> 0$ be parameters satisfying $\kappa_2\geq \frac12$ and $ |\kappa_1-\kappa_2|\leq \frac 12.$

Then 
 $\mathcal F^{\kappa_1,\kappa_2,\gamma}$ denotes 
  the family of all 1-homogeneous real-valued  phase functions $\phi$  defined on a  suitable conic neighborhood of  
  $\mathbf \Gamma$ satisfying bounds on their derivatives of the form
\begin{equation}\label{deriv}
	 |\partial_{n(\xi_0)}^{\alpha_1}\partial_{t(\xi_0)}^{\alpha_2} \phi(\xi_0)|\leq C_{\alpha} R^{\alpha_1\kappa_1+\alpha_2\kappa_2-\gamma}
	 \quad \text{for all}\  \xi_0\in\Gamma^{\pm}_{R}, \ \alpha_1+\alpha_2\geq 2,
\end{equation}
uniformly in $R\gg1 $.
It is not difficult to see that for $R\gg1$ sufficiently large, these classes do not depend on the choice of level function $F$ for $\mathbf \Gamma$ (compare with the arguments leading to Lemma 3.5 in \cite{BDIM25}).

 For instance, a {\it classical} phase function, i.e., a smooth 1-homogeneous phase function on a conic neighborhood of the light-cone with uniform bounds on its  derivatives for $|\xi|\sim 1$ can easily be seen to be contained in the class $\mathcal F^{\frac\gamma 2,\frac\gamma 2,\gamma}$ for any $\gamma>0$ (compare with Example 3.10 in \cite{BDIM25}).

In a similar way as in \cite{BDIM25}, we next decompose the sector $\Omega\subset \Bbb R^2\setminus \{0\}$  into sectors of angular width  $R^{-1/2},$ and correspondingly we can cover $\Gamma^\pm_R$ by    finitely many overlapping ``sectors ''or ``caps'' $\theta$ of angular width $R^{-1/2},$ so that each  cap is then essentially a linear rectangular box of dimensions about $R^{-1}\times R^{-1/2}\times 1.$ 

However, the phase function $\la \phi$  can be seen to behave almost like a  linear  function only on smaller boxes $\vartheta$ inside caps $\theta$ of side lengths  $\rho_1\times\rho_2\times 1$, which are determined by the maximal size of its second order partial derivatives.  We  shall denote these  boxes $\vartheta$ as   {\it L-boxes}.  Since for $\phi\in\mathcal{F}^{\ka_1,\ka_2,\gamma}$, we have by \eqref{deriv} that
\begin{eqnarray*}
&&|\partial_{n(\xi)}^2\la \phi(\xi)| \lesssim \la R^{2\kappa_1-\gamma} \\
&&|\partial_{t(\xi)}^2\la\phi(\xi)| \lesssim \la R^{2\kappa_2-\gamma}.
\end{eqnarray*}
It is thus  natural to define, for $j=1,2,$ 
$$
\rho_j:=\min\big\{(\lambda R^{2\kappa_j-\gamma})^{-1/2},R^{-1/j}\big\}. 
$$
Since we shall always assume that $\la\ge R^\ga,$ we note that $(\lambda R^{2\kappa_2-\gamma})^{-1/2}\le R^{-1/2},$ 
so that 
\begin{eqnarray}\label{rhodef}
\rho_1=\min\big\{(\lambda R^{2\kappa_1-\gamma})^{-1/2},R^{-1}\big\}	\qquad \rho_2 =(\lambda R^{2\kappa_2-\gamma})^{-1/2}.
\end{eqnarray}

We can then  decompose  each sector $\theta$ of angular width $R^{-1/2}$ in $\Gamma_R^\pm$ into at most  $N(\lambda,R)$ L-boxes $\vartheta,$ where 
\begin{equation}\label{Ndefine}
	N(\lambda,R)\sim \frac{R^{-1}\cdot R^{-1/2}}{\rho_1\rho_2}.
\end{equation}

Finally, we shall also need a strengthening of  the estimates \eqref{deriv} for the particular case $\alpha_1=\alpha_2=1,$ the \textit{small mixed derivative condition} (SMD) (for the class  $\mathcal F^{\kappa_1,\kappa_2,\gamma}$):
\begin{equation}\label{smd}
	|\partial_{n(\xi_0)}\partial_{t(\xi_0)}\, \la \phi(\xi_0)|
	\leq C \, \big((R\rho_2)^{-1}\wedge 1\big) \rho_1^{-1}\rho_2^{-1}.
\end{equation}

Then the following extension of Theorem 4.6   in  \cite{BDIM25} to more general curved cones  of the form  \eqref{ccone}  holds true:
\begin{theorem}\label{FIOconespec}
Let $m_{\la,R}$ be an FIO-cone multiplier as in \eqref{FIOconemult}, with a  phase $\phi\in \mathcal{F}^{\ka_1,\ka_2,\gamma}$ and an admissible amplitude $a_{\la,R}$, and assume that $\lambda\geq R^\gamma\gg 1.$ Assume also that $\phi$ satisfies the small mixed derivative condition (SMD) given by \eqref{smd}. 
Then for every  $\epsilon>0$ the following estimate 
\begin{equation}\label{conefiosest}
	\|T_{m_{\la,R}}\bigl\|_{L^{4} \rightarrow L^{4}}
	\leq C_{\epsilon} R^{\epsilon} N(\lambda,R)^{\frac12}
\end{equation}
holds true, with a constant $ C_{\epsilon}$ depending  only on $\epsilon$ and  the constants $B_\alpha$ in \eqref{admisamp} and $C_\alpha$ in \eqref{deriv}, but not on $R$ or $\la.$ 
\end{theorem}

{\it As for the proof of this theorem:} we first recall that for the case of the light-cone, the   proofs in   \cite{BDIM25}  build on a key estimate in  the article \cite{GWZ20} by Guth, Wang and Zhang, namely  estimate (5.4), which states that for every $\epsilon>0$ and every function $f$ whose Fourier support is contained in $\Gamma^\pm_R $ we can estimate 
\begin{equation}\label{gwzest}
\| f\|_4\le C_\epsilon R^\epsilon \|f\|_{GWZ}.
\end{equation}

Here $ \|\cdot\|_{GWZ}$   is a  complicated but crucial semi-norm introduced in \cite{GWZ20}, called the Guth-Wang-Zhang semi-norm in \cite{BDIM25}, which is well-suited for induction on scales arguments. The key point  in \cite{BDIM25} was then to  bound the Guth-Wang-Zhang semi-norm of  $T_{m_{\la,R}} f$ in a suitable way by  the $L^4$-norm of $f$. 

An inspection of the arguments in \cite{BDIM25}  reveals that once the  natural analogue of the estimate \eqref{gwzest} is established also for any curved  cone of the form \eqref{ccone}, then the proof of Theorem \ref{FIOconespec} follows by only slight and straight-forward modifications of  the arguments in \cite{BDIM25}. 

And, as already explained in Remark 1.7 of \cite{GLMX23}, the proof by Guth, Wang and Zhang  of estimate \eqref{gwzest} for the case of the light-cone  can be extended to more general curved cones by means of slight modifications, with the exception of \cite[Lemma 4.2]{GWZ20}, for which, however,  an alternative proof working also for more general cones is provided in \cite{GLMX23}.

We shall give more details on  how the natural Guth-Wang-Zhang semi-norms ought to be defined for our more general cones, and also sketch an argument how the estimate  \eqref{gwzest} can alternatively be also derived directly from the results in 
\cite{GLMX23} in the Appendix \ref{GWZcont}.
\smallskip

The following corollary follows immediately from Theorem \ref{FIOconespec} by duality and interpolation with the trivial $L^2$-estimate:
\begin{corollary}\label{FIOconespecc}
Let $m_{\la,R}$ be an FIO-cone multiplier as in \eqref{FIOconemult}, with a  phase $\phi\in \mathcal{F}^{\ka_1,\ka_2,\gamma}$ and an admissible amplitude $a_{\la,R}$, and assume that $\lambda\geq R^\gamma\gg 1.$ Assume also that $\phi$ satisfies the \textit{small mixed derivative condition } (SMD). 
Then for every  $p$ in the range $4/3\le p\le 4$ and every $\epsilon>0$ the following estimate 
\begin{equation}\label{conefiosest}
	\|T_{m_{\la,R}}\bigl\|_{L^{p} \rightarrow L^{p}}
	\leq C_{\epsilon} R^{\epsilon} N(\lambda,R)^{|1-\frac2p|}
\end{equation}
holds true, with a constant $ C_{\epsilon}$ depending  only on $\epsilon$ and  the constants $B_\alpha$ in \eqref{admisamp} and $C_\alpha$ in \eqref{deriv}, but not on $R$ or $\la.$ 
\end{corollary}

Of particular importance to us will here be the phase functions 
$$\phi^{\gamma}(\xi) := \xi_{3} |F(\xi)|^{\gamma},$$
where $F$ is a level function as in \eqref{FGa}.

Our next lemma shows  that   the class $\mathcal F^{1,\frac{\ga}2,\gamma}$  will be of particular   interest to us in relation to these phases. We observe that  our assumption  $\lambda\geq R^\gamma$ implies that for  this  class   we 
have $\rho_1\leq R^{-1},$ so that
\begin{equation}\label{rhoga}
\rho_1=\la^{-\frac 12}R^{-(1-\frac {\ga}2)},\qquad  \rho_2= \la^{-1/2},
\end{equation}
and thus 
\begin{equation}\label{Ndefinega}
N(\lambda,R)=\lambda R^{-\frac{1+\gamma}2}.
\end{equation}

\begin{lemma}\label{onphiga}
Assume that $1\le\ga\le 2$ and $\la\ge R^\ga.$ Then $\phi^\ga\in \mathcal F^{1,\frac12,\gamma}.$ 
Moreover, if $\phi_0$ is  any classical phase function, then 
 $\phi_0+\phi^\ga\in \mathcal F^{1,\frac{\ga}2,\gamma},$ and $\phi_0$ as well as $\phi^\ga$ satisfies  the small mixed derivative condition (SMD) in \eqref{smd}  for this class $\mathcal F^{1,\frac{\ga}2,\gamma}.$
\end{lemma}

 \begin{proof} Fix $\xi_0\in\Gamma^{\pm}_{R},$ and let  $n(\xi_0),t(\xi_0),\xi_0$ be the orthogonal frame at $\xi_0$ devised before. We  first want to show that $\phi^\ga$ satisfies the estimates of the form  \eqref{deriv} with $\ka_1=1$ and $\ka_2=1/2.$ 
 
 Denote by $E_1,E_2,E_3$ the orthonormal frame at $\xi_0$ given by 
 $E_1:=|n(\xi_0)|^{-1} n(\xi_0), E_2:=|t(\xi_0)|^{-1} t(\xi_0), E_3:= |\xi_0|^{-1} \xi_0.$ %We may assume that this  is a positively oriented frame at $\xi_0.$ 
 Observe first  that the estimates \eqref{deriv}  are then equivalent to estimates of the form 
 \begin{equation}\label{paEj}
 |\partial_{E_1}^{\alpha_1}\partial_{E_2}^{\alpha_2}  \phi^\ga(\xi_0)|\leq C'_{\alpha} R^{\alpha_1+\frac {\alpha_2}2-\ga},
 \qquad \al_1+\al_2\ge 2.
\end{equation}
Next, it  easy to see that the estimates \eqref{deriv} are invariant under rotations of $\RR^3$ (compare Section 3.2 in  \cite{BDIM25}). More specifically,  denote by $e_1,e_2,e_3$ the canonical basis of $\RR^3.$ We then choose  $A\in O(3,\RR)$ in such a way that  $\eta_0:= A\xi_0=|\xi_0|e_3,$  and so that the tangent plane to the rotated cone $A\big(\mathbf \Gamma(\xi_0)\big)$ at $\eta_0$ is spanned by $e_2$ and $e_3.$ In the coordinates 
 $\eta:=A\xi$ the cone is then described by the rotated level function $F_A(\eta):= F(A^{-1}\eta),$ i.e.,
 $$
 A\big(\mathbf \Gamma(\xi_0)\big)\subset \{\eta: F_A(\eta)=F_A(\eta_0)\}.
 $$
 Moreover, in the coordinates $\eta,$ the orthonormal frame $E_1,E_2,E_3$  corresponds to the orthonormal frame  $A\cdot E_1,A\cdot E_2,A\cdot E_3,$ in the following sense:
 
Denote by $\phi_A^\ga(\eta):=\phi^\ga(A^{-1} \eta)$ the phase $\phi^\ga$ when expressed in the rotated coordinates $\eta.$ Then
$\partial_{E_j}\phi^\ga(\xi)=\partial_{A E_j}\phi_A^\ga(\eta), \, j=1,2,3.$ But notice that by our choice of $A,$ we may assume that $AE_1=e_1,$ since $AE_1$ is normal to $A\big(\mathbf \Gamma(\xi_0)\big)$ at $\eta_0,$ and $AE_3=e_3,$ hence also $AE_2=e_2.$ 
Thus  the estimates \eqref{paEj} are  equivalent to the following estimates:
\begin{equation}\label{paej}
 |\partial_{\eta_1}^{\alpha_1}\partial_{\eta_2}^{\alpha_2}  \phi_A^\ga(\eta_0)|\leq C'_{\alpha} R^{\alpha_1+\frac {\alpha_2}2-\ga},
 \qquad \al_1+\al_2\ge 2.
 \end{equation}
 
Next, let us put $c:=F(\xi_0)=F_A(\eta_0),$ where $|c|\sim R^{-1},$ and $G(\eta):= F_A(\eta)-c$, so that  $A\big(\mathbf \Gamma(\xi_0)\big)\subset \{\eta: G(\eta)=0\}.$ Note that $G$ homogeneous of degree $0,$ so that
$$
G(\eta)=\tilde G(\si_1,\si_2),
$$
if we use the abbreviation $\si_1:=\eta_1/\eta_3,\, \si_2:=\eta_2/\eta_3.$ We also set $\si:=(\si_1,\si_2).$ 

Note that then 
$$
\tilde G(0,0)=0\quad \text{ and  } \quad \partial_{\si_2}\tilde G(0,0)=0,
$$ 
since $G(0,0, |\xi_0|)=G(\eta_0)=0$ and $\partial_{\eta_2}G(0,0,|\xi_0|))=\partial_{\eta_2}G(\eta_0)=0.$ Thus, a Taylor expansion of $\tilde G$ at $(0,0)$ yields that
$$
\tilde G(\si_1,\si_2)=\si_1\varphi_1(\si)+\si_2^2\varphi_2(\si),
$$
with smooth functions $\varphi_1, \varphi_2$ defined near the origin. Since moreover $ \phi_A^\ga(\eta)=|F_A(\eta)|^\ga=|\tilde G(\sigma)+c|^\ga,$  we finally observe  that \eqref{paej} is equivalent to proving estimates of the form
\begin{equation}\label{paesi}
 \big|\partial_{\si_1}^{\alpha_1}\partial_{\si_2}^{\alpha_2}  |\tilde G+c|^\ga (0,0)\big|\leq C'_{\alpha} R^{\alpha_1+\frac {\alpha_2}2-\ga},
 \qquad \al_1+\al_2\ge 2.
 \end{equation}
 Let us finally consider the re-scaled function 
 $$
 H_R(s_1,s_2):= R\, \big[\tilde G(R^{-1}s_1, R^{-1/2} s_2)+c\big]=s_1\varphi_1(R^{-1}s_1, R^{-1/2} s_2)+ s_2^2 \varphi_2(R^{-1}s_1, R^{-1/2} s_2)+\tilde c,
$$
where $|\tilde c|\sim 1,$ and where we may assume that $|s_1|,|s_2|\ll 1.$ Then the estimates \eqref{paesi} are equivalent to estimates of the form
$$
 \big|\partial_{s_1}^{\alpha_1}\partial_{s_2}^{\alpha_2}  |H_R(s_1,s_2)|^\ga (0,0)\big|\leq C'_{\alpha}, 
 \qquad \al_1+\al_2\ge 2,
$$
with constants $C'_{\alpha}$ not depending on $R\gg 1.$ But these are obvious.
\smallskip

We have thus shown that $\phi^\ga\in \mathcal F^{1,\frac12,\gamma}.$ 

Next, by  \eqref{rhoga}, the small mixed derivative condition  for the class $\mathcal F^{1,\frac{\ga}2,\gamma}$ asks for an estimate of the form
\begin{equation}\label{smd3}
|\partial_{n(\xi_0)}\partial_{t(\xi_0)} \phi^{\gamma}(\xi_0)|\leq C \,\big(\frac {\la^{\frac 12}} R\wedge 1\big) R^{1-\frac \ga 2}.
\end{equation}
To this end, we shall  show that the following estimate holds true:
\begin{equation}\label{smd2}
	|\partial_{n(\xi_0)}\partial_{t(\xi_0)} \phi^\gamma(\xi_0)|
	\leq C R^{1-\gamma}.
\end{equation}
It is easily checked that this estimate implies the required estimate under our assumptions $\la\ge R^\ga$ and $\ga\ge 1.$
In the coordinates $\eta,$ the estimate \eqref{smd2} is equivalent to $|\partial_{\eta_1}\partial_{\eta_2}  \phi_A^\ga(\eta_0)|\lesssim R^{1-\ga},$ i.e., to 
$$
\big|\partial_{\si_1}\partial_{\si_2}  |\tilde G(\si)+c|^\ga\vert_{\si=(0,0)}\big|\lesssim R^{1-\ga}.
$$
But, assuming that $c>0,$ say, and that $|\si|\ll 1,$ then 
$$
\partial_{\si_2} |\tilde G(\si)+c|^\ga=\ga|\tilde G(\si)+c|^{\ga-1}\big(\si_1\partial_{\si_2}\varphi_1(\si)+2\si_2 \varphi_2(\si)+\si_2^2 \partial_{\si_2} \varphi_2(\si)\big),
$$
so that at $\si=(0,0)$ we obtain 
$$
\partial_{\si_1} \partial_{\si_2} |\tilde G(\si)+c|^\ga\vert_{\si=(0,0)}=\ga(\ga-1)|\tilde G(0,0)+c|^{\ga-2}\cdot 0
+\mathcal{O}\big(\ga|\tilde G(0,0)+c|^{\ga-1}\big)=\mathcal{O}(R^{1-\ga}).
$$

As for the classical phase $\phi_0,$ it is easy to check that $\phi_0\in \mathcal F^{1,\frac\ga2,\gamma}$, so that also
 $\phi_0+\phi^\ga\in \mathcal F^{1,\frac\ga 2,\gamma}.$
 In order to show that also $\phi_0$ satisfies the  condition (SMD), in view of \eqref{smd3} it suffices to show that 
 $$
 1\lesssim  C \,\big(\frac {\la^{\frac 12}} R\wedge 1\big) R^{1-\frac \ga 2},
  $$
  i.e., that $1\lesssim (\la R^{-\ga})^{1/2}$ and $1\lesssim R^{1-\frac \ga 2}.$ But these estimates are clear since $\la\ge R^{\ga}$ and $\ga\le 2.$ 
  \end{proof}

%%%%%%%%%%%%%%%%%%%%%%%%%%%%%%%%%%%%%%%%%%%%%%%%%%%%%%%%%%%%%%%%%%%%%%%%%%%%%%%%%%%%%%%%%%%%%%%%%%%%

\subsection{Applying FIO-cone multiplier estimates}
We now turn back to our multipliers $m_{\lambda,k}^\pm,$ which  according to \eqref{mpmlk2} and \eqref{phipm} can be written as $m^{\pm}_{\lambda,k}=\la^{-1} R^{\frac{M-2}{2M-2}}\,m_{\la,R}(\xi),$ 
if we set
\begin{equation}\label{mlaR}
 m_{\la,R}(\xi):=m_R(\xi)\, a_{\la,R}^\pm\big(F(\xi),\xi_3, s_2\big)\, e^{-i\la \xi_3\phi^\pm(\xi)}.
\end{equation}
Here we have  put (again with $s_1=\xi_1/\xi_3, s_2=\xi_2/\xi_3$)
$$
m_R(\xi):= \tilde{\chi}_1(s_1)\chi_1(s_2)\tilde {\tilde{\chi}}_1(\xi_3)\, 
 \chi_1\big(R|F(\xi)|\big),
 $$
and the phase is given by   $\xi_3\phi^\pm(\xi)=\phi_0(\xi)+\tilde \phi^\ga(\xi),$  with the classical phase 
$$
\phi_0(\xi):=\xi_3\Big[s_1 s_2^{\frac{\ka_1}{1-\ka_2}}x_{0,1}+s_2^{\frac{1}{1-\ka_2}} \Psi(x_{0,1})\Big],
$$
and where
$$
\tilde \phi^\ga(\xi):=
\xi_3 \, s_{2}^{\frac{\ka_1}{1-\kappa_{2}}} q_{\pm}\big(s_2^{\frac{(\gamma-1)(\ka_1-1)}{1-\ka_2}}|F(\xi) |^{\gamma-1}\big)\, |F(\xi)|^{\gamma}.
$$
Recall also from  \eqref{ga+R} that $\gamma=\frac{M}{M-1},$ so that $1<\ga<2.$ Re-writing correspondingly 
  $\frac{M-2}{2M-2}=1-\frac \ga 2,$ we obtain that $\la^{-1} R^{\frac{M-2}{2M-2}}=\la^{-1}R^{1-\frac \ga 2}.$
We can thus write 
\begin{equation}\label{mpmlk4}
m^{\pm}_{\lambda,k}=\la^{-1}R^{1-\frac \ga 2} \,m_{\la,R}.
\end{equation}
  
Recall also from \eqref {formulaF}  that 
$$
	F(\xi)=\tilde F(s):= s_1-c s_2^b,
$$
where $c:=-\Psi^{\prime}(x_{0,1})\ne 0$ and $b:=\frac{1-\kappa_{1}}{1-\kappa_{2}}\ne 1.$ 
Note that if we define the function 
$$
h(\eta_1,\eta_2):= \Big(\frac {\eta_1}{c\eta_2^b}\Big)^{\frac 1{1-b}},
$$
then we can view $F$ is a level function for the cone 
$$
\mathbf \Gamma:=\{(\eta,h(\eta)): \eta\in \Omega\}\subset \RR^3\setminus\{0\}
$$
(where $\Omega$ is a suitable sector containing the support of the function $\tilde{\chi}_1(\eta_1)\chi_1(\eta_2)$),  i.e.,
$$
\mathbf \Gamma \subset \{F=0\}.
$$
Since the curve in the plane given by $s_1=cs_2^b, s_2\sim 1,$ has non-vanishing curvature, this cone is curved.
Note also that the function $m_R$ localizes to the union of $\Gamma_R^+$ and $\Gamma_R^-.$ 
\smallskip

As for the amplitude $a_{\la,R}^\pm\big(F(\xi),\xi_3, s_2\big)$ in \eqref{mlaR}, by means of the scaling argument used in the proof of Lemma \ref{onphiga} it is easy to see that the estimates \eqref{ampest} imply that this amplitude is admissible (compare also with Lemma 4.5 and its proof in \cite{BDIM25}).

\smallskip
Let us finally  show that Lemma \ref{onphiga} remains valid also for our phase $\phi_0+\tilde \phi^\ga,$ i.e., that 
 $\phi_0+\tilde \phi^\ga$ lies in the  class $\mathcal F^{1,\frac \ga 2,\gamma}$ and satisfies the small mixed derivative condition (SMD): 
 
 \smallskip
For the first claim, it will again suffice to prove that  $\tilde \phi^\ga\in \mathcal F^{1,\frac 12,\gamma}.$  Now, since 
 $q_\pm$ is analytic, by means of a Taylor expansion of  $q_\pm$ we can expand
\begin{equation}\label{q+taylor}
\tilde \phi^\ga(\xi)
=\sum_{l=1}^\infty c_l s_2^{\frac{\kappa_{1}-1}{\kappa_{2}-1}\cdot (1-\gamma)l+\frac{\ka_1}{1-\kappa_{2}}}\, \xi_3\, |F(\xi)|^{\gamma+l(\gamma-1)}.
\end{equation}

Since the exponents   $\gamma_l:=\gamma+l(\gamma-1)$ satisfy $\ga_l>\gamma>1$, we can apply the same arguments as in the proof of Lemma \ref{onphiga} to deduce that
$$
 \xi_3|F(\xi)|^{\gamma+l(\gamma-1)} \in \mathcal F ^{1,\frac12,\gamma_l}\subset R^{-l(\ga-1)}\cdot \mathcal F ^{1,\frac12,\gamma}.
$$
Note also that the additional factors $s_2^{\frac{\kappa_{1}-1}{\kappa_{2}-1}\cdot (1-\gamma)l+\frac{\ka_1}{1-\kappa_{2}}}$ in the  terms of this series is harmless, since its derivatives in $\xi$ are bounded independently of $R.$

Thus by applying Leibniz' rule we see that indeed $\tilde \phi^\ga\in \mathcal F^{1,\frac 12,\gamma}.$

Finally, to see that $\phi_0+\tilde \phi^\ga$ satisfies the small mixed derivative condition (SMD), it suffices by Lemma \ref{onphiga} to prove this for $\tilde \phi^\ga$. In analogy with \eqref{smd2}, it will  suffice to prove that
$$
	|\partial_{n(\xi_0)}\partial_{t(\xi_0)} \tilde \phi^\gamma(\xi_0)|
	\leq C R^{1-\gamma}.
$$
Let us show this for the main term $c_0s_2^{\frac{\ka_1}{1-\kappa_{2}}}\, \xi_3\, |F(\xi)|^{\gamma}$ with $l=0$ in \eqref{q+taylor} -- for the higher order terms with $l\ge 1$  we can again argue in a very similar way:

 We again apply Leibniz' rule. Note that if at least one of the derivatives $\partial_{n(\xi_0)}$ or $\partial_{t(\xi_0)}$ hits the factor $s_2^{\frac{\ka_1}{1-\kappa_{2}}},$ then we can clearly estimate the corresponding term of $\partial_{n(\xi_0)}\partial_{t(\xi_0)} \tilde \phi^\gamma(\xi_0)$ by nothing worse than $\mathcal {O} (R^{1-\ga}),$ which is what we want. And, when  both of these derivatives hit the factor  $|F(\xi)|^{\gamma}$, then estimate \eqref{smd2} yields  the desired estimate.

\smallskip

We have thus shown that $m_{\la,R}$ is an FIO-cone multiplier to which we can apply Corollary \ref{FIOconespecc}, so that for every  $p$ in the range $4/3\le p\le 2$ and every $\epsilon>0$  we can estimate
\begin{equation}\label{TmlaB}
\|T_{m_{\la,R}}\bigl\|_{L^{p} \rightarrow L^{p}}
	\leq C_{\epsilon}\, R^{\epsilon}  N(\lambda,R)^{\frac2p-1}.
\end{equation}

Recalling from  \eqref{Ndefinega} that $N(\lambda,R)=\lambda R^{-\frac{1+\gamma}2},$ in combination with \eqref{mpmlk4} we obtain that
 \begin{equation}\label{mpmlaest2}
\|T_{m^\pm_{\la,k}}\bigl\|_{L^{p} \rightarrow L^{p}}
	\leq C_{\epsilon}  \lambda^{\frac{2}{p}-2} R^{\frac 32-\frac{1+\gamma}{p} +\epsilon},\qquad \qquad 4/3\le p\le 2.
 \end{equation}
 This  estimate  coincides with the estimate  (9.22)  in  \cite{BDIM25}.

\smallskip

We want to apply Theorem \ref{multipliertomaximal} in order to also bound the maximal operator 
$\mathcal{M}_{m_{\lambda,k}^{\pm}}.$ 
To this end, we need to consider also  the multiplier 
$$
\dot {m}_{\la,R}:=\frac {d}{dt} m_{\la,R}(t\xi)\big|_{t=1},
$$ 
which we shall then write as $\dot {m}_{\la,R}=\la \widetilde{m_{\la,R}}.$
From \eqref{mlaR}, we compute that 
\begin{eqnarray}\label{partialtmjk}
 \dot{m} _{\la,R}(\xi)&=&-i\la\xi_3\, \phi^\pm(\xi)\, m_R(\xi)\, a_{\la,R}^\pm\big(F(\xi),\xi_3, s_2\big)\, e^{-i\la \xi_3\phi^\pm(\xi)}\\ \nonumber
&&\qquad +\dot m_R(\xi)\, a_{\la,R}^\pm\big(F(\xi),\xi_3, s_2\big)\, e^{-i\la \xi_3\phi^\pm(\xi)}\\ \nonumber
&&\qquad + m_R(\xi)\, \dot a_{\la,R}^\pm\big(F(\xi),\xi_3, s_2\big)\, e^{-i\la \xi_3\phi^\pm(\xi)}. \nonumber
\end{eqnarray}
Here we have exploited that, as functions of $\xi,$   $s_1,s_2,F$ and $\phi^\pm$ are homogeneous of degree $0,$ whereas  $\xi_3$ is  homogeneous of degree $1.$ Accordingly we decompose 
$$
 \widetilde{m_{\la,R}}= \widetilde{m_{\la,R,I}}+\la^{-1} \widetilde{m_{\la,R,II}},
 $$
 where 
 \begin{equation}\label{tildemladecomp}
\widetilde{m_{\la,R,I}}(\xi):=-i\xi_3\, \phi^\pm(\xi) \, m_R(\xi)\, a_{\la,R}^\pm\big(F(\xi),\xi_3, s_2\big)\, e^{-i\la \xi_3\phi^\pm(\xi)}.
\end{equation}

As for $\widetilde{m_{\la,R,II}},$ it is easily seen that $\dot m_R(\xi)$ is just a slightly modified version of the function $m_R$ which localizes to $\Gamma^\pm_R,$ and similarly  $\dot a_{\la,R}^\pm$ is slightly modified version of  $a_{\la,R}^\pm,$ satisfying again the estimates \eqref{ampest}.

As for $\widetilde{m_{\la,R,I}},$ recall that $\xi_3\phi^\pm=\phi_0+\tilde \phi^\ga.$ 
The product of classical phase $\phi_0$ with  $m_R$ can  here easily be absorbed into  the cut-offs in $s_1$ and $s_2$,  just leading  to  modified versions of those.

 Similarly, since 
 $$
|F(\xi)|^\ga \chi_1\big(R|F(\xi)|\big) =R^{-\ga} \chi_{1,\ga}(R|F(\xi)|),
 $$
 if we set $\chi_{1,\ga}(v):=v^\ga \chi_1(v),$ which  is  again an admissible  amplitude, we can see in  a similar way that the product with $\tilde \phi^\ga$ can also be absorbed into a slightly modified version of the cut-off $m_R.$

These discussions show that we can apply Corollary \ref{FIOconespecc} in the same way as before to  $\widetilde{m_{\la,R}}$ in place off  $m_{\la,R}$ and obtain
  \begin{equation}\label{TmlaBtilde}
\|T_{\widetilde{m_{\la,R}}}\bigl\|_{L^{p} \rightarrow L^{p}}
	\leq C_{\epsilon}\, R^{\epsilon}  N(\lambda,R)^{\frac2p-1}.
\end{equation}
Thus, if we write 
$
\dot {m}^\pm_{\la,k}:=\frac {d}{dt} m^\pm_{\la,k}(t\xi)\big|_{t=1}=\la \, \widetilde{{m}^\pm_{\la,k}},
$ 
then 
\begin{equation}\label{mpmlaesttilde}
\|T_{\widetilde{m_{\lambda,k}^{\pm}}}\bigl\|_{L^{p} \rightarrow L^{p}}
	\leq C_{\epsilon}  \lambda^{\frac{2}{p}-2} R^{-\frac{1+\gamma}{p} +\frac 32+\epsilon}, \qquad \qquad 4/3\le p\le 2.
 \end{equation}
 
Finally, by means of Lemma \ref{generalkernelestimate} we can verify that Condition (ii)   in Theorem \ref{multipliertomaximal}  holds here again with a constant $B=\la^C$ as in \eqref{kernelA2}.

Note also that for $\la\ge R^\ga$  and $R\gg 1$ we have  $A:=\lambda^{\frac{2}{p}-2} R^{\frac 32-\frac{1+\gamma}{p} +\epsilon}\ll 1.$ Thus, by applying Theorem \ref{multipliertomaximal}, we obtain that
$$
\|\mathcal{M}_{m^\pm_{\la,k}} \|_{L^p\to L^p}\le C_{\epsilon,\delta,p}\,\left\{ \la^{\frac 1p}\left( \la^{ \frac 2p-2} R^{-\frac{1+\ga}{p}+\frac 32+\epsilon}\right)^{1-\delta}+\la^{\frac 1p}\log \la \,\la^{ \frac 2p-2} R^{-\frac{1+\ga}{p}+\frac 32+\epsilon} \right\}.
$$
Hence, by choosing $\epsilon$ and $\delta$ sufficiently small, we see that if $4/3\le p\le 2$, then
\begin{equation}\label{Mest1}
\|\mathcal{M}_{m^\pm_{\la,k}} \|_{L^p\to L^p} \le C_{\epsilon,p}\, \la^{ \frac 3p-2+\delta} R^{-\frac{1+\ga}{p}+\frac 32+\delta}
\end{equation}
 for every $\delta>0.$
 \smallskip
 
 Assume now that  $p >p_M:=\max\big\{\frac{3}{2}, \frac{2(M+1)}{M+3}\big\},$ so that in particular  $p>3/2.$  
Since $\la\ge R^\ga,$  by choosing $\delta>0$ sufficiently small we can estimate
$$
\la^{ \frac 3p-2+\delta} R^{-\frac{1+\ga}{p}+\frac 32+\delta}\le \la^{-\delta}R^{(\frac 3p -2+2\delta)\ga}R^{-\frac{1+\ga}{p}+\frac 32+\delta}\le 
\la^{-\delta}R^{\frac 1p( 2\ga-1)-2\ga+\frac 32+\delta(2\ga+1)}.
$$
Note that  $\frac 1p( 2\ga-1)-2\ga+\frac 32<0$ if and only if 
$$
p>\frac {2\ga-1}{2\ga-3/2}=2\frac {M+1}{M+3}=p_M.
$$
Thus, by choosing $\delta$ sufficiently small, we see that for $p>p_M$ we can establish an estimate 
$$
\|\mathcal{M}_{m^\pm_{\la,k}} \|_{L^p\to L^p} \le C_p\la^{-\delta_p}
$$
for some $\delta_p>0.$ Since each of the sums in \eqref{m2taudecomp} consists of at most $k(\la)=\mathcal{O}(\log \la)$ terms, the previous inequality implies that for $p>p_M$ 
\begin{equation}\label{sumMkla}
\sum\limits_{\la\gg 1} \sum\limits_{k=1}^{k(\la)}\|\mathcal{M}_{m^\pm_{\la,k}} \|_{L^p\to L^p}\le C'_p<\infty
\end{equation}
(recall that we are here summing only over dyadic values of  $\la$).
 
 This is exactly what we  had to prove for roots $\ga_\tau$ of Type $\mathrm{B_{T}}.$
 
As for the contributions by the multipliers  for  $m^E_{\la,k} $ in \eqref{m2taudecomp}, note that the function $E_\la$ in \eqref{IE} satisfy even better estimates than the symbols $A^\pm_\la$ in \eqref{Ik}, and thus we obtain the same type of estimate
\eqref{sumMkla} for  the maximal operators $\mathcal{M}_{m^E_{\la,k}} $ in place of $\mathcal{M}_{m^\pm_{\la,k}}.$

A similar argument applies also for the multiplier $m^0_\la$ in  \eqref{m2taudecomp} (compare also with the very similar arguments used for the maximal operator in Section 9 of \cite{BDIM25}).

This proves Theorem \ref{theorem2.1} (b) for the case of a root $\ga_\tau$ of Type  $\mathrm{B_{T}}.$

 \subsubsection{The case of roots of Type  $\mathrm{B_{NT}}.$}
For roots of  Type $\mathrm{B_{NT}}$ we need to improve the $p$-range above to the range $p>3/2.$ 

To this end, let us have a closer look at the phase $\xi_3\phi^\pm(\xi)=\phi_0(\xi)+\tilde \phi^\ga(\xi).$  We shall show that 
if $\ga_\tau$ is  of  Type $\mathrm{B_{NT}},$ then 
\begin{equation}\label{phipmi}
\xi_3\phi^\pm(\xi)=\mathcal{O}(R^{-1}).
\end{equation}
 
 Recall first from Lemma \ref{Legendret}  a) that in this case 
$\Psi(x_{0,1})=x_{0,1}\Psi'(x_{0,1}),$
  and that 
$\xi_3\phi^\pm(\xi)=\phi_0(\xi)+\tilde \phi^\ga(\xi),$  with the classical phase 
$$
\phi_0(\xi):=\xi_3s_2^{\frac{\ka_1}{1-\ka_2}}\Big[s_1 x_{0,1}+s_2^{\frac{1-\ka_1}{1-\ka_2}} \Psi(x_{0,1})\Big],
$$
and 
$$
\tilde \phi^\ga(\xi):=
\xi_3 \, s_{2}^{\frac{\ka_1}{1-\kappa_{2}}} q_{\pm}\big(s_2^{\frac{(\gamma-1)(\ka_1-1)}{1-\ka_2}}|F(\xi) |^{\gamma-1}\big)\, |F(\xi)|^{\gamma}.
$$
We can then write
$$
s_1 x_{0,1}+s_2^{\frac{1-\ka_1}{1-\ka_2}} \Psi(x_{0,1})=x_{0,1}\big[s_1 +s_2^{\frac{1-\ka_1}{1-\ka_2}} \Psi'(x_{0,1})\big]=x_{0,1} F(\xi),
$$
where $|F(\xi)|\sim R^{-1}$ on the support of $m_R.$ Since $\ga>1,$ we then obtain \eqref{phipmi}. 

This shows that we can here write
\begin{eqnarray*}
 \widetilde{m_{\la,R,I}}(\xi)&=&-i\, R^{-1}\xi_3\, \tilde m_R(\xi)\, a_{\la,R}^\pm\big(F(\xi),\xi_3, s_2\big)\, e^{-i\la \xi_3\phi^\pm(\xi)},\\
 &=:&R^{-1} \tilde{\tilde m}^\pm_{\la,R,I}(\xi),
\end{eqnarray*}
with as slightly modified cut-off $\tilde m_R.$
Thus 
$$
 \dot {m} _{\lambda,R}(\xi)=(\la R^{-1})\Big[ \tilde{\tilde m}^\pm_{\lambda,R,I}(\xi)+\tfrac R{\la} \,\widetilde{m_{\la,R,II}}(\xi)\Big]=(\la R^{-1})\, \tilde{\tilde m}_{\la,R}(\xi).
$$
Note that here $\la R^{-1}\gg 1,$ since $\la\ge R^\ga,$ with $\ga>1$ and $R\gg 1,$ so that  $T_{\tilde{\tilde m}^\pm_{\lambda,R}}$  satisfies  again an FIO-cone multiplier estimate of the form \eqref{conefiosest}.
 We may thus apply Theorem \ref{multipliertomaximal} here with $\la R^{-1}$ in place of $\la,$ and obtain
$$
\|\mathcal{M}_{m^\pm_{\la,k}} \|_{L^p\to L^p}\le C_{\epsilon,\delta,p}\,\left\{ (\la R^{-1})^{\frac 1p}\left( \la^{ \frac 2p-2} R^{-\frac{1+\ga}{p}+\frac 32+\epsilon}\right)^{1-\delta}+(\la R^{-1})^{\frac 1p}\log \la \,\la^{ \frac 2p-2} R^{-\frac{1+\ga}{p}+\frac 32+\epsilon} \right\},
$$
hence
 \begin{equation}\label{Mestitilde}
\|\mathcal{M}_{m^\pm_{\la,k}} \|_{L^p\to L^p} \le C_{\epsilon,p}\,R^{-\frac 1p} \la^{ \frac 3p-2+\delta} R^{-\frac{1+\ga}{p}+\frac 32+\delta},
\end{equation}
 for every $\delta>0.$
 
 Now assume that $p>3/2$ and that $\delta>0$ is sufficiently small. Then the exponent of $\la$ in this estimate is negative, and since $\la\ge R^\ga,$ we can estimate 
 $$
 R^{-\frac 1p} \la^{ \frac 3p-2+\delta} R^{-\frac{1+\ga}{p}+\frac 32+\delta}\le \la^{-\delta} R^{-\frac 1p} R^{ (\frac 3p-2+2\delta)\ga} R^{-\frac{1+\ga}{p}+\frac 32+\delta}=\la^{-\delta} R^{\frac 32-\frac 2p-2\ga (1-\frac 1p)+2\delta\ga}.
 $$
 Since $\ga>1,$ we have 
 $$
 \frac 32-\frac 2p-2\ga(1-\frac 1p)+2\delta\ga<  \frac 32-\frac 2p-2(1-\frac 1p)+2\delta\ga=-\frac 12+2 \delta\ga<0,
 $$
 and thus 
 $
 \|\mathcal{M}_{m^\pm_{\la,k}} \|_{L^p\to L^p} \le C_{\epsilon,p}\, \la^{-\delta}.
 $
  By means of the same arguments as before we thus obtain now that for $p>3/2$ 
\begin{equation}\label{sumMklaNT}
\sum\limits_{\la\gg 1} \sum\limits_{k=1}^{k(\la)}\|\mathcal{M}_{m^\pm_{\la,k}} \|_{L^p\to L^p}<\infty.
\end{equation}
The same kind of estimates (even better ones) hold  also for the maximal operators associated to the multipliers 
$m^E_{\la,k}$ and $m^0_\la.$ 

\smallskip
This proves Theorem \ref{theorem2.1} (b) also for the case of a root $\ga_\tau$ of Type  $\mathrm{B_{NT}}.$
\qed

\color{black}
%%%%%%%%%%%%%%%%%%%%%%%%%%%%%%%%%%%%%%%%%%%%%%%%%%%%%%%%%%%%%%%%%%%%%%%%%%%%%%%%%%%%%%%%%%%%%%%%%%%%

\section{Type C: Exactly one of $\partial_1\Phi$ and $\partial_2\Phi$ does not vanish along $\gamma_{\tau} $}\label{caseC}

In this section, we shall prove Theorem \ref{theorem2.1}(c). So assume here that $\tau\in R$ parametrizes a fixed real root $\gamma_\tau$ of Type $\mathrm{C}$, and let $\tilde m_\tau$ be the associated Fourier multiplier defined in \eqref{dmiudef}, for which $\tilde{\mathcal{M}}_{\tau}=\mathcal{M}_{\tilde m_\tau},$ i.e., 
$$
 \tilde m_\tau(\xi)=\int_{\mathbb{R}^{2}}e^{-i\bigl(\xi_1 x_1+\xi_2 x_2+\xi_3\Phi(x)\bigl)} \tilde \psi(x)\, \rho_\tau(x) \,\chi_1(x_1)\, dx,
$$
where we recall from \eqref{rhotau} that 
$\rho_\tau(x):=\chi_0\Big(\frac{x_2-\tau x_1^{a}}{\epsilon x_1^{a}}\Big).$

Since here  exactly one of $\Phi_{1}$ and $\Phi_{2}$ does not vanish along $\gamma_{\tau},$  as in Lemma \ref{Legendret} we may and shall assume without loss of generality   that $\Phi_2$ does not vanish   along 
$\gamma_{\tau},$ so that  
\begin{equation}\label{CPhi1Phi2}
	\Phi_{1}(x_{1},\tau  x_{1}^{\frac{\kappa_{2}}{\kappa_{1}}}) = 0, \quad \quad \Phi_{2}(x_{1},\tau  x_{1}^{\frac{\kappa_{2}}{\kappa_{1}}})  \neq 0 \qquad \text{for every} \quad x_1>0.
\end{equation}
Recall also from  Lemma \ref{tauinR'} that here necessarily   $T_{\tau} =0,$ hence 
\begin{equation}\label{CHPhiPhi}
	\mathrm{H}\Phi(x_{1}, \tau  x_{1}^{\frac{\kappa_{2}}{\kappa_{1}}}) =0, \quad \quad \Phi(x_{1},\tau  x_{1}^{\frac{\kappa_{2}}{\kappa_{1}}}) \neq 0 \qquad \text{for every} \quad x_1>0.
\end{equation}

Following the first steps of the proof of Theorem \ref{theorem2.1}(b),  by means of  a non-homogeneous dyadic decomposition in $|\xi|$ we may first reduce to estimating the maximal operator $\mathcal{M}_{ \tilde m_{\tau}^\la}$ associated to the Fourier multiplier $ \tilde m_\tau^\la$ defined by 
$$
\la:=2^j\gg 1 \quad\text{and} \quad \tilde m_{\tau}^\la(\xi):=\beta(\la^{-1}|\xi|) \,\tilde m_{\tau}(\xi).
$$
The full phase $\phi_\xi$ is here simply $\phi_{\xi}(x) :=\xi_1x_1+\xi_2 x_2+\xi_3\Phi(x),$ and for 
$x\in \ga_\tau$ we have 
$$
\partial_{x_1}\phi_{\xi}(x) =  \xi_1, \qquad \partial_{x_2}\phi_{\xi}(x) =  \xi_2+\xi_3\Phi_2(x);
$$
moreover,  we may assume that on the  support of $\rho_\tau(x) \,\chi_1(x_1)$ we have $|\Phi_1(x)|\ll 1 \sim |\Phi_2(x)|.$ 

In a first step, we shall here show that for all $\alpha\in \Bbb N^3$ and $N \in \mathbb{N},$
\begin{equation}\label{BrapiddecayC}
|\partial_{\xi}^{\alpha}\tilde m_\tau^\la(\xi)| \lesssim \frac{C_{N,\alpha}}{(1+|\xi|)^{N}}=\mathcal {O}(\la^{-N}),
\end{equation}
unless 
\begin{equation}\label{allsimC}
|\xi_1|\ll |\xi_2|\sim |\xi_3|.
\end{equation}

To this end, let $\xi=(\xi_1,\xi_2,\xi_3)\in\Bbb R^3$ with $|\xi|\sim \la$ be given. Choose again a permutation $(i_1,i_2,i_3)$ of $(1,2,3)$ such that 
$$
|\xi_{i_1}|\le |\xi_{i_2}|\le |\xi_{i_3}|, 
$$  so that in particular $|\xi|\sim |\xi_{i_3}|.$ 

Assume  first that  $i_3=1.$ Then  $|\partial_{x_{i_3}}\phi_{\xi}(x)| \sim |\xi_{i_3}|\sim |\xi|,$ and thus \eqref{BrapiddecayC} follows by integrations by parts in $x_{i_3}.$ 

Next, if $i_3=2,$ then integrations by parts in $x_{i_3}$  show that  \eqref{BrapiddecayC}  holds unless 
$|\xi_{i_3}|\sim |\xi_{3}|.$ 

In the latter case, assume we had  $i_1=3.$ Then $i_2=1,$ and $|\xi_i|\sim |\xi|\sim \la$ for all $i,$ in particular for $i=i_2,$ and integrations by parts in $x_{i_2}=x_1$ would yield \eqref{BrapiddecayC} as before.
So, assume that $i_1=1$ and $i_2=3.$  If we then had $|\xi_{i_1}|\gtrsim |\xi_{i_3}|,$ again integration by parts in $x_1=x_{i_1}$ would lead to \eqref{BrapiddecayC}. Thus, we are here left with the situation of \eqref{allsimC}.

Finally, assume that $i_3=3.$ Then by integrations by parts in $x_2$ we may reduce to the case where $|\xi_{i_3}|\sim |\xi_{2}|.$ If then $i_1=2,$ we are back in the situation where $|\xi_i|\sim |\xi|\sim \la$ for all $i.$ Thus there remains the case where $i_1=1, i_2=2$ and $i_3=3,$ and we again arrive at  \eqref{allsimC}.

Assume now that \eqref{allsimC} holds.  In a similar way as for roots of Type B, we can then again make use of Lemma \ref{Legendret}. As in this lemma, we shall assume without loss of generality that   $\Phi_2(x)<0$ for all $x$ in the support of $\rho_\tau(x) \,\chi_1(x_1).$ 
 
 If then $\xi_2/\xi_3<0,$ we would again find that $|\partial_{x_2}\phi_{\xi}(x)| =  |\xi_2+\xi_3\Phi_2(x)|\sim |\xi|,$ so that  integrations by parts in $x_2$  show that again  \eqref{BrapiddecayC} holds, which allows  us to reduce to the case where $\xi_2/\xi_3>0.$
 
By means of similar arguments as in Section \ref{caseB}, up to small error terms, we can thus reduce to estimating the maximal operator associated to the multiplier
$$
\tilde m_{\tau,1}^{\la}(\xi) := \chi_{0,\epsilon}\bigl(\frac{\xi_1}{\xi_3}\bigl) \chi_1\big(\frac{\xi_2}{\xi_3}\big)\, \tilde m_{\tau}^{\la}(\xi),
$$
where $\chi_{0,\epsilon}$ and $\chi_1$ are non-negative smooth cut-off functions such that $\chi_{0,\epsilon}$ is supported in a short interval $[-\epsilon, \epsilon],$ and $\chi_1$ in $ [1/2,2].$

Thus $\tilde m_{\tau,1}^{\la}$ is supported where $ |\xi_1| \ll |\xi_{2}| \sim |\xi_{3}|\sim \la$ and $\xi_{2}/\xi_{3} > 0. $

\medskip 

From here on, it will be more convenient to work with the {\it re-scaled} multiplier 
$$
m_{\tau,1}^{\la}(\xi):= \tilde m_{\tau,1}^{\la}(\la \xi).
$$
Using again the short-hand notation
\[s_1:= \frac{\xi_1}{\xi_3},  \ s_2:= \frac{\xi_2}{\xi_3}, \ s:=(s_1,s_2)
 \quad\text{(note that}\ |s_1|\ll 1\sim s_2,\, |\xi_3|\sim 1),
  \]
we may assume for simplicity that this multiplier then takes the form (compare \eqref{mla1B})
$$
 m_{\tau,1}^{\la}(\xi)=\chi_{0,\epsilon}(s_1)\chi_1(s_2)\tilde {\tilde{\chi}}_1(\xi_3) \, \int_{\mathbb{R}^{2}}
 e^{-i\la\xi_3\big(s_1 x_1+s_2 x_2+\Phi(x)\big)}\,  \tilde \psi(x)\,\rho_\tau(x)\,\chi_1(x_1)\, dx.
$$
 We shall also use the abbreviation
$$
\phi_{s}(x):=s_1 x_1+s_2 x_2+\Phi(x).
$$

Let  $x_0=(x_{0,1},x_{0,2})$ be the unique solution in $\gamma_\tau$ of the equation \eqref{x10} (cf. Lemma \ref{Legendret}), i.e., 
$$
x_0\in \gamma_\tau \quad \text{and}\quad  1+\Phi_2(x_0)=0.$$

We can now argue exactly in the same way as for roots of Type B and reduce considerations (up to small error terms)  by means of a suitable scaling argument and an application of the method of stationary phase to the integration in $x_2$ to the multiplier
\begin{eqnarray}\nonumber
 m_{\tau,2}^{\la}(\xi)&=&\la^{-\frac 12} \chi_{0,\epsilon}(s_1)\chi_1(s_2)\tilde {\tilde{\chi}}_1(\xi_3) \, \int_{\mathbb{R}}\, e^{-i\la \xi_3\big[s_1s_2^{\frac{\ka_1}{1-\ka_2}} x_1+s_2^{\frac{1}{1-\ka_2}}\Psi(x_1)\big]} \\
&&\hskip6cm  \times A_\la(\xi_3,s_2;x_1)\chi_{1,s_2}(x_1)\chi_0(\frac{x_1-x_{0,1}}\epsilon)\, dx_1,\label{mla1B4}
\end{eqnarray}
where again 
$$
\Psi(x_1) := x^{c}_{2}(x_{1}) + \Phi(x_{1},x^{c}_{2}(x_{1})), 
$$
and where 
  $A_\la(\xi_3,s_2;x_1)$ is smooth and so that  every partial derivative of order $k$ of $A_\la$ is uniformly bounded by a constant $C_k$ not depending on $\la$ in the  frequency region where   $s_2\sim1\sim|\xi_3|.$
  
Moreover, by Lemma \ref{Legendret}, 
\begin{equation*}%\label{Legendra}
	\Psi(x_{1}) = \Psi(x_{0,1}) + \Psi^{\prime}(x_{0,1})(x_{1}-x_{0,1}) + (x_{1}-x_{0,1})^{M}G(x_{1}), 
\end{equation*}
where  $G$ is an analytic function  with $G(x_{0,1}) \neq 0,$ and  where $M=M_\tau:=N_\tau+2\ge 3,$ if $N_\tau$ denotes the multiplicity of the root $\gamma_\tau$.  

Note, however, that by Euler's identity \eqref{Euler} and \eqref{CHPhiPhi} we have for any $x\in \gamma_\tau$ that
$\Phi(x)=\ka_2 x_2\Phi_2(x).$ For the point  $x_0\in \gamma_\tau,$ which according  to \eqref{x10} satisfies  $\Phi_2(x_0)=-1,$  we thus obtain  that here 
\begin{equation}\label{Gammax10C}
\Psi(x_{0,1})=x_{0,2} +\Phi(x_0)=x_{0,2} +\ka_2x_{0,2}\Phi_2(x_0)=(1-\ka_2)x_{0,2}\ne 0.
\end{equation} 
On the other hand, by \eqref{Gamma'} and \eqref{CPhi1Phi2} we now have that
\begin{equation}\label{Gammaprimeneq01C}
\Psi^{\prime}(x_{0,1}) = \Phi_{1}(x_0) = 0.
\end{equation}

Applying the coordinate change $y=x_1-x_{0,1},$ \eqref{mla1B4} can be re-written as
\begin{eqnarray*}\nonumber
 m_{\tau,2}^{\la}(\xi)&=&\la^{-\frac 12} \chi_{0,\epsilon}(s_1)\chi_1(s_2)\tilde {\tilde{\chi}}_1(\xi_3) \,  \int_{\mathbb{R}}\, e^{-i\la \xi_3\big[s_1s_2^{\frac{\ka_1}{1-\ka_2}} (x_{0,1}+y)+s_2^{\frac{1}{1-\ka_2}}\tilde \Psi(y)\big]} \\
&&\hskip6cm  \times \tilde A_\la(\xi_3,s_2;y) \chi_0(\frac{y}\epsilon)\, dy,
\end{eqnarray*}
where $\tilde A_\la(\xi_3,s_2;y):=\chi_{1,s_2}(y+x_{0,1})\,A_\la(\xi_3,s_2;x_{0,1}+y),$ and 
$$
\tilde \Psi(y) = \Psi(x_{0,1}) + y^{M}\tilde G(y), 
$$
with $\tilde G(0)\ne 0$ and $\Psi(x_{0,1})\ne 0.$
Thus 
\begin{equation}\label{mla1B3}
m_{\tau,2}^{\la}(\xi)=\la^{-\frac 12} \chi_{0,\epsilon}(s_1)\chi_1(s_2)\tilde {\tilde{\chi}}_1(\xi_3) \,  e^{-i\la\xi_3\Big[s_1 s_2^{\frac{\ka_1}{1-\ka_2}}x_{0,1}+s_2^{\frac{1}{1-\ka_2}} \Psi(x_{0,1})\Big]} \times  J(\la,\xi_3,s),
\end{equation}
where 
$$
J(\la,\xi_3,s): =\int_{\mathbb{R}}\, e^{-i\la \xi_3s_2^{\frac{1}{1-\ka_2}}\Big[s_1s_2^{\frac{\ka_1-1}{1-\ka_2}}y+y^M\tilde G(y)\Big]}\, 
\tilde A_\la(\xi_3,s_2;y)\chi_0(\frac{y}\epsilon)\, dy
$$
is again an Airy-type integral. Here we shall therefore  put
\[
u:=-s_{1}s_{2}^{\frac{\kappa_{1}-1}{1-\kappa_{2}}}\ \text{and} \quad  F(\xi):=s_1,\quad  \text{where}\ |s_{1}| \ll 1 \text{ and } |s_{2}| \sim 1,
\]
so that $J(\la,\xi_3,s)$ can again be re-written as in \eqref{Iju}, i.e., 
$$
  J(\la,\xi_3,s) =  \int_{\mathbb{R}}\, e^{-i\la \xi_3s_2^{\frac{1}{1-\ka_2}}\big[y^M\tilde G(y)-uy\big]}\, 
\tilde A_\la(\xi_3,s_2;y)\chi_0(\frac{y}\epsilon)\, dy.
$$
 
Then $J(\la,\xi_3,s)$ is as in \eqref{Iju}, and we decompose it as in \eqref{decomposeIju} into the  sum of the  terms $J_0(\la,\xi_3,s), J^{\pm}_{k}(\la,\xi_3,s)$ and $J^{E}_{k}(\la,\xi_3,s),$ where $2^{k(\la)}\sim \la^{\frac {M-1}M}.$ 

Accordingly, we  decompose the multiplier $m_{\tau,2}^{\la}$  as in \eqref{m2taudecomp}, i.e.,
$$
m_{\tau,2}^{\la}=m^{0}_{\lambda}+\sum_{\pm}\sum^{k(\lambda)}_{k=1}m^{\pm}_{\lambda,k}+\sum^{k(\lambda)}_{k=1}m^{E}_{\lambda,k},
$$
where the summands are defined as in \eqref{m0k}--\eqref{mEk}.

We shall  again separately estimate the maximal operators corresponding to each of these Fourier multipliers, focusing, however, mostly on the main terms given by the  multipliers $m^{\pm}_{\lambda,k}.$

\medskip

To this end {\it let us  fix $k$  and the corresponding multiplier $m^{\pm}_{\lambda,k},$} and introduce the following abbreviations: We put 
\begin{equation}\label{abbrevsC}
 \gamma := M/(M-1) ,\qquad R=R_k(\la):=2^{-k}\lambda^{\frac{M-1}{M}}.
 \end{equation}
Since  $M\ge 3,$  we may assume here  that 
\begin{equation}\label{abbrevs}
1<\ga <2, \quad R\gg 1 \quad  \text{and} \quad \la\ge R^\ga.
\end{equation}
Then as in \eqref{mpmlk2}  we can re-write 
$$
m^{\pm}_{\lambda,k}=\la^{-1} R^{\frac{M-2}{2M-2}}\,  \chi_{0,\epsilon}(s_1)\chi_1(s_2)\tilde {\tilde{\chi}}_1(\xi_3)\, 
 \chi_1(R|s_1|)\, a_{\la,R}^\pm\big(s_1,\xi_3, s_2\big)\, e^{-i\la \xi_3\phi^\pm(\xi)},
$$
with the phase function 
$$
\phi^\pm(\xi):=\Big[s_1 s_2^{\frac{\ka_1}{1-\ka_2}}x_{0,1}+s_2^{\frac{1}{1-\ka_2}} \Psi(x_{0,1})\Big]+
s_{2}^{\frac{\ka_1}{1-\kappa_{2}}}|s_1|^{\gamma} q_{\pm}\big(s_2^{\frac{(\gamma-1)(\ka_1-1)}{1-\ka_2}}|s_1|^{\gamma-1}\big),
$$
and a smooth  amplitude  $a_{\la,R}^\pm\big(s_1, \xi_3,s_2\big)$, where $a_{\la,R}^\pm$ satisfies estimates of the form 
$$
|\partial^\alpha_{s_1} \partial^\beta_{\xi_3} \partial^\rho_{s_2} a_{\la,R}^\pm(s_1, \xi_3,s_2)|\le C_{k,\beta} R^\alpha 
\quad \quad \forall \alpha, \beta, \rho\in  \mathbb{N}.
$$

Observe next that, in contrast to the case of roots of Type B, here the cone $\{F(\xi)=0\}$ is given by the flat hypersurface $\{\xi_1=0\}.$ This allows for a much easier and more classical treatment of the maximal operator $\mathcal{M}_{m^{\pm}_{\lambda,k}}$ than in the previous section.

Indeed,  we can here consider the {\it scaled} multiplier 
$$
\mathring m^{\pm}_{\lambda,k}(\xi_1,\xi_2,\xi_3):=m^{\pm}_{\lambda,k}(R^{-1}\xi_1,\xi_2,\xi_3)
$$
in place of $m^{\pm}_{\lambda,k},$ for which according to \eqref{mGLn2} we have $\|\mathcal {M}_{m^{\pm}_{\lambda,k}}\|_{L^p\to L^p}=\|\mathcal {M}_{\mathring m^{\pm}_{\lambda,k}}\|_{L^p\to L^p}.$ It takes the form
\begin{equation}\label{mscalC}
\mathring m^{\pm}_{\lambda,k}=\la^{-1} R^{\frac{M-2}{2M-2}}\, \chi_1(s_2)\tilde {\tilde{\chi}}_1(\xi_3)\, 
 \chi_1(|s_1|)\,  \chi_{0,\epsilon}(R^{-1}s_1)\, \mathring a_{\la,R}^\pm\big(s_1,\xi_3, s_2\big)\, e^{-i\la \xi_3\mathring\phi^\pm(\xi)},
\end{equation}
so that  after the scaling now $|s_1|\sim s_2\sim 1.$ 
The  phase is given by 
\begin{equation}\label{phipmC}
\mathring\phi^\pm(\xi):=\Big[R^{-1}s_1 s_2^{\frac{\ka_1}{1-\ka_2}}x_{0,1}+s_2^{\frac{1}{1-\ka_2}} \Psi(x_{0,1})\Big]+
R^{-\ga}s_{2}^{\frac{\ka_1}{1-\kappa_{2}}}|s_1|^{\gamma} q_{\pm}\big(R^{-(\ga-1)}s_2^{\frac{(\gamma-1)(\ka_1-1)}{1-\ka_2}}|s_1|^{\gamma-1}\big)
\end{equation}
and  the ``amplitude'' by 
\begin{equation}\label{ampestC}
\mathring a_{\la,R}^\pm\big(s_1,\xi_3, s_2\big):= a_{\la,R}^\pm\big(R^{-1}s_1,\xi_3, s_2\big).
\end{equation}
Note that by the preceding estimates the latter satisfies
\begin{equation}\label{ampderivC}
|\partial^\alpha_{s_1} \partial^\beta_{\xi_3} \partial^\rho_{s_2} \mathring a_{\la,R}^\pm(s_1, \xi_3,s_2)|\le C_{k,\beta} 
\quad \quad \forall \alpha, \beta, \rho\in  \mathbb{N}.
\end{equation}
\smallskip

Again we shall first estimate the multiplier operator $T_{\mathring m^{\pm}_{\lambda,k}}$ on $L^p.$ 
By Plancherel's theorem, \eqref{mscalC} implies that
\begin{equation}\label{CL21}
 \|T_{\mathring m_{\lambda,k}^{\pm}}\|_{L^{2}\rightarrow L^{2}} \le \|\mathring m_{\lambda,k}^{\pm}\|_{L^{\infty}} 
 \lesssim \la^{-1} R^{\frac{M-2}{2M-2}}=\lambda^{-(\frac{1}{2}+\frac{1}{M})} 2^{-\frac{M-2}{2M-2}k}.
 \end{equation}
 
 And, in view of \eqref{ampderivC} and \eqref{phipmC}, we can apply  Lemma \ref{lemma3.2}, with coordinates 
 $(\eta_1,\eta_2,\eta_3)=(\xi_2,\xi_1,\xi_3),$ $\sigma = R^{-\gamma}$ and $\rho=1/\ga\ge 1/2$ (since $\gamma<2$),
 and obtain that
 $$
\|\mathcal{F}^{-1}\mathring m_{\lambda,k}^{\pm}\|_{L^{1}} \lesssim  \big(\lambda^{-(\frac{1}{2}+\frac{1}{M})} 2^{-\frac{M-2}{2M-2}k}\big)\lambda^{\frac{1}{2}}(\lambda \sigma)^{\frac{1}{2}}
=\big(\lambda^{-(\frac{1}{2}+\frac{1}{M})} 2^{-\frac{M-2}{2M-2}k}\big)\la^{\frac 12} 2^{\frac M{2M-2}k}=\lambda^{  -\frac{1}{M}} 2^{\frac{k}{M-1}}.
$$
Thus, by Young's inequality,
\begin{equation}\label{CL11}
  \|T_{\mathring m_{\lambda,k}^{\pm}}\|_{L^{1}\rightarrow L^{1}} \le \|\mathcal{F}^{-1}\mathring m_{\lambda,k}^{\pm}\|_{L^{1}} \lesssim \lambda^{  -\frac{1}{M}} 2^{\frac{k}{M-1}}.
  \end{equation}
  Interpolating these two estimates, we obtain for $1\le p\le 2$ that
 \begin{equation}\label{CLp1}
  \|T_{\mathring m_{\lambda,k}^{\pm}}\|_{L^{p}\rightarrow L^{p}}  \lesssim \lambda^{(\frac{1}{p} -\frac{M+1}{M})} 2^{k(\frac{M}{M-1}\frac{1}{p}-1)}.
 \end{equation}
Again we want to apply Theorem \ref{multipliertomaximal} in order to also bound the maximal operator 
$\mathcal{M}_{\mathring m_{\lambda,k}^{\pm}}.$ 

To this end, we first  observe  that  from  \eqref{mscalC} it is easy  to see that the multiplier 
$$
\dot {m}_{\lambda,k}^{\pm}:=\frac {d}{dt} \mathring m_{\lambda,k}^{\pm}(t\xi)\big|_{t=1},
$$ 
as defined in Theorem \ref{multipliertomaximal}, is of the form $\la \, \widetilde{\mathring m_{\lambda,k}^{\pm}},$ where 
$\widetilde{\mathring m_{\lambda,k}^{\pm}}$ is of the same form \eqref{mscalC} as  $\mathring m_{\lambda,k}^{\pm},$ only with a slightly modified  ``amplitude''.

There remains thus  to check condition (ii) in this theorem. As in Section \ref{caseAi},  by 
Lemma  \ref{generalkernelestimate}  it will suffice for our purposes   to prove  that 
\[
\bigl|\partial ^{\alpha}\mathring m_{\lambda,k}^{\pm}(\xi)\bigl| \lesssim \lambda^{|\alpha|}.
\]
However, this follows is immediately from \eqref{mscalC}, and an analogous argument applies to $\widetilde{\mathring m_{\lambda,k}^{\pm}}.$

\smallskip
We are now in a position to apply  Theorem \ref{multipliertomaximal}, with constants
\[
A = \lambda^{\frac{1}{p} -\frac{M+1}{M}} 2^{k(\frac{M}{M-1}\frac{1}{p}-1)}, \quad \quad B =\lambda^{C}
\]
(note that when $3/2\le p\le 2$, then  $A \lesssim 1,$ since $M\ge 3$).
We thus  obtain that for every $\delta>0$
\begin{eqnarray*}
 \| \mathcal{M}_{m_{\lambda,k}^{\pm}}\|_{L^{p} \rightarrow L^{p}} &=& \| \mathcal{M}_{\mathring m_{\lambda,k}^{\pm}}\|_{L^{p} \rightarrow L^{p}}\\
 &\le& C_{\delta} \biggl\{ \lambda^{\frac{1}{p}} \biggl( \lambda^{\frac{1}{p} -\frac{M+1}{M}} 2^{k(\frac{M}{M-1}\frac{1}{p}-1)} \biggl)^{1-\delta}  + \lambda^{\frac{1}{p}} (\log\lambda) \lambda^{\frac{1}{p} -\frac{M+1}{M}} 2^{k(\frac{M}{M-1}\frac{1}{p}-1)} \biggl\}.
\end{eqnarray*}

By choosing $\delta$  sufficiently small, we see that if $3/2\le p\le 2,$ then
\begin{equation}\label{Mla0estp}
 \|  \mathcal{M}_{m_{\lambda,k}^{\pm}}\|_{L^{p} \rightarrow L^{p}} \le C_{p,\delta} \lambda^{\frac{2}{p} -\frac{M+1}{M} + \delta} 2^{k(\frac{M}{M-1}\frac{1}{p}-1 +\delta)}.
\end{equation}
For $p > \frac{2M}{M+1}$, by choosing $\delta >0$ sufficiently small, we have
$\frac{2}{p} -\frac{M+1}{M} + \delta <0$ and $\frac{M}{M-1}\frac{1}{p}-1 +\delta< 0,$ since $M \ge 3$. Summing over all $k$ and all  dyadic numbers $\lambda\gg 1$, we obtain that
\[\sum_{\lambda \gg 1} \sum_{k=1}^{k(\lambda)} \|  \mathcal{M}_{m_{\lambda,k}^{\pm}}\|_{L^{p} \rightarrow L^{p}} \lesssim 1. \]

\medskip
Next  we estimate $\mathcal{M}_{m_{\lambda}^{0}}$. Recall from \eqref{I0} and \eqref{m0k} that
$$
m^{0}_{\lambda}(\xi)= \la^{-(\frac 12+\frac 1M)} \chi_{0}(R|s_1|)\, \tilde{\chi}_1(s_1)\chi_1(s_2)\tilde {\tilde{\chi}}_1(\xi_3) g(s_1,\xi_3, s_{2})\, e^{-i\la\xi_3\Big[s_1 s_2^{\frac{\ka_1}{1-\ka_2}}x_{0,1}+s_2^{\frac{1}{1-\ka_2}} \Psi(x_{0,1})\Big]},
$$
where we have put $R:=\lambda^{\frac{M-1}{M}}$ (which formally agrees with \eqref{abbrevsC} if we choose $k=0$).
Passing to the scaled  multiplier $\mathring m^{0}_{\lambda}(\xi_1,\xi_2,\xi_3):=m^{0}_{\lambda}(R^{-1}\xi_1,\xi_2,\xi_3),$
we see that 
\begin{eqnarray*}
\mathring m^{0}_{\lambda}(\xi)&=& \la^{-(\frac 12+\frac 1M)} \chi_{0}(|s_1|)\, \tilde{\chi}_1(R^{-1}s_1)\chi_1(s_2)\tilde {\tilde{\chi}}_1(\xi_3) \,g(R^{-1}s_1,\xi_3, s_{2})\\
&&\hskip3cm \times e^{-i\la\xi_3\Big[R^{-1}s_1 s_2^{\frac{\ka_1}{1-\ka_2}}x_{0,1}+s_2^{\frac{1}{1-\ka_2}} \Psi(x_{0,1})\Big]}.
\end{eqnarray*}
Apparently the amplitude and the (slightly simpler) phase behave very much in the same way as the ones for 
$\mathring m_{\lambda,k}^{\pm},$ if   we formally put $k=0.$  Consequently, the same kind of estimates hold for the maximal operator $\mathcal {M}_{m^{0}_{\lambda}}$ as for $\mathcal{M}_{m_{\lambda,k}^{\pm}}$ when choosing $k=0$ in 
\eqref {Mla0estp}, i.e., 
$$
 \|\mathcal {M}_{m^{0}_{\lambda}} \|_{L^{p} \rightarrow L^{p}} \le C_{p,\delta} \lambda^{\frac{2}{p} -\frac{M+1}{M} + \delta}.
$$
For  $p > \frac{2M}{M+1}$, this allows again  to sum these estimates over  all dyadic numbers $\la\gg 1.$

\medskip
Finally, we consider the maximal operators $\mathcal{M}_{m_{\lambda,k}^{E}}.$ If $R$ is defined as in \eqref{abbrevsC}, then  \eqref{IE} and \eqref{mEk} show that 
$$
m^{E}_{\lambda,k}(\xi)=\la^{-\frac 32} \chi_{1}\biggl(R|s_1| \biggl) \tilde{\chi}_1(s_1)\chi_1(s_2)\tilde {\tilde{\chi}}_1(\xi_3) 
|s_1|^{-1} E_\la(|s_1|, \xi_3,s_2)\, e^{-i\la\xi_3\Big[s_1 s_2^{\frac{\ka_1}{1-\ka_2}}x_{0,1}+s_2^{\frac{1}{1-\ka_2}} \Psi(x_{0,1})\Big]},
$$
where $E_\la$ satisfies estimates of the form \eqref{symbol2}, which are stronger than the ones in \eqref{symbol1}.
The re-scaled multiplier $\mathring m^{E}_{\lambda,k}(\xi_1,\xi_2,\xi_3):=m^{E}_{\lambda,k}(R^{-1}\xi_1,\xi_2,\xi_3)$
is then of the form
$$
\mathring m^{E}_{\lambda,k}(\xi)=\la^{-\frac 32} R \tilde \chi_{1}(s_1) \tilde{\chi}_1(R^{-1}s_1)\chi_1(s_2)\tilde {\tilde{\chi}}_1(\xi_3) 
E_\la(R^{-1}|s_1|, \xi_3,s_2)\, e^{-i\la\xi_3\Big[R^{-1}s_1 s_2^{\frac{\ka_1}{1-\ka_2}}x_{0,1}+s_2^{\frac{1}{1-\ka_2}} \Psi(x_{0,1})\Big]}.
$$

But note that the inequality  $\la\ge R^\ga$ in \eqref{abbrevs} is equivalent to the inequality 
$$
\la^{-\frac 32} R\le \la^{-1} R^{\frac{M-2}{2M-2}}.
$$
A comparison with the expression for $\mathring m^{\pm}_{\lambda}$ in \eqref{mscalC} thus shows that the amplitude of 
$\mathring m^{E}_{\lambda,k}$ satisfies the same kind of estimates as the ones for $\mathring m^{\pm}_{\lambda,k}$
(potentially even stronger ones). Thus by following the same scheme of proof as devised for $\mathring m^{\pm}_{\lambda,k}$,  we see that $\mathcal{M}_{m_{\lambda,k}^{E}}$  satisfies   the same kind of  estimates 
$$
\| \mathcal{M}_{m_{\lambda,k}^{E}}\|_{L^{p} \rightarrow L^{p}} \le C_{p,\delta} \lambda^{\frac{2}{p} -\frac{M+1}{M} + \delta} 2^{k(\frac{M}{M-1}\frac{1}{p}-1 +\delta)}
$$
 as $\mathcal{M}_{m_{\lambda,k}^{\pm}}$ in \eqref{Mla0estp} (potentially even stronger ones). Thus for $p > \frac{2M}{M+1}$ we also obtain that 
 \[
 \sum_{\lambda \gg 1} \sum_{k=1}^{k(\lambda)} \|  \mathcal{M}_{m_{\lambda,k}^{E}}\|_{L^{p} \rightarrow L^{p}} \lesssim 1. \]
In combination, all these estimates then  complete the proof of Theorem \ref{theorem2.1}(c)  by interpolation with the trivial $L^\infty$-bounds.
\qed

\section{Appendix: Control by Guth-Wang-Zhang seminorms}\label{GWZcont}
We explain the seminorms $\|\cdot\|_{GWZ}$ mentioned in \eqref{gwzest} and sketch how to prove that inequality by the results in \cite{GLMX23}.

Let  $s$ be a dyadic parameter in the range $R^{-1/2}\leq s\lesssim 1$. The standard decomposition of $\Gamma_R^\pm$ can be done by decomposing $\RR^2$ into sectors $\theta$ of angular width $R^{-1/2}$. We group the sectors $\theta$ into sectors $\tau$ of angular width  $s$ (also called the aperture $s=d(\tau)$ of $\tau$  in  \cite{GWZ20}) in the $(\xi_1,\xi_2)$-plane. 
For instance, if $s=R^{-1/2}$, each sector $\tau$ simply is a single cap $\theta$. And, if $s\sim1$, $\tau$ consists essentially of all caps $\theta$.
Let us denote by
$$ \tilde\theta:= \{\xi\in\Gamma^\pm_R:(\xi_1,\xi_2)\in \theta\}$$
the lift of $\theta$ to the thickened cone $\Gamma^\pm_R$.

%Recall that each cap $\theta$ is essentially a box of dimensions $R^{-1}\times R^{-1/2}\times 1$ with respect to the coordinates 
%$\eta$ associated to the  orthonormal frame 
%$$E_1=n(\xi_\theta),\quad E_2=t(\xi_\theta),\quad E_3=\xi_\theta/|\xi_\theta|,$$
%if $\xi_\theta$ denotes the center of the cap $\theta.$ 
The dual convex body $\theta^*$ to the convex hull of $\tilde\theta$ is then a rectangular box of dimensions $R\times R^{1/2}\times 1$, with orientation depending on $\theta$.
%Note that if  $\chi_\theta$ is  a smooth bump function adapted to the box $\theta$, then, by the uncertainty principle,
%$\widehat{\chi_\theta}$ is essentially supported in $\theta^*$.
%More precisely: if, in the coordinates $\zeta$ given by  $\xi=\xi_\theta+E\zeta$, $\chi_\theta$ is of the form 
%$$
%\chi_\theta(\xi)=\chi_0(R\zeta_1)\chi_0(R^{1/2}\zeta_2)\chi_0(\zeta_3), 
%$$
%then by integration by parts, $\widehat{\chi_\theta}$ decays rapidly far away from $\theta^*$.
%\smallskip
Of course, the orientation of the boxes $\theta^*$ differs for different $\theta$ in $\tau,$ but there is a ``bounding box''
$U_{\tau,R}$ containing all these dual boxes $\theta^*$,  which is essentially the convex hull of the union of all these $\theta^*,$ i.e., 
$$
U_{\tau,R} \simeq \text{Convex Hull} (\bigcup_{\theta\subset\tau} \theta^*).
$$
There are different ways to describe $U_{\tau,R}$, compare \cite{GWZ20} and \cite{GLMX23}.
%If $\xi_\tau\in \Gamma_R$ denotes the center of $\tau,$  then in the  coordinates $x$ corresponding to the orthonormal frame at $\xi_\tau$ this box is given by
%$$
%U_{\tau,R} :=\{x\in\RR^3: |x_1|\le R, |x_2|\le Rs \text{ and } |x_3|\le Rs^2\}.
%$$
Let us again look at the   two extremal cases: For  $s=R^{-1/2}$, $U_{\tau,R}$ consists of a single dual box $\theta^*,$ which is indeed a box of dimensions $R\times R^{1/2}\times 1$. And, for $s=1$, $U_{\tau,R}$ is a cube of side $R$.
Now by $\{U\}_{U// U_{\tau,R}}$, we denote a tiling of $\RR^3$ by boxes $U$ which are translates of $U_{\tau,R}$ and only overlap on a set of measure zero. 
Define the square function over $\tau$ by 
$$S_\tau (F)=(\sum_{\theta\subset\tau} |F_\theta|^2)^\frac12$$ 
(recall that $\hat F_\theta=\hat F\chi_{\tilde\theta}$, where $\chi_{\tilde\theta}$ is a smooth bump function adapted to $\tilde\theta$) and the following semi-norms:
\begin{eqnarray*}\label{GWZs}
	\|F\|_{GWZ,s}&:=& \left(\sum_{\angle(\tau)=s} \sum_{U// U_{\tau,R}} |U|^{-1}\big \|S_\tau(F) \big \|_{L^2(U)}^4 \right)^\frac14 \\
	\|F\|_{GWZ}&:=& \sum_{R^{-1/2}\le s\le 1}\|F\|_{GWZ,s}.
\end{eqnarray*}
Here, $\sum_{\angle(\tau)=s}$ means summation  over a decomposition into sectors $\tau$ of angular width  $s$ with measure zero overlap, $\sum_{R^{-1/2}\le s\le 1}$  means summation over all dyadic values of $s$ in the range $R^{-1/2}\leq s\leq 1$.

In \cite{GLMX23}, the authors essentially show that \eqref{gwzest} holds for functions given by oscillatory integral operators 
\begin{equation}\label{extension}
 F(x,t) = \mathcal{F}^R(f)(x,t) =\int e^{i\phi^R(x,t,\eta)} a^R(x,t,\eta) \hat f(\eta)  d\eta,
\end{equation}
which includes our phase $\phi^R(x,t,\eta)=x\cdot\eta+th(\eta)$, and $a^R$ is a suitable amplitude function. Here, we may and will choose $a^R(x,t,\eta)= \chi((x-x_0,t-t_0)/R)\chi_1(\eta)$, localising to a ball $B_R$ of radius $R$ and centre $(x_0,t_0)$ and to an annulus $1\leq|\eta|\leq2$ on the frequency side. Since any translation of $F$ only leads to a modulation on the Fourier side which leaves the Fourier support invariant, we may assume without loss of generality that the ball of radius $R$ is centred at the origin.
 
 The Fourier support of such functions $F$ is indeed essentially contained in a $1/R$-neighbourhood of the cone, but not any function $G$ with such Fourier support is of the specific form \eqref{extension}. However, we can write $G$ as an average of such functions. Such averaging arguments are frequently used  in restriction and decoupling theory; since the seminorms $\|\cdot\|_{GWZ}$ are quite new and elaborate, we feel it warranted to sketch an argument for this situation.

It is a short exercise to check that
\begin{equation}
	G(x,t) = \int_\RR F^s(x,t+s) R^{-1}\chi(R^{-1}(t+s)) ds 
\end{equation}
for all $(x,t)\in B_R$, where $\chi$ is a smooth bump function and
\begin{equation}
	F^s = \mathcal{F}^R(f^s),\qquad\qquad f^s(x)=G(x,s).
\end{equation}
The precise estimate in \cite{GLMX23} is 
\def\thesum{\sum\limits_{R^{-1/2}\leq s\leq 1} \sum\limits_{\substack{\tau: \\ d(\tau)=s}} \sum\limits_{\substack{U//U_{\tau,R} \\ U\subset B_{2R} }} |U|^{-1} }

\begin{equation}
	\| \mathcal{F}^R(f)\|_{L^4(B_R)}^4 \lesssim R^{\epsilon} \thesum \| S_\tau(\mathcal{F}^R(f))\|_{L^2(U)}^4 + \mathcal{O}(R^{-N}) \|f\|_{L^2(B'_R)}^4. 
\end{equation}
Here, $B_{2R}$ is the ball with the same centre as $B_R$ and double the radius, and $B'_{R}$ its projection to $\RR^2$.
Actually, they have some form of rapid decay instead of the à priori stronger statement $U\subset B_{2R}$, but it is easy to see that the sum over all $U$ not contained in $B_{2R}$ can be absorbed into the error term.

We claim that for all $|s|\leq R$
\begin{eqnarray*}
	\sum_U \| S_\tau(F^s)\|_{L^2(U)}^4
	\lesssim \sum_U \|S_\tau( G)\|_{L^2(U)}^4 +\mathcal{O}(R^{-N})  \|f^s\|_{L^2(B'_R)}.
\end{eqnarray*}
The reason for this is that $F_\theta^s$ is morally constant on the tubes that are the the Fourier duals of $\tilde\theta$ and can be seen using a wave package analysis. We provide more details:\\
Write 
$$U_t=\{x:(x,t)\in U\}$$
for the $t$-section of $U$. Since the orientation of $U$ is determined by the caps $\theta$ in $\tau$, we can write
$$U=U_0\times\{0\}+[-R,R](v_\tau,1)$$
for a vector $(v_\tau,1)$ depending on $\tau$. Further,
$$ U_t=U_0+tv_\tau.$$
Let $$f^s_{\theta,U}(x):= f^s_\theta(x)\chi_{U_0}(x)$$
where $\chi_{U_0}$ is a smooth bump function adapted to $U_0$ so that
$$f^s_\theta=\sum_U f^s_{\theta,U}.$$
Note that by construction of the sets $U$, 
$$\mathop{supp} \hat{f}^s_{\theta,U} \subset 2\theta.$$
Since
\begin{eqnarray*}
	 F^s_{\theta,U}(x,t)&:=&\mathcal{F}^R(f^s_{\theta,U})(x,t)  \\
	&=&\int e^{i(x\cdot\eta+th(\eta))} \hat f^s_{\theta,U}(\eta) \chi_1(\eta)d\eta, \\ 
	&=&\int\int e^{i([x-y]\cdot\eta+th(\eta))} \chi_1(\eta) d\eta\ f^s_{\theta,U}(y) dy,
\end{eqnarray*}
an integration by parts argument shows that $F^s_{\theta,U}$ decays rapidly for $(x,t)\not\in 2U$. Hence
\begin{eqnarray*}
	\| S_\tau(F^s)\|_{L^2(U)}^2 
	=\sum_{\theta\subset\tau} \| F_\theta^s\|_{L^2(U)}^2
	\lesssim\sum_{\theta\subset\tau}\sum_{U'\subset 2U} \|F_{\theta,U'}^s\|_{L^2}^2 +\mathcal{O}(R^{-N})  \|f^s\|_{L^2(B'_R)}, 
\end{eqnarray*}
and for any $t\in[-R,R]$ we have
\begin{eqnarray*}
	\|F_{\theta,U'}^s(\cdot,t)\|_{L^2}^2 
	&=&\int |\int e^{i(x\eta+th(\eta))} \hat f^s_{\theta,U'}(\eta) \chi_1(\eta) d\eta|^2 dx \\
	&=&\int |e^{ith(\eta)} \hat f^s_{\theta,U'}(\eta) \chi_1(\eta) |^2 d\eta \\ 
	&\leq&\int |f^s_{\theta,U'}(x) |^2 dx  \\
	&=& \int_{U'_0} |f^s_{\theta}(x) |^2 dx\ =\ \| G_\theta(\cdot,s)\|_{L^2(U_0')}^2  ,
\end{eqnarray*}
by Plancherel's theorem, so that
\begin{eqnarray*}
	\sum_U |U|^{-1} \| S_\tau(F^s)\|_{L^2(U)}^4  
	&\lesssim& \sum_{U'} |U'|^{-1} \left(R \sum_{\theta\subset\tau} \|G_\theta(\cdot,s)\|_{L^2(U'_0)}^2 \right)^2 
	 +\mathcal{O}(R^{-N})  \|f^s\|_{L^2(B'_R)} \\
	&\lesssim& \sum_{U} |U|^{-1} \left(R \sum_{\theta\subset\tau} \|G_\theta(\cdot,s)\|_{L^2(U_0-sv_\tau)}^2 \right)^2 
	+\mathcal{O}(R^{-N})  \|f^s\|_{L^2(B'_R)} .
\end{eqnarray*}
Since $G_{\theta}$  has Fourier support in $\theta$, $G_{\theta}$ behaves essentially constant on $\theta^*$ and translates; decomposing $U$ into translates of $\theta^*$ (which is feasable by construction of $U_{\tau,R}$)
$$ R \|G_\theta(\cdot,s)\|_{L^2(U_0-sv_\tau)}^2 
= R \|G_\theta(\cdot,s)\|_{L^2(U_s)}^2 \approx
   \|G_\theta\|_{L^2(U)}^2$$
 and the claim follows.

 For our general function $G$, we conclude
\begin{eqnarray*}
	\|G\|_{L^4(B_R)} &\lesssim& 
	R^{-1} \int_\RR  \| F^s \|_{L^4(B_R)} (1+|s|/R)^{-N} ds \\
	&\lesssim&  R^{-1}\int_\RR  \left(\thesum \| S_\tau(F^s)\|_{L^2(U)}^4 \right)^\frac14\ (1+|s|/R)^{-N} ds \\
	&& \qquad\qquad\qquad + 	R^{-1} \int_\RR \mathcal{O}(R^{-N})  \|f^s\|_{L^2(B'_R)}  (1+|s|/R)^{-N} ds \\
	&\lesssim& \left(\thesum \| S_\tau(G)\|_{L^2(U)}^4 \right)^\frac14 \\
	&& \qquad\qquad\qquad +\mathcal{O}(R^{-N'}) 
	\left(\int_\RR \|G(\cdot,s)\|_{L^4(B'_R)}^4 (1+|s|/R)^{-N} ds\right)^\frac14.
\end{eqnarray*}
Summing over a finite overlapping set of balls $B_R$ covering $\RR^3$ and bootstraping the error, we obtain 
\begin{eqnarray*}
	\|G\|_{L^4(\RR^3)} &\lesssim& 
	\left(\sum\limits_{R^{-1/2}\leq s\leq 1} \sum\limits_{\substack{\tau: \\ d(\tau)=s}} \sum\limits_{U//U_{\tau,R}} |U|^{-1} \| S_\tau(G)\|_{L^2(U)}^4 \right)^\frac14 
\end{eqnarray*}
for all functions $G$ with Fourier support in $\Gamma_R^\pm$.

 %%%%%%%%%%%%%%%%%%%%%%%%%%%%%%%%%%%%%%%%%%%%%%%%%%%%%%%%%%%%%%%%%%%%%%
%%%%%%%%%%%%%%%  References %%%%%%%%%%%%%%%%%%%%%%%%%%%%%%%%%%%%%%%%%%%%is
%%%%%%%%%%%%%%%%%%%%%%%%%%%g%%%%%%%%%%%%%%%%%%%%%%%%%%%%

\begin{flushleft}
\vspace{0.3cm}\textsc{Stefan Buschenhenke\\
Mathematisches Seminar, C.A.-Universit\"{a}t zu Kiel, 24118, Kiel, Germany\\
Email address:}  buschenhenke@math.uni-kiel.de

\vspace{0.3cm}\textsc{Wenjuan Li\\
	School of Mathematics and Statistics, Northwestern Polytechnical University, 710129, Xi'an, People's Republic of China\\
	Email address:} liwj@nwpu.edu.cn

\vspace{0.3cm}\textsc{Detlef M\"{u}ller\\
	Mathematisches Seminar, C.A.-Universit\"{a}t zu Kiel, 24118, Kiel, Germany\\
	Email address:} mueller@math.uni-kiel.de

\vspace{0.3cm}\textsc{Huiju Wang\\
	School of Mathematics and Statistics, Henan University, 475000, Kaifeng, People's Republic of China\\
	Email address:} huijuwang@mail.nwpu.edu.cn
\end{flushleft}


\begin{thebibliography}{99}
	
	\bibitem{bourgain85} J. Bourgain,
	Estimations de certaines fonctions maximales,
	{\it C. R. Acad. Sci. Paris S\'er. I Math.}, \textbf{301(2)} (1985), 499--502.
	
	\bibitem{Bo86} J. Bourgain, Averages in the plane over convex curves and maximal operators, {\it J. Anal. Math.}, {\bf47} (1986), 69--85.
	
	\bibitem{BNW}  J. Bruna, A. Nagel, and S. Wainger, Convex hypersurfaces and Fourier transforms, {\it Ann. of Math.}, {\bf127} (1988), 333--365.
	
	\bibitem{BDIM19}  S. Buschenhenke, S. Dendrinos, I. A. Ikromov, and D. M\"{u}ller, Estimates for maximal functions associated to hypersurfaces in $\mathbb{R}^3$ with height h $<$ 2: Part I, {\it Trans. Amer. math. Soc.}, {\bf372} (2019), 1363--1406.
	
	\bibitem{BIM25}  S. Buschenhenke,  I. A. Ikromov, and D. M\"{u}ller, Estimates for maximal functions associated to hypersurfaces in $\mathbb{R}^3$ with height h $<$ 2: Part II: A geometric conjecture and its proof for generic 2-surfaces, {\it Ann. Sc. Norm. Super. Pisa Cl. di Sc.}, \textbf{XXVI(5)} (2025), 1765-1877.
	
\bibitem{BDIM25}  S. Buschenhenke, S. Dendrinos, I. A. Ikromov, and D. M\"{u}ller, $L^p$ estimates for FIO-cone multipliers,  arXiv:2511.05243, 2025.

	\bibitem{CM86} M. G. Cowling and G. Mauceri, Inequalities for some maximal functions, II, {\it Trans. Amer. Math. Soc.}, \textbf{296} (1986), 341--365.
	
	
	\bibitem{DZ19}S. Dendrinos  and E. Zimmermann,
	On $L^{p}$-improving for averages associated to mixed homogeneous
	polynomial hypersurfaces in $\mathbb{R}^{3}$, {\it J. Anal. Math.}, \textbf{138(2)} (2019), 563--595.
	
	\bibitem{F86}J. L. Rubio de Francia, Maximal functions and Fourier transforms, {\it Duke Math. J.}, \textbf{53} (1986), 395-404.
	
	
\bibitem{GWZ20} L. Guth, H. Wang, and R. Zhang,
 A sharp square function estimate for the cone in $\mathbb{R}^{3}$, \textit{Ann. of Math.}, \textbf{192(2)} (2020), 551-581.

\bibitem{GLMX23} C. Gao, B. Liu,  C. Miao and Y. Xi, Square function estimates and local smoothing for Fourier integral operators. \textit{Proc. London Math. Soc.}, \textbf{126} (2023), 1923-1960.

\bibitem{Gr13} M. Greenblatt, $L^{p}$-boundedness of maximal averages over hypersurfaces in $\mathbb{R}^{3}$, {\it Trans. Amer. Math. Soc.}, \textbf{365} (2013), 1875--1900.	
	
	
	\bibitem{Gr81}  A. Greenleaf, Principal curvature and harmonic analysis, {\it Indiana Univ. Math. J.}, {\bf30} (1981), 519--537.
	
	\bibitem{IKM05}  I. A. Ikromov, M. Kempe, and D. M\"{u}ller, Damped oscillatory integrals and boundedness of maximal operators associated to mixed homogenous hypersurfaces, {\it Duke Math. J.}, {\bf126} (2005), 471--490.
	
	\bibitem{IKM10}  I. A. Ikromov, M. Kempe, and D. M\"{u}ller, Estimates for maximal functions associated with hypersurfaces in $\mathbb{R}^3$ and related problems of harmonic analysis, {\it Acta Math.}, {\bf204} (2010), 151--271.
	
	\bibitem{IM11} I. A. Ikromov  and D. M\"{u}ller, On adapted coordinate systems, {\it Trans. Amer. Math. Soc.}, {\bf 363} (2011), 2821--2848.
	
	
	\bibitem{IM16}  I. A. Ikromov  and D. M\"{u}ller, Fourier restriction for hypersurfaces in three dimensions and Newton polyhedra; Annals of Mathematics Studies {\bf194}, Princeton University Press, Princeton and Oxford 2016, 260 pages.
	
	\bibitem{Io94}  A. Iosevich, Maximal operators associated to families of flat curves in the plane, {\it Duke Math. J.}, {\bf76} (1994), 633--644.
	
	\bibitem{IS97}  A. Iosevich and E. Sawyer, Maximal averages over hypersurfaces, {\it Adv. Math.}, {\bf132} (1997), 46--119.
	
	\bibitem{ISS}  A. Iosevich. E. Sawyer, and A. Seeger, On averaging operators associated with convex hypersurfaces of finite type, {\it J. Anal. Math.}, {\bf79} (1999), 159--187.
	
	
	%\bibitem{LLOO} J. Lee, J. Lee, J. Oh and S. Oh, Maximal averages and non-transversality,  {\it arXiv: 2601.01880}, (2026).
	
	
	
	\bibitem{L}W. Li, Maximal functions associated with non-isotropic dilations of hypersurfaces in $\mathbb{R}^{3}$, {\it J. Math. Pures Appl.}, \textbf{113} (2018), 70--140.
	
	
	
	
	\bibitem{LWZ} W. Li, H. Wang, and Y Zhai, $L^{p}$-improving bounds and weighted estimates for maximal functions associated with curvature, {\it J. Fourier Anal. Appl.,} \textbf{29} (2023).
	
	
	\bibitem{LW24}  W. Li and H. Wang, Maximal functions related to homogeneous hypersurfaces in $\mathbb{R}^3$,  {\it arXiv: 2406.06876v2}, (2024).
	
	
	
	
	\bibitem{NSW93}  A. Nagel, A. Seeger, and S. Wainger, Averages over convex hypersurfaces,  {\it Amer. J. Math.}, {\bf115} (1993), 903--927.
	
	\bibitem{MSS93}  G. Mockenhaupt, A. Seeger, and C. D. Sogge, Local smoothing of Fourier integral operators and Carleson--Sj\"{o}lin estimates, {\it J. Amer. Math.}, {\bf6} (1993), 65--130.
	

\bibitem{PS97} D. H. Phong and E. M. Stein,
The Newton polyhedron and oscillatory integral operators,
{ \it Acta Math.}, \textbf{179(1)} (1997), 105--152.

\bibitem{Schwend} J. Schwend, Near optimal $L^{p} \rightarrow L^{q}$ estimates for euclidean averages
	over prototypical hypersurfaces in $\mathbb{R}^{3}$, \textit{Math. Ann.}, \textbf{390} (2024), 1309--1364.


\bibitem{See}
A. Seeger,   Some inequalities for singular convolution operators in $L^p$-spaces, {\it Trans. Amer. Math. Soc. } \textbf{308(1)} (1988), 259--272.

	
	\bibitem{SS85}  C. D. Sogge and E. M. Stein, Averages of functions over hypersurfaces in $\mathbb{R}^n$, {\it Invent. Math.}, {\bf82} (1985), 543--556.
	
	\bibitem{St76}  E. M. Stein, Maximal functions, I: Spherical means, {\it Proc. Nat. Acad. Sci. U.S.A.}, {\bf73} (1976), 2174--2175.
	

	\bibitem{Stbook} E. M. Stein,
	\emph{Hamonic Analysis: real-variable methods, orthogonality and
		Oscillatory integrals}, Princeton Mathematical Series, 43, Monographs in Harmonic Analysis, Princeton University Press, Princeton, NJ, 1993.
	
		
	\bibitem{Zi14}  E. Zimmermann, On $L^p$-estimates for maximal average over hypersurfaces not satisfying the transversality condition, \textit{Phd thesis, Christian-Albrechts Universit\"{a}t Bibliothek Kiel}, 2014.
		\color{black}
	
	
	
	
\end{thebibliography}
\end{document}